%% file: Observer_Parabolic_nonl-arx-v2.tex
\documentclass[a4paper,reqno]{amsart}

\usepackage{amscd,amsthm,amsmath,amssymb,mathtools}
\usepackage{verbatim}
\usepackage{color}
\usepackage{hyperref}
\usepackage{url}
\usepackage{enumerate,enumitem}
\usepackage{graphicx}
\usepackage{epstopdf,epsfig,subfigure}
\usepackage{curve2e}
\usepackage{array}

\usepackage[numbers,sort&compress]{natbib}

\theoremstyle{plain}
 \newtheorem*{mainresult}{Main Result}

\newtheorem{theorem}{Theorem}[section]
\newtheorem{proposition}[theorem]{Proposition}
\newtheorem{lemma}[theorem]{Lemma}
\newtheorem{corollary}[theorem]{Corollary}

\newtheorem{assumption}[theorem]{Assumption}
\newtheorem{conjecture}[theorem]{Conjecture}

\theoremstyle{definition}
\newtheorem{definition}[theorem]{Definition}
\newtheorem{remark}[theorem]{Remark}

\numberwithin{equation}{section}

\usepackage{geometry}

\input{Mathcommands}

\usepackage{todonotes}
\newcommand{\todoautc}[3][]{%
    \ifthenelse{\equal{#1}{}}{\todo[size=\scriptsize]{{\bf#2} #3}{}}{\todo[color=#1,size=\scriptsize]{{\bf#2} #3}{}}%
}

\begin{document}
\title{Semiglobal oblique projection exponential dynamical observers for nonautonomous semilinear parabolic-like equations}
\author{S\'ergio S.~Rodrigues$^{\tt1}$}
\thanks{
\vspace{-1em}\newline\noindent
\newline\noindent
$^{\tt1}$ Johann Radon Institute for Computational and Applied Mathematics,
  \"OAW, \newline\indent
  Altenbergerstr. 69, 4040 Linz, Austria.
{\sc Email}:
{\small\tt   sergio.rodrigues@ricam.oeaw.ac.at}
 }

\begin{abstract}
The estimation of the full state of a nonautonomous semilinear parabolic equation is achieved by a Luenberger type dynamical observer.
The estimation is derived from an output given by a finite number of average measurements of the state on small regions.
The state estimate given by the observer converges exponentially to the real state, as time increases. The result is semiglobal
in the sense that the error dynamics can be made stable for an arbitrary given initial condition, provided a large enough number of
measurements, depending on the norm of the initial condition, is taken.
The output injection operator is explicit and involves a suitable oblique projection.
The results of numerical simulations are presented showing the exponential stability of the error dynamics.
\end{abstract}

%

\maketitle

%

\pagestyle{myheadings} \thispagestyle{plain} \markboth{\sc S. S. Rodrigues}
{\sc Semiglobal oblique projection dynamical observer}

\section{Introduction}
We consider evolutionary nonlinear parabolic-like equations, for time~$t\ge0$, as
 \begin{align}\label{sys-y-o} 
 \dot y +Ay+A_{\rm rc}y +\clN(y)=f,\qquad w=\clZ_{S} y,
\end{align} 
evolving in a Hilbert space~$V$. $A$ and~$A_{\rm rc}=A_{\rm rc}(t)$ are, respectively,
a time-independent symmetric linear diffusion-like operator and a
time-dependent linear reaction-convection-like operator. Further~$\clN(y)=\clN(t,y)$ is
a time-dependent nonlinear operator and~$f=f(t)$
is a time-dependent external forcing.
The triple $(A,A_{\rm rc},\clN,f)$, defining the dynamics, is assumed to be known.

The unknown state of the equation is the variable~$y=y(t)\in V$, where~$V$ is a suitable Hilbert space.
The vector output~$w = w (t)=\clZ_{S} y(t)\in\bbR^{S_\sigma\times 1}$
consists of a finite number of measurements, where~$S_\sigma$ is a positive integer. The output
operator~$\clZ_{S}\colon V\to \bbR^{S_\sigma\times 1}$ is linear.

The initial state~$y(0)\in V$, at time~$t=0$, is assumed to be unknown. 
Our task is to estimate the state~$y$ from the output~$w$, which is assumed
to be given in the form of ``averages'' as
\begin{subequations}\label{sens_prop}%
\begin{equation}\label{sens_prop-av}
  w (t)=\begin{bmatrix}
           w_{1} (t) \\  w_{2}(t) \\ \vdots \\ w_{S}(t)
         \end{bmatrix}
,\qquad  w_{i} (t)\coloneqq(\fkw_{i} ,y(t))_{H},\quad 1\le i\le S_\sigma,
\end{equation}
where~$(\Bigcdot,\Bigcdot)_{H}$ is the scalar product in a pivot Hilbert space~$H\supset V$.
Each $\fkw_{i} \in H$ will be referred to as a sensor, and we assume that
\begin{equation}\label{sens_prop-linind}
 \mbox{the family of sensors, }W_{S} \coloneqq\{\fkw_{i} \mid 1\le i\le S_\sigma\}\subset H,
 \mbox{ is linearly independent}.
\end{equation}

We consider the case where we can place the sensors, depending on their number~$S_\sigma$,
so that we will actually have
\begin{equation}\label{sens_prop-loc}
w_{i} =w_{i,S} ,\quad \fkw_{i} (t)=\fkw_{i,S} (t)\subset H,\qquad 1\le i\le S_\sigma,
\end{equation}
\end{subequations}

\begin{remark} For simplicity, we may think of~$S_\sigma = S$. In the application of the
result to concrete examples, it is convenient to have a particular
subsequence~$(S_\sigma )_{S\in \bbN_0}$ of positive integer numbers, as we shall see in
Section~\ref{S:parabolic-OK}, where we shall take~$\sigma(S)=(2S)^d$, for scalar parabolic
equations evolving in
rectangular spatial domains~$\Omega\subset\bbR^d$.
\end{remark}

In real applications, for a fixed instant of time~$t$, it is not possible to recover~$y(t)$
from~$ w (t)$, in general.
However, from the knowledge of
the dynamics of~\eqref{sys-y-o}, it may be possible to construct a
Luenberger type dynamical observer, giving us an estimate~$\widehat y(t)$ of~$y(t)$, so that~$\widehat y(t)$
converges exponentially
to~$y(t)$ as time increases.

Together with the family of sensors we will need also a family of auxiliary functions
\begin{equation}\notag
 \widetilde W_{S} \coloneqq\{\widetilde \fkw_{i}=\widetilde \fkw_{i,S} \mid 1\le i\le S_\sigma\}\subset\rmD(A),
 \mbox{ which is linearly independent},
\end{equation}
where~$\rmD(A)\subset V$ is another Hilbert space, to be precised later on, namely,
as the domain of the diffusion operator~$A$,.
We will also consider the corresponding linear spans
\[
 \clW_{S}\coloneqq\linspan W_{S}\subset H\quad\mbox{and}\quad
 \widetilde \clW_{S}\coloneqq\linspan \widetilde W_{S}\subset\rmD(A).
\]

\begin{remark}
Sometimes, in~\cite{AzouaniOlsonTiti14,OlsonTiti03} this problem of constructing a dynamic state estimate
is referred to as ``continuous data assimilation''. 
\end{remark}

\subsection{The main result}
We shall show that Luenberger observers as
\begin{subequations}\label{sys-haty-o} %
 \begin{align}
 \dot{\widehat y} +A\widehat y+A_{\rm rc}\widehat y+\clN(\widehat y)
 &=f+\fkI_{S}^{[\lambda,\ell]}
 (\clZ\widehat y-w),\quad
 \widehat y(0)=\widehat y_0\in V,\label{sys-haty-o-Dyn}
 \end{align}
 with the output injection operator given by

 \begin{equation}\label{sys-haty-o-Inj}
 \fkI_{S}^{[\lambda,\ell]}\coloneqq-\lambda A^{-1}P_{\clW_{S}}^{\widetilde\clW_{S}^\perp}
 A^\ell P_{\widetilde\clW_{S}}^{\clW_{S}^\perp}\bfZ^{W_s},\qquad \lambda>0,\quad0\le\ell\le2,
 \end{equation}
are able to estimate the state~$y$ of system~\eqref{sys-y-o}, for any given~$\ell\in[0,2]$ and suitable
tuples~$(\lambda,\clW_{S},\widetilde\clW_{S})$. 
Here~$\bfZ^{W_s}\colon\bbR^{S_\sigma}\to\clW_{S}$ is the linear operator defined by
\begin{align}\label{sys-haty-o-bfZ}
\bfZ^{W_S}\bfz 
\coloneqq\sum\limits_{i=1}^{S_\sigma}
\left([\clV_{S}]^{-1} \bfz\right)_i\fkm_{i,S},\qquad \bfz\in\bbR^{S_\sigma},
\end{align}
where~$[\clV_{S}]\in\bbR^{S_\sigma\times S_\sigma}$ is the generalized Vandermonde matrix,
whose entries in the~$i$th row and~$j$th column are
\begin{align}
 [\clV_{S}]_{(i,j)}=(\fkm_{i,S},\fkm_{j,S})_H,
\end{align}
\end{subequations} 
and~$P_F^G$ denotes the oblique projection in~$H$ onto~$F$ along~$G$.

For suitable~$\varrho\ge1$ and $\mu>0$, we will have the inequality
\begin{align}\label{goal} 
 \norm{\widehat y(t)-y(t)}{V}\le \varrho\ex^{-\mu (t-s)}\norm{\widehat y(s)-y(s)}{V},
 \quad\mbox{for all}\quad t\ge s\ge0.
\end{align}

Note that, from~\eqref{sys-y-o} and~\eqref{sys-haty-o}, the error~$z=\widehat y-y$ satisfies%
\begin{subequations}\label{sys-z-o} 
\begin{align}
 &\dot{z} +Az+A_{\rm rc}z +\fkN_{y}(z)=\fkI_{S}^{[\lambda,\ell]}\clZ z,\quad z(0)=z_0\in V.\label{sys-z-o-dyn}
\intertext{where}
&\fkN_{y}(z)\coloneqq\clN(y+z)-\clN(y)=\clN(\widehat y)-\clN(y),
\end{align}
\end{subequations}
and~$z_0=\widehat y_0-y(0)$.
Our goal~\eqref{goal},  reads now
\begin{align}\label{goal-dif} 
 \norm{z(t)}{V}\le \varrho\ex^{-\mu (t-s)}\norm{z(s)}{V},\quad\mbox{for all}\quad t\ge s\ge0.
\end{align}

\begin{remark}
Observe that~$z_0$ in~\eqref{sys-z-o} is unknown for us, because so is~$y(0)$. On the other hand,
the choice of~$\widehat y_0=\widehat y(0)$ is at our disposal,
for example, we can choose~$\widehat y_0$ as an initial guess we might have for~$y(0)$.
\end{remark}

\begin{remark}
We can see that in~\eqref{sys-haty-o-Inj} we have that~$\fkI_{S}^{[\lambda,\ell]}\clZ z
=-\lambda A^{-1}P_{\clW_{S}}^{\widetilde\clW_{S}^\perp}
 A^\ell q$ with~$q\in\widetilde\clW_S\subset\rmD(A)$. Hence, if~$\ell>1$ we may have
 that~$p\coloneqq A^\ell q\in\rmD(A^{1-\ell})\setminus \rmD(A^{0})$.
 Therefore, it is clear that in the case~$p\notin \rmD(A^{0})=H$ we cannot
 see~$P_{\clW_{S}}^{\widetilde\clW_{S}^\perp}p$ as an oblique
 projection in~$H$, but rather as an extension of such oblique projection to~$\rmD(A^{1-\ell})$.
 Such extensions are well defined, as we shall see later in Proposition~\ref{P:extOblPro}.
\end{remark}

Omitting the details at this point, the main result of this paper is as follows.
\begin{mainresult}
 Under general conditions on the tuple~$(\widehat y,A,A_{\rm rc},\clN,y)$ and under particular conditions
 on the tuple~$(\widetilde W_{S},W_{S} )$ it holds the following.
 For any given~$\ell\in[0,2]$, $R>0$, $\varrho>1$, and~$\mu>0$, there are large
 enough~$S\in\bbN_0$ and~$\lambda>0$ such that: for all
 initial error~$z_0$ satisfying~$\norm{z_0}{V}\le R$, it follows that the
 corresponding solution of~\eqref{sys-z-o},
 with the output injection operator as in~\eqref{sys-haty-o-Inj}, satisfies~\eqref{goal-dif}.
 \end{mainresult}

\begin{definition}
 If~\eqref{goal-dif} holds true, we say that the error dynamics is exponential stable
 with rate~$-\mu<0$ and transient bound~$\varrho\ge1$.
\end{definition}

Note that Main Result says that we can stabilize the nonlinear error
dynamics for {\em arbitrary large initial errors}~$\norm{z_0}{V}\le R$, 
with an {\em arbitrary small exponential rate}~$-\mu<0$, and
{\em arbitrary small transient bound}~$\varrho>1$. For that,
we simply have to take a large enough number of suitable sensors~$S_\sigma$ and
a large enough~$\lambda$. In general, the ``optimal'' transient bound~$\varrho=1$
cannot be taken in Main Result. However, later on,
in Section~\ref{S:finalremks} we shall
give classes of systems where we
can indeed take~$\varrho=1$. Such classes include linear and suitable semilinear systems.
The case~$\varrho=1$
is interesting simply because it means that the error norm is {\em strictly} decreasing. 
Observe also that~$\varrho=1$ is the smallest value possible for~$\varrho$ in~\eqref{goal-dif}
(e.g., by taking~$t=s$).

In the particular case our system is linear, $\clN=0$,
then it is not difficult to show that the observer proposed here is a {\em global} observer.
By a global observer we understand that the output injection
operator~$\fkI_{S}^{[\lambda,\ell]}$ can be taken independent of the norm of the initial error.
The observer proposed here is different from the
one proposed in~\cite{Rod20-sicon,Rod-20-CL},
hence this manuscript also contributes with a new result to the linear case.

\subsection{Motivation}\label{sS:motivation}
Observers are demanded in applications, for example, in the implementation of output based stabilizing controls.
Suppose we have a feedback operator~$K(t)$ such that the associated feedback control~$f(t)=K(t)y(t)$
stabilizes system~\eqref{sys-y-o}. See~\cite{Rod20-eect} for such stabilizing feedback control.
In the case where the state~$y$ is modeled by partial differential equations, the state
is infinite-dimensional and it is not realistic to expect that
we will be able to know/measure the entire state~$y(t)$ at each instant of time~$t$.
However, we can expect that, with a good enough
estimate~$\widehat y(t)$ for~$y(t)$, the
approximated control~$f(t)=K(t)\widehat y(t)$ will be able to stabilize~\eqref{sys-y-o}.

We cannot expect
that an infinite-dimensional state~$y(t)$ can be reconstructed from the
finite set~$w (t)=\clZ_{S}y(t)$ at a fixed time~$t$, hence we look for
a {\em dynamical} observer in order to construct an estimate~$\widehat y(t)$ for~$y(t)$, 
which will be improving as time increases.

\subsection{On previous related works in literature}

For partial differential equations, the results in the literature on state estimation concern mainly the
autonomous case. For example, we refer
to~\cite{AliGiriKrsticLagarrBurl16,FengGuo17,JadachowskiMeurerKugi11,
KangFridman19,RamdaniTucsnakValein16,BuchotRaymondTiago15,Fujii80,OrlovPisaPillUsai17,ZhangWu20}. 
Exceptions are~\cite{MeurerKugi09,Meurer13,JadachowskiMeurerKugi13} for one-dimensional
parabolic equations, $d=1$, by using the
nontrivial backstepping and Cole-Hopf transformations. In~\cite{Meurer13} reaction type
Lipschitz nonlinearities are considered, while in~\cite{JadachowskiMeurerKugi13} convection
nonlinearities are also included, where some details are
omitted concerning the stability of the semilinear error dynamics, as also referred by
the authors in~\cite[Sect.~C]{JadachowskiMeurerKugi13}.
See also the auxiliary nonautonomous heat equation in~\cite[see Eq.~(17)]{MeurerKugi09}.

In the investigation of the autonomous case, as in~\cite{RamdaniTucsnakValein16},
the spectral properties of the time-independent operator dynamics play a crucial
role in the derivation of the results. The (un)stability
results in~\cite{Wu74} suggest that such spectral properties in the nonautonomous
case (at each fixed time~$t>0$)
are not an appropriate tool to deal with the nonautonomous case. 
The recent work~\cite{AstrovskiiGaishun19} also shows that, in general, the state
estimation problem in the nonautonomous
case is not an easy task even for the case of finite-dimensional systems.

The approach in~\cite{AzouaniOlsonTiti14} is applicable to state estimation of parabolic-like systems
for which we can derive the existence of a finite set of so-called determining parameters. This includes the {2D}
Navier--Stokes equations, whose weak solutions are well posed and are exponentially stable under
the absence of external forces, i.e., when $f=0$.
The method in~\cite{AzouaniOlsonTiti14} is quite interesting because, depending on the nature of the ``chosen''
determining parameters, it can be applied to several types of measurements,
including average-like measurements as we are particularly
interested in. However, in this manuscript we consider a class of nonlinear equations whose free dynamics evolution is not well posed
in the sense of weak solutions (for initial states given in the pivot Hilbert space~$H$).
We will need strong solutions (for ``more regular'' initial states given in the Hilbert space~$V\subset H$),
but even for such solutions
the free dynamics evolution will be well posed only for short time, that is, in general the free dynamics has
strong solutions which blow up in finite time. Hence, to deal with state estimation for such class of systems,
the method in~\cite{AzouaniOlsonTiti14}
is (or, seems to be) not appropriate.

In~\cite{Rod20-sicon} a global observer was presented to estimate
the state of linear parabolic equations, where the placement of the actuators play an important role.
The results in this manuscript are also derived under the assumption that we are allowed
to suitably place the sensors.
Such assumption seems to be natural and to reflect common sense: it matters (or, may matter)
where we take our measurements in. Again, the observer in~\cite{Rod20-sicon}  provides
an estimate for the weak solution and the exponential convergence is derived in the pivot
$H$~norm. As we said above, weak solutions do not necessarily exist for the
class of nonlinear systems we consider, this is one reason we will (need to) use a different output
injection operator in this manuscript, to deal with strong solutions and derive the exponential convergence in the
stronger $V$~norm.

Finally, we must say that some of the above mentioned works, as~\cite{KangFridman19,BuchotRaymondTiago15},
do not consider the observer design problem alone,
but (already) coupled with a stabilization problem (output based feedback control). Also, some
of the above works deal with boundary measurements, while here we deal with internal measurements.

\subsection{Illustrating example. Scalar parabolic equations}\label{sS:illust-parabolic}

The results will follow under general assumptions on the plant dynamics
operators, on the external force, and on the targeted real state. We shall need also a particular assumption
involving the set of sensors. Such assumptions will be presented later on and will be satisfied,
in particular, for a general class of
semilinear parabolic equations, under either Dirichlet or Neumann boundary conditions, including
\begin{subequations}\label{sys-y-o-parab} 
\begin{align} 
 &\tfrac{\p}{\p t} y +(-\Delta+\Id) y+ay +b\cdot\nabla y
 +\widetilde a\norm{y}{\bbR}^{r-1}y +(\widetilde b\cdot\nabla y)\norm{y}{\bbR}^{s-1}y=f,\\
 &\clG y\rest{\Gamma}=g,\qquad w =\clZ_{S} y,
 \end{align}
 with~$r\in(1,5)$ and~$s\in[1,2)$, 
\end{subequations}
defined in a bounded connected open spatial subset~$\Omega\in\bbR^d$, $d\in\{1,2,3\}$, with
boundary~$\Gamma=\p\Omega$. $\Omega$ is assumed to be either
smooth or a convex polygon. The state
is a function $y=y(x,t)$, defined for~$(x,t)\in\Omega\times(0,+\infty)$. The operator~$\clG$
imposes the boundary conditions,
\begin{align}
  \clG &=\Id,\quad\mbox{for Dirichlet boundary conditions},\notag\\
  \clG &=\bfn\cdot\nabla=\tfrac{\p}{\p\bfn},\quad\mbox{for Neumann boundary conditions,}\notag
\end{align}
where~$\bfn=\bfn(\bar x)$ stands for the outward unit normal vector to~$\Gamma$, at~$\bar x\in\Gamma$.

The functions~$a=a(x,t)$, $b=b(x,t)$, $\widetilde a=\widetilde a(x,t)$, $\widetilde b=\widetilde b(x,t)$,
and~$f=f(x,t)$
are defined in~$\Omega\times(0,+\infty)$, and the function~$g=g(\bar x,t)$ is defined
for~$(\bar x,t)\in\Gamma\times(0,+\infty)$. Thus
the data tuple~$(a,b,\widetilde a,\widetilde b,f,g)$ is allowed to depend on both space and time variables.
We assume that,
\begin{subequations}\label{assum.abf.parab}
\begin{align}
 &\mbox{$a$ and $\widetilde a$ are in~$L^\infty(\Omega\times(0,+\infty))$},\label{assum.abf.parab-a}\\
 &\mbox{$b$ and~$\widetilde b$ are in~$L^\infty(\Omega\times(0,+\infty))^d$},\label{assum.abf.parab-b}\\
 &\mbox{There exists }\tau_y>0\mbox{ such that }
 \sup_{t\ge0}\left(\norm{y(t)}{H^1(\Omega)}+\norm{y}{L^2((t,t+\tau_y),H^2(\Omega))}\right)<+\infty.
 \label{assum.abf.parab-y}
\end{align}

\end{subequations}

\begin{remark}\label{R:assumpt-fgy}
In~\eqref{assum.abf.parab-y} we assume, in particular, that the real state~$y$ must be a globally defined strong
solution~$y\in \clY\coloneqq L^\infty_{\rm loc}((0,+\infty),H^1(\Omega))
\bigcap L^2_{\rm loc}((0,+\infty),H^2(\Omega))$.
In general, for regular enough external force~$f$ (e.g., for~$f=0$)
we will only have the local existence in time:
for a suitable~$\tau_*>0$, $y\in L^\infty((0,\tau),H^1(\Omega))
\bigcap L^2((0,\tau),H^2(\Omega))$, for~$\tau<\tau_*$. There are, however,
cases where~\eqref{assum.abf.parab-y} will hold true, for example, for the case where
$g=0$ and~$f=Ky$ is a stabilizing feedback control. See Section~\ref{sS:motivation}.
Another example is the case
of time-periodic systems having time-periodic solutions. A third example are Lyapunov
stable (not necessarily asymptotic stable) systems.
\end{remark}

As output we take the averages of the solution in subdomains~$\omega_{i}=\omega_{i,S}\subset\Omega$, as
\begin{equation}\label{output-parab}
  w_{i} (t)=(\indf_{\omega_{i}},y(\Bigcdot,t))_{L^2(\Omega)}=\int_{\omega_{i}}y(x,t)\,\ed x,
  \qquad 1\le i\le S_\sigma.
\end{equation}
We will be interested in the case the regions~$\omega_{i}$, where we take the measurements in,
are constrained to cover an a priori fixed volume, namely,
$\vol(\bigcup_{i=1}^{S_\sigma}\omega_{i,S})\le r\vol(\Omega)$ with~$0<r<1$ independent of~$S$.
In other words, we allow ourselves to
take/place as many sensors as we want/need, but we are allowed to perform
measurements only in (at most) a fixed percentage of the spatial
domain~$\Omega$, namely,~$100r\%$.

\begin{remark}
The usual average over~$\omega_{i}$ 
is~$\breve{w}_{i}\coloneqq\frac{\int_{\omega_{i}}y(x,t)\,\ed x}{\int_{\omega_{i}}\,\ed x}$. However, 
we assume that we know our sensors, that is, we know the regions~$\omega_{i}$ where we take the measurements
in. In this case, knowing/measuring~$\breve{w}_{i}$ is equivalent to knowing/measuring~${w}_{i}$. 
\end{remark}

In order to apply our results to system~\eqref{sys-y-o-parab}, we have just to
rewrite~\eqref{sys-y-o-parab} as an evolutionary equation~\eqref{sys-y-o}.
To this purpose, 
we define for both Dirichlet, respectively
 Neumann, boundary conditions the spaces
\begin{align}
\rmD(A)=H^2_\clG(\Omega)&\coloneqq\{h\in  H^2(\Omega)\mid \clG h\rest\Gamma=0\},
\quad\mbox{for}\quad\clG\in\{\Id,\tfrac{\p}{\p\bfn}\},\notag
\intertext{and}
V=H^1_\Id(\Omega)&\coloneqq H^1_0(\Omega)=\{h\in  H^1(\Omega)\mid h\rest\Gamma=0\},\qquad
H^1_{\frac{\p}{\p\bfn}}(\Omega)\coloneqq H^1(\Omega),\notag
\end{align}
with the operators
\begin{align}
 A\coloneqq-\nu\Delta+\Id,\quad A_{\rm rc}\coloneqq a\Id  +b\cdot\nabla,\quad\mbox{and}\quad
 \clN(t,y)\coloneqq\widetilde a(t)\norm{y}{\bbR}^{r-1} y
 +(\widetilde b(t)\cdot\nabla y)\norm{ y}{\bbR}^{s-1} y.\notag
\end{align}
Then, we just
construct the Luenberger observer as in~\eqref{sys-haty-o} and apply the Main Result.

\subsection{Contents and notation}
In Section~\ref{S:assumptions} we present the assumptions we require for the
dynamics plant operators and for all the ``parameters'' involved in
the output injection operator. 
In Section~\ref{S:exp_error} we prove that under such assumptions the error
of the observer estimate decreases
exponentially to zero. In Section~\ref{S:parabolic-OK}
we show that the required assumptions are satisfiable for standard parabolic equations evolving
in rectangular domains. In Section~\ref{S:simul}
we present the results of numerical simulations showing the exponential stability of the error
dynamics, for a rectangular domain, namely the unit square.
In Section~\ref{S:finalremks}
we comment on the derived results.
Finally, the Appendix gathers the proofs of auxiliary results
needed to derive the main result.

\smallskip

Concerning the notation, we write~$\bbR$ and~$\bbN$ for the sets of real numbers and nonnegative
integers, respectively, and we set $\bbR_r\coloneqq(r,+\infty)$, $r\in\bbR$,
and~$\bbN_0\coloneqq\bbN\setminus\{0\}$.

Given two Banach spaces~$X$ and~$Y$, if the inclusion
$X\subseteq Y$ is continuous, we write $X\xhookrightarrow{} Y$. We write
$X\xhookrightarrow{\rm d} Y$, respectively $X\xhookrightarrow{\rm c} Y$,
if the inclusion is also dense, respectively compact.

Let $X\subseteq Z$ and~$Y\subseteq Z$ be continuous inclusions, where~$Z$
is a Hausdorff topological space.
Then we can define the Banach spaces $X\times Y$, $X\cap Y$, and $X+Y$,
endowed with the  norms 
$|(h,g)|_{X\times Y}:=\bigl(|h|_{X}^2+|g|_{Y}^2\bigr)^{\frac{1}{2}}$,
$|\hat h|_{X\cap Y}:=|(\hat h,\hat h)|_{X\times Y}$, and
$|\tilde h|_{X+Y}:=\inf\limits_{(h,g)\in X\times Y}\bigl\{|(h,g)|_{X\times Y}\mid \tilde h=h+g\bigr\}$,
respectively.
In case we know that $X\cap Y=\{0\}$, we say that $X+Y$ is a direct sum and we write $X\oplus Y$ instead.

The space of continuous linear mappings from~$X$ into~$Y$ is denoted by~$\clL(X,Y)$. In case~$X=Y$ we 
write~$\clL(X)\coloneqq\clL(X,X)$.
The continuous dual of~$X$ is denoted~$X'\coloneqq\clL(X,\bbR)$.
The adjoint of an operator $L\in\clL(X,Y)$ will be denoted $L^*\in\clL(Y',X')$.

The space of continuous functions from~$X$ into~$Y$ is denoted by~$\clC(X,Y)$.
The space of real valued
increasing functions, defined in~$\overline{\bbR_0}$ and vanishing at~$0$ is denoted by:
\begin{equation}\notag
 \clC_{0,\iota}(\overline{\bbR_0},\bbR) \coloneqq \{\fki\in \clC(\overline{\bbR_0},\bbR)\mid
 \; \mathfrak i(0)=0,\quad\!\!\mbox{and}\quad\!\!
 \mathfrak i(\varkappa_2)\ge\mathfrak i(\varkappa_1)\;\mbox{ if }\; \varkappa_2\ge \varkappa_1\ge0\}.
\end{equation}

We also denote the vector subspace~$\clC_{\rm b, \iota}(X, Y)\subset \clC(X,Y)$ by  
\begin{equation}\notag
 \clC_{\rm b, \iota}(X, Y)\coloneqq
 \left\{f\in \clC(X,Y) \mid \exists\mathfrak i\in \clC_{0,\iota}(\overline{\bbR_0},\bbR)\;\forall x\in X:\;
\norm{f(x)}{Y}\le \mathfrak i (\norm{x}{X})
\right\}.
\end{equation}

The orthogonal complement to a given subset~$B\subset H$ of a Hilbert space~$H$,
with scalar product~$(\Bigcdot,\Bigcdot)_H$,  is 
denoted~$B^\perp\coloneqq\{h\in H\mid (h,s)_H=0\mbox{ for all }s\in B\}$.

Given two closed subspaces~$F\subseteq H$ and~$G\subseteq H$ of the
Hilbert space~$H=F\oplus G$, we denote by~$P_F^G\in\clL(H,F)$
the oblique projection in~$H$ onto~$F$ along~$G$. That is, writing $h\in H$ as $h=h_F+h_G$
with~$(h_F,h_G)\in F\times G$, we have~$P_F^Gh\coloneqq h_F$.
The orthogonal projection in~$H$ onto~$F$ is denoted by~$P_F\in\clL(H,F)$.
Notice that~$P_F= P_F^{F^\perp}$.

Given a sequence~$(a_j)_{j\in\{1,2,\dots,n\}}$ of real nonnegative constants, $n\in\bbN_0$, $a_i\ge0$, we
denote~$\|a\|\coloneqq\max\limits_{1\le j\le n} a_j$.

By
$\overline C_{\left[a_1,\dots,a_n\right]}$ we denote a nonnegative function that
increases in each of its nonnegative arguments~$a_i$, $1\le i\le n$.

Finally, $C,\,C_i$, $i=0,\,1,\,\dots$, stand for unessential positive constants.

\section{Assumptions}\label{S:assumptions}

The results will follow under general assumptions on the plant dynamics
operators~$A$, $A_{\rm rc}$, $\clN$, and on our targeted real state~$y$. We will also need a particular assumption
on the triple~$(\clW_{S}, \widetilde\clW_{S},\lambda)$.

The Hilbert space~$H$, in which system~\eqref{sys-z-o} is evolving in,
will be set as a pivot space, that is, we identify,~$H'=H$.
Let~$V$ be another Hilbert space
with~$V\subset H$.
\begin{assumption}\label{A:A0sp}
 $A\in\clL(V,V')$ is symmetric and $(y,z)\mapsto\langle Ay,z\rangle_{V',V}$ is a complete scalar product in~$V.$
\end{assumption}

From now on, we suppose that~$V$ is endowed with the scalar product~$(y,z)_V\coloneqq\langle Ay,z\rangle_{V',V}$,
which still makes~$V$ a Hilbert space.
Necessarily, $A\colon V\to V'$ is an isometry.
\begin{assumption}\label{A:A0cdc}
The inclusion $V\subseteq H$ is dense, continuous, and compact.
\end{assumption}

Necessarily, we have that
\[
 \langle y,z\rangle_{V',V}=(y,z)_{H},\quad\mbox{for all }(y,z)\in H\times V,
\]
and also that the operator $A$ is densely defined in~$H$, with domain $\rmD(A)$ satisfying
\[
\rmD(A)\xhookrightarrow{\rm d,\,c} V\xhookrightarrow{\rm d,\,c} H\xhookrightarrow{\rm d,\,c} V'
\xhookrightarrow{\rm d,\,c}\rmD(A)'.
\]
Further,~$A$ has a compact inverse~$A^{-1}\colon H\to \rmD(A)$,
and we can find a nondecreasing
system of (repeated accordingly to their multiplicity) eigenvalues
$(\alpha_n)_{n\in\bbN_0}$ and a corresponding complete basis of
eigenfunctions $(e_n)_{n\in\bbN_0}$:
\begin{equation}\notag
0<\alpha_1\le\alpha_2\le\dots\le\alpha_n\le\alpha_{n+1}\to+\infty \quad\mbox{and}\quad Ae_n=\alpha_n e_n.
\end{equation}

We can define, for every $\xi\in\bbR$, the fractional powers~$A^\xi$, of $A$, by
\[
 y=\sum_{n=1}^{+\infty}y_ne_n,\quad A^\xi y=A^\xi \sum_{n=1}^{+\infty}y_ne_n
 \coloneqq\sum_{n=1}^{+\infty}\alpha_n^\xi y_n e_n,
\]
and the corresponding domains~$\rmD(A^{|\xi|})\coloneqq\{y\in H\mid A^{|\xi|} y\in H\}$, and
$\rmD(A^{-|\xi|})\coloneqq \rmD(A^{|\xi|})'$.
We have that~$\rmD(A^{\xi})\xhookrightarrow{\rm d,\,c}\rmD(A^{\zeta_1})$, for all $\xi>\xi_1$,
and we can see that~$\rmD(A^{0})=H$, $\rmD(A^{1})=\rmD(A)$, $\rmD(A^{\frac{1}{2}})=V$.

For the time-dependent operator and external forcing we assume the following:

\begin{assumption}\label{A:A1}
For almost every~$t>0$ we have~$A_{\rm rc}(t)\in\clL(V, H)$,
and we have a uniform bound as $\norm{A_{\rm rc}}{L^\infty(\bbR_0,\clL(V, H))}\eqqcolon C_{\rm rc}<+\infty.$
\end{assumption}

\begin{assumption}\label{A:N}
 We have~$\clN(t,\Bigcdot)\in\clC_{\rm b,\iota}(\rmD(A),H)$ 
and
there exist constants $C_\clN\ge 0$, $n\in\bbN_0$, 
$\zeta_{1j}\ge0$, $\zeta_{2j}\ge0$,
 $\delta_{1j}\ge 0$, ~$\delta_{2j}\ge 0$, 
 with~$j\in\{1,2,\dots,n\}$, such that
 for  all~$t>0$ and
 all~$(y_1,y_2)\in \rmD(A)\times \rmD(A)$, we have
\begin{align}
&\norm{\clN(t,y_1)-\clN(t,y_2)}{H}\le C_\clN\textstyle\sum\limits_{j=1}^{n}
  \left( \norm{y_1}{V}^{\zeta_{1j}}\norm{y_1}{\rmD(A)}^{\zeta_{2j}}+\norm{y_2}{V}^{\zeta_{1j}}
  \norm{y_2}{\rmD(A)}^{\zeta_{2j}}\right)
   \norm{d}{V}^{\delta_{1j}}\norm{d}{\rmD(A)}^{\delta_{2j}},\notag
\end{align}
with~$d\coloneqq y_1-y_2$, $\zeta_{2j}+\delta_{2j}<1$ and~$\delta_{1j}+\delta_{2j}\ge1$.
\end{assumption}

\begin{assumption}\label{A:realy}
The targeted real state~$y$, satisfying~\eqref{sys-y-o}, satisfies the uniform persistent
boundedness estimate as follows. There are constants~$C_y\ge0$ and~$\tau_y>0$ such that
\begin{align}
\sup_{s\ge0}\norm{y(s)}{V}\le C_y\quad\mbox{and}\quad\sup_{s\ge0}\norm{y}{L^2((s,s+\tau_y),\rmD(A))}<C_y.\notag
\end{align}
\end{assumption}

\begin{assumption}\label{A:DS}
The pair~$(\sigma,(\widetilde W_{S}, W_{S})_{S\in\bbN_0})$ satisfies:
\[
 \sigma\colon\bbN_0\to\bbN_0\mbox{ is strictly increasing}
\]
and, with $S_\sigma\coloneqq\sigma(S)$, $\widetilde\clW_{S}=\linspan\widetilde W_{S}$,
and~$\clW_{S}=\linspan W_{S}$,
 \begin{align}
&W_{S}\coloneqq\{\fkw_j\mid  1\le j\le S_\sigma\}\subset H,\notag\\
&\widetilde W_{S}\coloneqq\{\widetilde \fkw_j\mid  1\le j\le S_\sigma\}\subset\rmD(A)\subset H,\notag\\
&\dim\clW_{S}=S_\sigma=\dim\widetilde\clW_{S}\mbox{ and }H=\clW_{S}\oplus\widetilde \clW_{S}^\perp.\notag
\end{align}
\end{assumption}

The key assumption concerns the following Poincar\'e-like constant 
\begin{align}\label{Poinc_const}
\beta_{S_{\sigma+}}\coloneqq\inf_{Q\in(\rmD(A)\bigcap\clW_{S}^\perp)\setminus\{0\}}
\tfrac{\norm{Q}{\rmD(A)}^2}{\norm{Q}{V}^2}.
\end{align}
\begin{assumption}\label{A:Poincare}
The sequence~$(\beta_{S_{\sigma+}})_{S\in\bbN_0}$ in~\eqref{Poinc_const}
is divergent, $\lim\limits_{S\to+\infty}\beta_{S_{\sigma+}}=+\infty$.
\end{assumption}

The last assumption concerns the type of outputs.

\begin{assumption}\label{A:output}
The output~$w=\clZ    y\in\bbR^{S_\sigma}$ is of the form
~$w_i(t)=(\fkw_i,y(t))_H$, with~$\fkw_i\in W_S$. 
\end{assumption}

Assumptions~\ref{A:A0sp}--\ref{A:DS} are satisfiable for parabolic systems as~\eqref{sys-y-o-parab}.
Assumptions~\ref{A:A0sp}--\ref{A:A1} are usually not hard to check for such systems.
Assumption~\ref{A:N} is satisfied by a general class of polynomial nonlinearities as in~\eqref{sys-y-o-parab}.
Assumption~\ref{A:realy} is a requirement on our targeted state, which simply says that the real state to be estimated
is a strong solution which is bounded in a general appropriate way.
It is also not difficult to construct spaces satisfying
Assumption~\ref{A:DS}, and then in Assumption~\ref{A:output} we are simply requiring the form of the output. 

The satisfiability of Assumption~\ref{A:Poincare} is nontrivial.
We shall prove in Section~\ref{S:parabolic-OK} that it is 
satisfied for scalar parabolic equations evolving in rectangular
spatial domains~$\Omega\subset\bbR^d$, for suitable placement of the 
sensors (as indicator functions). The proof can be adapted to general convex polynomial domains.
The satisfiability of the Assumption~\ref{A:Poincare} for general smooth domains is an open question.
See the discussion in~\cite[Sect.~7.3]{Rod20-sicon}.

\begin{remark}
Note that Assumption~\ref{A:A1} is stronger than the one taken in~\cite[Assum.~2.3]{Rod20-sicon}
in the linear setting.
We need extra regularity for~$A_{\rm rc}$ because weak solutions,
as considered in~\cite{Rod20-sicon}, living in~$W_{\rm loc}(\bbR_0,V,V')$,
are not regular
enough to deal with the entire class of nonlinear systems we shall consider here.
We need strong solutions, living in~$W_{\rm loc}(\bbR_0,\rmD(A),H)$,
to guarantee the existence and uniqueness of
solutions for all systems involved in our analysis.
\end{remark}

\section{Exponential stability of the error dynamics}\label{S:exp_error}

For given~$S\in\bbN_0$ and~$\ell\in\bbR$, we define another Poicar\'e-like constant as follows
\begin{equation}\label{und-alpha}
0<\underline\alpha_{S,\ell}\coloneqq\inf\limits_{q\in\widetilde\clW_S\setminus\{0\}}
\tfrac{\norm{q}{\rmD\bigl(A^\frac{\ell}2\bigr)}^2}{\norm{q}{\rmD(A)}^2}.
\end{equation}

We prove the following more general abstract version of the main Main Result. 
 \begin{theorem}\label{T:main}
 Let Assumptions~\ref{A:A0sp}--\ref{A:DS} hold true and let us be
 given~$\ell\in[0,2]$, $R>0$, $\varrho>1$, and~$\mu>0$. Then there exists a
 pair~$(S^*,\lambda^*)\in\bbN_0\times\bbR_0$ such that:
 for all pairs~$(S,\lambda)$ satisfying~$S\ge S^*$
 and~$\lambda\ge\lambda^*(S)$,
 the error dynamical system
\begin{subequations}\label{sys-error}
\begin{align} 
 &\dot{z} +Az+A_{\rm rc}(t)z +\fkN_y(t,z)=\widetilde \fkI_{S}^{[\lambda,\ell]}z,\quad z(0)=z_0\in V,\qquad t\ge 0
  ,\label{sys-error-dynic}
  \intertext{where}
  &\fkN_y(t,z)\coloneqq\clN(t,z+y)-\clN(t,y)
  \qquad\mbox{and}\qquad\widetilde \fkI_{S}^{[\lambda,\ell]}z
  = -\lambda A^{-1}P_{\clW_{S}}^{\widetilde\clW_{S}^\perp}A^\ell P_{\widetilde\clW_{S}}^{\clW_{S}^\perp}z,
\end{align}
is exponentially stable with rate~$-\mu$ and transient bound~$\varrho$.
For all~$z_0\in V$ the solution of~\eqref{sys-error} satisfies
\begin{align}\label{goal-dif-error} 
 \norm{z(t)}{V}\le \varrho\ex^{-\mu (t-s)}\norm{z(s)}{V},\quad\mbox{for all}\quad t\ge s\ge0.
\end{align}
Furthermore, the constants~$S^*$ and~$\lambda^*(S)$ can be taken of the form
\begin{align}
 S^*&=\ovlineC{R,\mu,\varrho,\frac1{\varrho^\frac12-1},\tfrac1{\tau_y},\tau_y,
 \frac{2\dnorm{\zeta_{1}}{}}{1-\dnorm{\delta_{2}+\zeta_{2}}{}},
 \frac{2+\dnorm{\frac{2\zeta_{2}}{1-\delta_{2}}}{}}{2-\dnorm{\frac{2\zeta_{2}}{1-\delta_{2}}}{}}
 ,\frac{2\dnorm{\delta_{1}+\delta_{2}+\zeta_{1}+\zeta_{2}}{}-2}{1-\dnorm{\delta_{2}+\zeta_{2}}{}},C_{\rm rc},C_y},
\intertext{and}
 \lambda^*(S)&=\ovlineC{\frac{1}{\underline\alpha_{S,\ell}},R,\mu,\varrho,\frac1{\varrho^\frac12-1},
 \tfrac1{\tau_y},\tau_y,
 \frac{2\dnorm{\zeta_{1}}{}}{1-\dnorm{\delta_{2}+\zeta_{2}}{}},
 \frac{2+\dnorm{\frac{2\zeta_{2}}{1-\delta_{2}}}{}}{2-\dnorm{\frac{2\zeta_{2}}{1-\delta_{2}}}{}}
 ,\frac{2\dnorm{\delta_{1}+\delta_{2}+\zeta_{1}+\zeta_{2}}{}-2}{1-\dnorm{\delta_{2}+\zeta_{2}}{}},C_{\rm rc},C_y},
 \label{formlam*}
\end{align}
where~$(C_{\rm rc},C_y,\tau_y,\delta,\zeta)$ is the data in Assumptions~\ref{A:A1}--\ref{A:realy}.
\end{subequations}
\end{theorem}

\begin{remark}
 Recall that~$\|a\|\coloneqq\max\limits_{1\le j\le n} a_j$, for example,
$\dnorm{\frac{2\zeta_{2}}{1-\delta_{2}}}{}=\max\limits_{1\le j\le n}\frac{2\zeta_{2j}}{1-\delta_{2j}}$.
\end{remark}

\begin{remark}\label{R:underAlpha}
 Observe that from~\eqref{formlam*}, if we can show that for a given~$\ell\in[0,2]$ we have
 that~$\underline\alpha_{S,\ell}\ge\underline\alpha>0$ with~$\underline\alpha$ independent of~$S$, then
 we can conclude that the lower bound $\lambda^*(S)$ can be taken independent of~$S$.
 This is always the case for~$\ell=2$ because~$\underline\alpha_{S,2}=1$.
 For~$0\le\ell<2$ the existence of such~$\underline\alpha>0$ is not clear and will/may depend on~$\widetilde\clW_S$.
We will come back to this point in Section~\ref{S:parabolic-OK}; see Proposition~\ref{P:underAlpha},  where
we give an example where such strictly positive lower bound $\underline\alpha$ does not exist for~$\ell\in\{0,1\}$. 
\end{remark}

Note that~\eqref{sys-z-o-dyn} is equivalent to~\eqref{sys-error-dynic}. Indeed,
denoting by~$P_{\clW_{S}}=P_{\clW_{S}}^{\clW_{S}^\perp}$
the orthogonal projection in~$H$ onto~$\clW_{S}$, from~\cite[sect.~2]{Rod20-sicon}, we know that 
  \begin{equation}\label{bfZZ-orthPW}
 \bfZ^{W_s}\clZ= P_{\clW_{S}},
 \end{equation}
which gives us~$\widetilde \fkI_{S}^{[\lambda,\ell]}z=\widetilde \fkI_{S}^{[\lambda,\ell]}P_{\clW_S} z
=\fkI_{S}^{[\lambda,\ell]}\clZ z=\fkI_{S}^{[\lambda,\ell]}(\clZ\widehat y-w)$.

 \subsection{Auxiliary results} In the proof of Theorem~\ref{T:main}, given in
 Section~\ref{sS:proofT:main}, we will use some auxiliary results, which are gathered in this section. 
 
 We start with results on appropriate estimates for the nonlinear term.
 \begin{lemma}\label{L:NN}
 Let Assumptions~\ref{A:A0sp}, \ref{A:A0cdc}, and~\ref{A:N} hold true, and let~$P\in\clL(H)$. Then 
there is a constant $\overline C_{\clN1}>0$
such that:
  for all~$\widehat\gamma_0>0$, all~$t>0$, 
  and all~$(y_1,y_2)\in \rmD(A)\times \rmD(A)$, 
  we have
\begin{align}
 &2\Bigl( P\left(\clN(t,y_1)-\clN(t,y_2)\right),A(y_1-y_2)\Bigr)_{H}\notag\\
 &\le \widehat\gamma_0 \norm{y_1-y_2}{\rmD(A)}^{2}
  +\!\left(\!1+\widehat\gamma_0^{-\frac{1+\|\delta_2\|}{1-\|\delta_2\|} }\right)
 \!\overline C_{\clN1}\sum\limits_{j=1}^n\norm{y_1-y_2}{V}^{\frac{2\delta_{1j}}{1-\delta_{2j}}}
 \sum\limits_{k=1}^2\norm{y_k}{V}^\frac{2\zeta_{1j}}{1-\delta_{2j}}
 \norm{y_k}{\rmD(A)}^\frac{2\zeta_{2j}}{1-\delta_{2j}}.\notag
\end{align}
  Further, the constant~$\overline C_{\clN1}$
  is of the form
  $ \overline C_{\clN1}=\ovlineC{n,\frac{1}{1-\|\delta_{2}\|},C_\clN,\norm{P}{\clL(H)}}$.
 \end{lemma}
The proof of the lemma is given in~\cite[Sect.~A.1]{Rod20-eect}
for operators as~$P=P_{\widetilde\clW_{S}^\perp}^{\clW_{S}}$, however the steps of such proof can be repeated
for a general operator~$P\in\clL(H)$.
See~\cite[Proposition~3.5]{Rod20-eect}.

Now, we present a sequence of auxiliary results as the following propositions.
The corresponding proofs are presented later
in the Appendix.

An estimate for $\fkN_y(t,\widehat y-y)=\clN(t,\widehat y)-\clN(t, y)$ is as follows.
\begin{proposition}\label{P:fkN}
Let Assumptions~\ref{A:A0sp}, \ref{A:A0cdc}, \ref{A:N}, and~\ref{A:realy} hold true. Then 
there are constants $\widetilde C_{\fkN1}>0$, and~$\widetilde C_{\fkN2}>0$
such that:
  for all~$\widehat\gamma_0>0$, all~$t>0$, 
  all~$(z_1,z_2)\in \rmD(A)\times \rmD(A)$, 
  we have
\begin{align}
 &2\Bigl( \fkN_y(t,z_1)-\fkN_y(t,z_2),A(z_1-z_2)\Bigr)_{H}
 \le \widehat\gamma_0 \norm{z_1-z_2}{\rmD(A)}^{2}\label{fkNyAy}\\
 &\hspace*{1em}
   +\!\left(\!1+\widehat\gamma_0^{-\frac{1+\|\delta_2\|}{1-\|\delta_2\|} }\right)
 \!\widetilde C_{\fkN1}\sum\limits_{j=1}^n\norm{z_1-z_2}{V}^{\frac{2\delta_{1j}}{1-\delta_{2j}}}
 \sum\limits_{k=1}^2\norm{y+z_k}{V}^\frac{2\zeta_{1j}}{1-\delta_{2j}}
 \norm{y+z_k}{\rmD(A)}^\frac{2\zeta_{2j}}{1-\delta_{2j}}.\notag\\
&2\Bigl( \fkN_y(t,z_1),Az_1\Bigr)_{H}
 \le \widehat\gamma_0 \norm{z_1}{\rmD(A)}^{2}\label{fkNyAy0}\\
 &\hspace*{.0em}  +\widetilde C_{\fkN2}\left(1+\widehat\gamma_0^{-\chi_5 }\right)\!
\left(1+\widehat\gamma_0^{-\frac{(\chi_5+1)\chi_2\chi_4}2  }\right)\!
   \left(1+\norm{y}{V}^{\chi_1} \right)\!
  \left(1+\norm{y}{\rmD(A)}^{\chi_2}\right)\!
  \left(1+\norm{z_1}{V}^{\chi_3}\right)\!
  \norm{z_1}{V}^2,\notag
  \end{align}
 with~$\widetilde C_{\fkN2}
  =\ovlineC{n,\widetilde C_{\clN1},\dnorm{\zeta_{1}}{},
  \dnorm{\zeta_{2}}{},\frac{1}{1-\dnorm{\delta_{2}}{}},\frac{1}{1-\dnorm{\zeta_{2}+\delta_{2}}{}}}$ and
 \begin{align}
 &\chi_1\coloneqq \tfrac{2\dnorm{\zeta_{1}}{}}{1-\dnorm{\delta_{2}+\zeta_{2}}{}}\ge0,
  \quad &\chi_2&\coloneqq\dnorm{\tfrac{2\zeta_{2}}{1-\delta_{2}}}{}\in[0,2),&&\label{chi12}\\
  &\chi_3\coloneqq
  \tfrac{2\dnorm{\delta_{1}+\delta_{2}+\zeta_{1}+\zeta_{2}}{}-2}{1-\dnorm{\delta_{2}+\zeta_{2}}{}}\ge0,
  \quad&\chi_4&\coloneqq\tfrac{1}{1-\dnorm{\delta_2+\zeta_2}{}  }>1,
  \quad&\chi_5&\coloneqq\tfrac{1+\|\delta_2\|}{1-\|\delta_2\|}>1.\label{chi345}
 \end{align}
  \end{proposition}

The next auxiliary results concern properties of oblique projections. Recall that~$\clW_{S}\subset H=\rmD(A^0)$
and~$\widetilde\clW_{S}\subset\rmD(A)=\rmD(A^1)$, due to Assumption~\ref{A:DS}.
\begin{proposition}\label{P:restOblPro}
Let~$\xi\in[0,1]$. The restriction of the oblique
projection~$P_{\widetilde\clW_{S}}^{\clW_{S}^\perp}\in\clL(H)$ to~$\rmD(A^\xi)\subseteq H$
 is the oblique projection in~$\rmD(A^\xi)$ onto~$\widetilde\clW_{S}$ along~$\clW_{S}^\perp\bigcap \rmD(A^\xi)$.
 That is,
 $P_{\widetilde\clW_{S}}^{\clW_{S}^\perp}\Bigr|_{\rmD(A^\xi)}=P_{\widetilde\clW_{S}}^{\clW_{S}^\perp\cap
 \rmD(A^\xi)}\in\clL(\rmD(A^\xi))$.
\end{proposition}

 For given~$\xi\in[0,1]$, let us define the
 mapping~$P_{\clW_{S}}^{\widetilde\clW_{S}^\perp}\Bigr|^{\rmD(A^{-\xi})}\colon \rmD(A^{-
 \xi})\to\rmD(A^{-\xi})$ by
\begin{equation}\label{extOblPro}
\left\langle P_{\clW_{S}}^{\widetilde\clW_{S}^\perp}\Bigr|^{\rmD(A^{-\xi})}
z,w\right\rangle_{\rmD(A^{-\xi}),\rmD(A^{\xi})}\coloneqq
 \left\langle z,\clP_{\widetilde\clW_{S}}^{\clW_{S}^\perp}w\right\rangle_{\rmD(A^{-\xi}),\rmD(A^{\xi})},
\end{equation}
for all~$(z,w)\in\rmD(A^{-\xi})\times\rmD(A^{\xi})$.

\begin{proposition}\label{P:extOblPro}
 Let~$\xi\in[0,1]$. The mapping~$P_{\clW_{S}}^{\widetilde\clW_{S}^\perp}\Bigr|^{\rmD(A^{-\xi})}$ is an extension of the
 oblique projection~$P_{\widetilde\clW_{S}}^{\clW_{S}^\perp}\in\clL(H)$ to~$\rmD(A^{-\xi})\supseteq H$,
 and we have the adjoint and norm identities as
 \[
  P_{\clW_{S}}^{\widetilde\clW_{S}^\perp}\Bigr|^{\rmD(A^{-\xi})}
  =\left(P_{\widetilde\clW_{S}}^{\clW_{S}^\perp}\Bigr|_{\rmD(A^{\xi})}\right)^*
  \!\!\quad\mbox{and}\quad
  \norm{P_{\clW_{S}}^{\widetilde\clW_{S}^\perp}\Bigr|^{\rmD(A^{-\xi})}}{\clL(\rmD(A^{-\xi}))}
  =\norm{P_{\widetilde\clW_{S}}^{\clW_{S}^\perp}\Bigr|_{\rmD(A^{\xi})}}{\clL(\rmD(A^{\xi}))},
 \]
 where~$P_{\widetilde\clW_{S}}^{\clW_{S}^\perp}\Bigr|_{\rmD(A^{\xi})}$ is the
 restriction in Proposition~\ref{P:restOblPro}.
\end{proposition}

Finally, we present auxiliary results that we use to analyze the stability of the nonlinear error dynamics.
\begin{proposition}\label{P:maxpoly1r}
 Let~$\eta_1>0$, $\eta_2>0$  and~$\fks\in(0,1)$. Then
 \[
 \max_{\tau\ge0}\{-\eta_1\tau+\eta_2\tau^\fks \}
 =(1-\fks)\fks^\frac{s}{1-s} \eta_2^\frac{1}{1-\fks}\eta_1^\frac{\fks}{\fks-1}.
 \]
\end{proposition}

 \begin{proposition}\label{P:ode-stab0}
 Let~$T>0$, $C_h>0$,  $\fkr>1$, and~$h\in L^\fkr_{\rm loc}(\bbR_0,\bbR)$ satisfying
 \begin{equation}\label{ode-h}
  \sup_{s\ge0}\norm{h}{L^\fkr((s,s+T),\bbR)}= C_h\le+\infty.
 \end{equation}
Let also, $\mu>0$, and~$\varrho>1$. Then
for every scalar~$\overline\mu>0$ satisfying
\begin{equation}\label{ode-condstab0}
 \overline\mu  \ge \max\left\{2\tfrac{\fkr-1}{\fkr}
 \left(\tfrac{C_h^\fkr}{\fkr\log(\varrho)}\right)^\frac{1}{\fkr-1}, 2\mu\right\}+T^\frac{-1}{\fkr}C_h,
\end{equation}
we have that the scalar~{\sc ode} system
\begin{equation}\label{ode-nonl-p0}
 \dot v=-(\overline\mu-\norm{h}{\bbR})v,\quad v(0)=v_0,
\end{equation}
is exponentially stable with rate $-\mu$ and transient bound~$\varrho$. For every~$v_0\in\bbR$,
\begin{equation}\notag
 \norm{v(t)}{}=\varrho\ex^{-\mu(t-s)}\norm{v(s)}{},\quad t\ge s\ge 0,\quad v(0)=v_0.
\end{equation}
\end{proposition}

 \begin{proposition}\label{P:ode-stab}
 Let~$T>0$, $C_h>0$,   $\fkr>1$, and~$h\in L^\fkr_{\rm loc}(\bbR_0,\bbR)$ satisfy~\eqref{ode-h}.
Let also $R>0$, $p>0$, $\mu>0$, $\varrho>1$, and~$c>1$. Then the  scalar~{\sc ode}
\begin{equation}\label{ode-nonl}
 \dot\varpi=-(\overline\mu-\norm{h}{\bbR}(1+\norm{\varpi}{\bbR}^p))\varpi,\quad \varpi(0)=\varpi_0, 
\end{equation}
is exponentially stable with transient bound~$\varrho$ and rate~$-\mu_0<-\mu$ as
\begin{equation}\label{ode-nonl-mumu0}
 \mu_0\coloneqq\max\left\{\mu,\tfrac{\log(2)}{pT},
 \left(\tfrac{\varrho^{2p+1}R^{p}C_h}{\varrho^\frac12-1}\right)^{\frac{\fkr}{\fkr-1}}
 \left(\tfrac{\fkr-1}{\fkr}\right)2^\frac{1}{\fkr-1}
 ,2^\frac{\fkr+1}{\fkr-1} \left(
  \varrho^{2p+\frac12} C_h\tfrac{p+1}{p}R^pc\right)^\frac{\fkr}{\fkr-1} p^\frac{1}{\fkr-1}\right\},
\end{equation}
if 
\begin{align}\label{ode-condstab}
&\norm{\varpi_0}{}\le R\quad\mbox{and}\quad\overline\mu 
\ge \overline\mu_*\coloneqq \max\left\{2\tfrac{\fkr-1}{\fkr}
\left(\tfrac{2C_h^\fkr}{\fkr\log(\varrho)}\right)^\frac{1}{\fkr-1}, 4\mu_0\right\}+T^\frac{-1}{\fkr}C_h.
\end{align}
That is, the solution satisfies
\begin{equation}\label{ode-stab}
 \norm{\varpi(t)}{\bbR}\le\varrho\ex^{-\mu_0 (t-s)}\norm{\varpi(s)}{\bbR},\quad\mbox{for all}\quad
 t\ge s\ge 0,\quad\mbox{if}\quad\norm{\varpi_0}{}<R.
\end{equation}
\end{proposition}

\subsection{Proof of the main Theorem~\ref{T:main}}\label{sS:proofT:main}
  We split the error into oblique components as
 \begin{align}\notag
  z=\theta+\varTheta,\quad\mbox{with}\quad \theta\coloneqq P_{\widetilde\clW_{S}}^{\clW_{S}^\perp}z
  \quad\mbox{and}\quad
  \varTheta\coloneqq P_{\clW_{S}^\perp}^{\widetilde\clW_{S}}z,
 \end{align}
 and observe that
\begin{align} 
 \dot{z}
 &=-  Az
  -A_{\rm rc}z
  -\clN(\widehat y)+\clN(y)
  -\lambda A^{-1}P_{\clW_{S}}^{\widetilde\clW_{S}^\perp}A^\ell P_{\widetilde\clW_{S}}^{\clW_{S}^\perp}z\notag\\
 &=-  Az
  -A_{\rm rc}z-\fkN_y(z)
  -\lambda A^{-1}P_{\clW_{S}}^{\widetilde\clW_{S}^\perp}A^\ell\theta\notag
  \end{align}
from which we obtain
\begin{align} 
 \tfrac{\ed}{\ed t}\norm{z}{V}^2
 &=2\left(-  Az
  -A_{\rm rc}z-\fkN_y(z)
  -\lambda A^{-1}P_{\clW_{S}}^{\widetilde\clW_{S}^\perp}A^\ell\theta,Az\right)_H.\label{dtz1}
 \end{align}

 Observe that, by direct computations, using Assumptions~\ref{A:A0sp}--\ref{A:A1}
 and the Young inequality, we find for
 all~$(\gamma_1,\gamma_2)\in(0,2)\times\bbR_0$,
  \begin{align} 
 &2\left(-  Az
  -A_{\rm rc}z,Az\right)_H
 \le-(2-\gamma_1)\norm{z}{\rmD(A)}^2+\gamma_1^{-1}C_{\rm rc}^2\norm{z}{V}^2\notag\\
 &\hspace*{1.5em}\le-(2-\gamma_1)(1-\gamma_2)\norm{\varTheta}{\rmD(A)}^2
 -(2-\gamma_1)(1-\gamma_2^{-1})\norm{\theta}{\rmD(A)}^2+\gamma_1^{-1}C_{\rm rc}^2\norm{z}{V}^2\notag\\
 &\hspace*{1.5em}\le-(2-\gamma_1)(1-\gamma_2)\norm{\varTheta}{\rmD(A)}^2
 +2\gamma_1^{-1}C_{\rm rc}^2\norm{\varTheta}{V}^2\notag\\
 &\hspace*{2.65em}-(2-\gamma_1)(1-\gamma_2^{-1})\norm{\theta}{\rmD(A)}^2
 +2\gamma_1^{-1}C_{\rm rc}^2\norm{\theta}{V}^2
.
 \label{thetas-est1}
  \end{align}
 Direct computations also give us
 \begin{align} 
 2\left(-\lambda A^{-1}P_{\clW_{S}}^{\widetilde\clW_{S}^\perp}A^\ell\theta,Az\right)_H
 &=-2\lambda\left( P_{\clW_{S}}^{\widetilde\clW_{S}^\perp}A^\ell\theta,z\right)_H
 =-2\lambda\left( A^\ell\theta,P_{\widetilde\clW_{S}}^{\clW_{S}^\perp}z\right)_H\notag\\
 &=-2\lambda\norm{\theta}{\rmD\bigl(A^\frac{\ell}2\bigr)}^2.
 \label{thetas-est2}
  \end{align}
 
For the nonlinear term,
using~\eqref{fkNyAy0} and the Young inequality, we find for all~$\gamma_3\in\bbR_0$,
\begin{subequations}\label{thetas-est3}
 \begin{align}
&2\Bigl( \fkN_y(t,z),Az\Bigr)_{H}
 \le \gamma_3 \norm{z}{\rmD(A)}^{2}\notag\\
 &\hspace*{.0em}  +\widetilde C_{\fkN2}\left(1+\widehat\gamma_0^{-\chi_5 }\right)\!
\left(1+\widehat\gamma_0^{-\frac{(\chi_5+1)\chi_2\chi_4}2  }\right)\!
   \left(1+\norm{y}{V}^{\chi_1} \right)\!
  \left(1+\norm{y}{\rmD(A)}^{\chi_2}\right)\!
  \left(1+\norm{z}{V}^{\chi_3}\right)\!
  \norm{z}{V}^2\notag
  \end{align} 
  which implies
 \begin{align} 
 2\Bigl( \fkN_y(t,z),Az\Bigr)_{H}&\le \gamma_3 \norm{z}{\rmD(A)}^{2}+\widehat C
   \Psi(y)
  \left(1+\norm{z}{V}^{\chi_3}\right)
  \norm{z}{V}^2,
  \intertext{with}
    \widehat C&=\ovlineC{n,\widetilde C_{\clN1},\dnorm{\zeta_{1}}{},
  \dnorm{\zeta_{2}}{},\frac{1}{1-\dnorm{\delta_{2}}{}},
  \frac{1}{1-\dnorm{\zeta_{2}+\delta_{2}}{}},\frac{1}{\gamma_3}},\\
 \Psi(y)&\coloneqq\left(1+\norm{y}{V}^{\chi_1} \right)
  \left(1+\norm{y}{\rmD(A)}^{\chi_2}\right).
   \end{align}
\end{subequations}

Combining~\eqref{dtz1}, \eqref{thetas-est1}, \eqref{thetas-est2}, and~\eqref{thetas-est3}, 
 it follows that
 \begin{align} 
 \tfrac{\ed}{\ed t}\norm{z}{V}^2
 &\le -\left((2-\gamma_1)(1-\gamma_2)-2\gamma_3\right)\norm{\varTheta}{\rmD(A)}^2
 +2\gamma_1^{-1}C_{\rm rc}^2\norm{\varTheta}{V}^2
 \notag\\
 &\quad-2\lambda\norm{\theta}{\rmD\bigl(A^\frac{\ell}2\bigr)}^2+\left((2-\gamma_1)
 (\gamma_2^{-1}-1)-2\gamma_3\right)\norm{\theta}{\rmD(A)}^2
 +2\gamma_1^{-1}C_{\rm rc}^2\norm{\theta}{V}^2
 \notag\\
 &\quad+\widehat C   \Psi(y)  \left(1+\norm{z}{V}^{\chi_3}\right)  \norm{z}{V}^2.\notag
\end{align}

Next, we choose/fix a triple~$(\gamma_1,\gamma_2,\gamma_3)$, small enough, such that
 \begin{align}
 & (\gamma_1,\gamma_2,\gamma_3)\in(0,2)\times(0,1)\times\bbR_0,
 \quad\mbox{and}\quad \notag\\
 &C_{\gamma,1}\coloneqq(2-\gamma_1)(1-\gamma_2)-2\gamma_3>0,\qquad 
 C_{\gamma,2}\coloneqq(2-\gamma_1)(\gamma_2^{-1}-1)-2\gamma_3>0.\notag
\end{align}

Let us recall the inequality~$\norm{q}{\rmD\bigl(A^\frac{\ell}2\bigr)}^2
\ge\underline\alpha_{S,\ell}\norm{q}{\rmD(A)}^2$, that
we have due to~\eqref{und-alpha}, the inequality~$\norm{\varTheta}{\rmD(A)}^2
\ge\beta_{S_{\sigma+}}\norm{\varTheta}{V}^2$, that
we have due to~\eqref{Poinc_const}, and also the inequality
~$\norm{\theta}{\rmD(A)}^2\ge\alpha_{1}\norm{\theta}{V}^2$,
where~$\alpha_1>0$ is the first eigenvalue of~$A$.
These inequalities lead us to
 \begin{align} 
 \tfrac{\ed}{\ed t}\norm{z}{V}^2
 &\le -\left(C_{\gamma,1}\beta_{S_{\sigma+}}-2\gamma_1^{-1}C_{\rm rc}^2\right)\norm{\varTheta}{V}^2
  -(2\lambda\underline\alpha_{S,\ell}-C_{\gamma,2}-2\gamma_1^{-1}C_{\rm rc}^2
  \alpha_1^{-1})\norm{\theta}{\rmD(A)}^2
  \notag\\
 &\quad+\widehat C   \Psi(y)  \left(1+\norm{z}{V}^{\chi_3}\right)  \norm{z}{V}^2
 .
 \label{dtz2+}
\end{align}

Next note that Assumption~\ref{A:Poincare} implies
that
\[
\beta^*_{\overline S}\coloneqq\min_{S\ge \overline S}\beta_{S_{\sigma+}}\xrightarrow[]{}+\infty
\quad\mbox{as}\quad\overline S\to+\infty.
\]

Therefore, for any given~$\overline\mu>0$ we can choose ~$S$ large enough so that
\begin{subequations}\label{choice-Slambda}
\begin{align}
C_S &\coloneqq C_{\gamma,1}\beta^*_{S}-2\gamma_1^{-1}C_{\rm rc}^2\ge2\overline\mu
\intertext{and, subsequently, we can choose~$\lambda=\lambda(S)$ large enough satisfying}
C_\lambda&\coloneqq  \alpha_1(2\lambda\underline\alpha_{S,\ell}-C_{\gamma,2}
-2\gamma_1^{-1}C_{\rm rc}^2\alpha_1^{-1})\ge2\overline\mu.
\end{align}
\end{subequations}
Hence, from~\eqref{dtz2+} and~\eqref{choice-Slambda}, we arrive at the estimate
\begin{align}
 \tfrac{\ed}{\ed t}\norm{z}{V}^2\notag
 &\le -2\overline\mu\left(\norm{\varTheta}{V}^2+\norm{\theta}{V}^2\right)
  +\widehat C   \Psi(y)  \left(1+\norm{z}{V}^{\chi_3}\right)  \norm{z}{V}^2\\
  &\le -\overline\mu\norm{z}{V}^2
  +\widehat C   \Psi(y)  \left(1+\norm{z}{V}^{\chi_3}\right)  \norm{z}{V}^2.\notag
 \end{align}

 Using Assumption~\ref{A:realy}, we arrive at
 \begin{subequations}\label{dtz5}
 \begin{align}
 &\tfrac{\ed}{\ed t}\norm{z}{V}^2
 \le -\Bigl(\overline\mu
 -\norm{h(y)}{}\left(1+\norm{z}{V}^{\chi_3}\right)\Bigr)\norm{z}{V}^2,\qquad \norm{z(0)}{}=\norm{z_0}{},\\
 \intertext{with}
&\norm{h(y)}{}=h(y)\coloneqq\widehat C   \Psi(y)\in L^{\fkr}_{\rm loc}(\bbR_0,\bbR),
\qquad \fkr\coloneqq\tfrac{2}{\chi_2}>1,\\
&\norm{h(y)}{L^{\fkr}((s,s+\tau_y),\bbR)}= \widehat C\norm{\Psi(y)}{L^{\fkr}((s,s+\tau_y),\bbR)}\notag\\
&\hspace*{4em}\le \widehat C\norm{1+\norm{y}{V}^{\chi_1} }{L^{\infty}((s,s+\tau_y),\bbR)}
  \norm{1+\norm{y}{\rmD(A)}^{\chi_2}}{L^{\fkr}((s,s+\tau_y),\bbR)}\notag\\
  &\hspace*{4em}\le \widehat C(1+C_y^{\chi_1})
  \left(\tau_y^\frac{1}{\fkr}+\norm{y}{L^{2}((s,s+\tau_y),\rmD(A))}^\frac{2}{\fkr}\right)\eqqcolon C_h.
 \end{align}
 \end{subequations}
 Therefore the norm~$\varpi=\norm{z}{V}^2$ satisfies system~\eqref{ode-nonl},
 with~$h=h(y)$ and~$p=\chi_3\ge0$.
 
 In the case~$p>0$, we use
 Proposition~\ref{P:ode-stab} to conclude that, for any given~$\varrho>1$ and~$\mu>0$, the norm satisfies
 \begin{equation}\label{final.expp}
 \norm{z(t)}{V}^2\le\varrho\ex^{-\mu (t-s)}\norm{z(s)}{V}^2,\quad\mbox{for}\quad
 t\ge s\ge 0,\quad\mbox{and}\quad\norm{z(0)}{V}^2<R, \qquad p>0,
 \end{equation}
provided we take~$\overline\mu$ large enough.

In the case~$p=0$, we use
 Proposition~\ref{P:ode-stab0} to conclude that, for any given~$\varrho>1$ and~$\mu>0$, the norm satisfies
 \begin{equation}\notag
 \norm{z(t)}{V}^2\le\varrho\ex^{-\mu (t-s)}\norm{z(s)}{V}^2,\quad\mbox{for}\quad
 t\ge s\ge 0,\quad\mbox{and}\quad z(0)\in V,  \qquad p=0,
  \end{equation}
provided we take~$\overline\mu$ large enough.

In particular~\eqref{final.expp} actually holds for all~$p\ge0$: we have that
 \begin{equation}\label{final.exp}
 \norm{z(t)}{V}^2\le\varrho\ex^{-\mu (t-s)}\norm{z(s)}{V}^2,\quad\mbox{for all}\quad
 t\ge s\ge 0,\quad\mbox{and all}\quad\norm{z(0)}{V}^2<R,
  \end{equation}
provided we take~$\overline\mu$ large enough.
That is, provided we take a large enough~$S$
and a large enough~$\lambda=\lambda(S)>0$.
Recalling~\eqref{choice-Slambda}, note that
~$C_S$ increases with~$S$, and also that, for a fixed~$S$, ~$C_\lambda$ increases with~$\lambda$.
Finally, note that from Proposition~\ref{P:ode-stab}, we can conclude that it is enough
to choose a pair~$(S_*,\lambda_*)\in\bbN_0\times\bbR_0$ such that
\[
C_{S_*}\ge 2\overline\mu,\quad C_{\lambda_*}\ge2\overline\mu,\quad\mbox{and}\quad\overline\mu\ge
=\ovlineC{\mu,\frac1{\tau_y},\varrho,\frac1{\varrho^\frac12-1},\frac{\fkr+1}{\fkr-1},R,C_h,\chi_3}.
\]
For that, using~\eqref{dtz5}, it is enough to choose, firstly~$S\ge S^*$ with~$S_*$ in the
form
\[S_*=\ovlineC{R,\mu,\varrho,\frac1{\varrho^\frac12-1},\frac1{\tau_y},\tau_y,\chi_1,\frac{2+\chi_2}{2-\chi_2},\chi_3,
C_{\rm rc},C_y},\]
and subsequently~$\lambda\ge\lambda^*(S)$ with~$\lambda^*(S)$ in the
form
\[
\lambda^*(S)=\ovlineC{\frac{1}{\underline\alpha_{S,\ell}},R,\mu,
\varrho,\frac1{\varrho^\frac12-1},\frac1{\tau_y},\tau_y,\chi_1,\frac{2+\chi_2}{2-\chi_2},\chi_3,
C_{\rm rc},C_y}.\]
We can finish the proof by recalling~\eqref{chi12} and~\eqref{chi345}.
\qed

\subsection{Boundedness of the output injection operator}\label{sS:boundInj-OPN}
Here we present estimates on the norm of the linear injection operator
\begin{align} 
\fkI_{S}^{[\lambda,\ell]}
=-\lambda A^{-1}P_{\clW_{S}}^{\widetilde\clW_{S}^\perp}A^\ell P_{\widetilde\clW_{S}}^{\clW_{S}^\perp}
\bfZ^{W_S} \in\clL(\bbR^{{S_\sigma}},H),\qquad\ell\in[0,2]. \notag
\end{align}

Due to~\eqref{sys-haty-o-bfZ} we have that~$\bfZ^{W_S} \in\clL(\bbR^{{S_\sigma}},\clW_s)$
and we show now that we can write
\begin{align} \notag
 \norm{\fkI_{S}^{[\lambda,\ell]}}{\clL(\bbR^{{S_\sigma}},H)}
 &\le \lambda \widetilde C_{\fkI_S}^{[\ell]}\norm{\bfZ^{W_S}}{\clL(\bbR^{{S_\sigma}},H)},
\end{align}
with
\begin{align} 
\widetilde C_{\fkI_S}^{[\ell]}&\coloneqq\norm{A^{-1}P_{\clW_{S}}^{\widetilde\clW_{S}^\perp}
 A^\ell P_{\widetilde\clW_{S}}^{\clW_{S}^\perp}}{\clL(H)}<+\infty.\notag
\end{align} 
To show such boundedness, we consider the cases~$\ell\in[1,2]$ and~$\ell\in[0,1]$ separately. 

In the case~$\ell\in[1,2]$, we have~$1-\ell\in[-1,0]$ and
~$P_{\clW_{S}}^{\widetilde\clW_{S}^\perp}\in\clL(\rmD(A^{1-\ell}))$,
due to Proposition~\ref{P:extOblPro}. Then we find
\begin{align} 
\widetilde C_{\fkI_S}^{[\ell]}
 &\le \norm{\Id\rest{\clW_S}}{\clL(\rmD(A^{1-\ell}),\rmD(A^{-1}))}
 \norm{P_{\clW_{S}}^{\widetilde\clW_{S}^\perp}}{\clL(\rmD(A^{1-\ell}))}
 \norm{\Id\rest{\widetilde\clW_S}}{\clL(H,\rmD(A))}
 \norm{P_{\widetilde\clW_{S}}^{\clW_{S}^\perp}}{\clL(H)}\!,\quad\ell\in[1,2],\notag
\end{align}
where we have also used~$\norm{A^{-1}}{\clL(\rmD(A^{-1}),H)}=1
=\norm{A^{\ell}}{\clL(\rmD(A^{1}),\rmD(A^{1-\ell}))}$.

In the case~$\ell\in[0,1]$, we have~$1-\ell\in[0,1]$ and
\begin{align} 
\widetilde C_{\fkI_S}^{[\ell]}
 &\le \norm{\Id}{\clL(\rmD(A),H)}
 \norm{P_{\clW_{S}}^{\widetilde\clW_{S}^\perp}}{\clL(H)}
 \norm{\Id}{\clL(\rmD(A^{1-\ell}),H)}
 \norm{\Id\rest{\widetilde\clW_S}}{\clL(H,\rmD(A))}
 \norm{P_{\widetilde\clW_{S}}^{\clW_{S}^\perp}}{\clL(H)},\quad\ell\in[0,1].\notag
\end{align}
where we have also used~$\norm{A^{-1}}{\clL(H,\rmD(A))}=1$.

Next we show that the total ``energy'' spent by the injection operator
is bounded, in case~\eqref{goal-dif-error} holds true.
Indeed, recalling~\eqref{bfZZ-orthPW} we find that
\begin{align}
\norm{\fkI_{S}^{[\lambda,\ell]} (\clZ\widehat y-w)}{L^2(\bbR_0,H)}
&=\norm{\fkI_{S}^{[\lambda,\ell]}\clZ z}{L^2(\bbR_0,H)}
\le  \lambda \widetilde C_{\fkI_S}^{[\ell]}\norm{P^{\clW_S}z}{L^2(\bbR_0,H)}\notag\\
 &\le\lambda \widetilde C_{\fkI_S}^{[\ell]}\norm{z}{L^2(\bbR_0,H)}
 \le\lambda \varrho\widetilde C_{\fkI_S}^{[\ell]}\norm{z(0)}{H}
 \left(\textstyle\int\limits_0^{+\infty}\ex^{-2\mu t}\,\ed t\right)^\frac12,\notag
  \end{align}
which leads us to
$
 \norm{\fkI_{S}^{[\lambda,\ell]} (\clZ\widehat y-w)}{L^2(\bbR_0,H)}
 \le \lambda\varrho (2\mu)^{-\frac12}\widetilde C_{\fkI_S}^{[\ell]}\norm{z(0)}{H}.
$

\subsection{On the existence and uniqueness of solutions for the error}\label{sS:unique-error}
The estimates in Section~\ref{sS:proofT:main} will also hold for Galerkin
approximations of system~\eqref{sys-error} as
\begin{subequations}\label{sys-error-Gal}
\begin{align} 
 &\dot{z}^N +Az^N+P_{E_N} A_{\rm rc}(t)z^N +P_{E_N}\fkN_y(t,z^N)
 =P_{E_N}\fkI_{S}^{[\lambda,\ell]}\clZ z^N,\qquad t\ge 0,\label{sys-error-dyn-Gal}\\
 &z^N(0)=P_{E_N}z_0\in V,\label{sys-error-ic-Gal}
\end{align}
\end{subequations}
 where~$P_{E_N}\in\clL(H)$ is the orthogonal projection in~$H$ onto the
 space~$E_N\coloneqq\linspan\{e_n\mid 1\le n\le N\}$ spanned by the first
eigenfunctions of~$A$.

Let us fix~$\varrho>1$, $\mu>0$, and~$s>0$.
We may repeat the estimates in Section~\ref{sS:proofT:main} and arrive to
the analogous of~\eqref{dtz5} and~\eqref{final.exp},
\begin{subequations}\label{dtz5-Gal}
 \begin{align}
 &\tfrac{\ed}{\ed t}\norm{z^N}{V}^2
 \le -\Bigl(\overline\mu
 -\norm{h(y)}{}\left(1+\norm{z^N}{V}^{\chi_3}\right)\Bigr)\norm{z^N}{V}^2,
 \end{align}
 \begin{equation}\label{final.exp-N}
 \norm{z^N(t)}{V}^2\le\varrho\ex^{-\mu (t-s)}\norm{z^N(s)}{V}^2,\quad\mbox{for all}\quad
 t\ge s\ge 0,\quad\norm{z_0}{V}^2<R. 
 \end{equation}
 \end{subequations}
provided we take a large enough~$S$ and a large enough~$\lambda>0$.

Note that~$\overline\mu$, $h(y)$, and~$\chi_3$ are independent of~$N$,
and that~$\norm{P_{E_N}z_0}{V}^2\le \norm{z_0}{V}^2<R$.
Hence, $S$ and~$\lambda$ can be taken independent of~$N$.
From~\eqref{final.exp-N} and~\eqref{sys-error-dyn-Gal} it follows
\[
\norm{z^N}{W((0,s),\rmD(A),H)}^2=\norm{z^N}{L^2((0,s),\rmD(A))}^2+\norm{\dot z^N}{L^2((0,s),H)}^2 \le C
\]
with~$C$ independent of~$N$. Indeed, proceeding as in~\cite[Sect.~4.3]{Rod20-eect},
multiplying the equation~\eqref{sys-error-dyn-Gal} by~$Az^N$,
\begin{align}
 \tfrac{\ed}{\ed t}\norm{z^N}{V}^2&\le-2\norm{z^N}{\rmD(A)}^2
 +2C_{\rm rc}\norm{z^N}{V}\norm{z^N}{\rmD(A)}+2\norm{\fkI_{S}^{[\lambda,\ell]}\clZ z^N}{H}\norm{z^N}{\rmD(A)}\notag\\
 &\hspace*{1em}+2\norm{\fkN_y(t,z^N)}{H}\norm{z^N}{\rmD(A)}\notag\\
 &\hspace*{-2.5em}\le-\norm{z^N}{\rmD(A)}^2+3C_{\rm rc}^2\norm{z^N}{V}^2
 +3\norm{\fkI_{S}^{[\lambda,\ell]}\clZ}{\clL(H)}^2\notag
 \norm{z^N}{H}^2
 +\widehat C\Phi(y)\left(1+\norm{z^N}{V}^{\chi_3}\right)\norm{z^N}{V}^2,
\end{align}
where we used~\eqref{thetas-est3} with~$\gamma_3=\frac{1}3$. By~\eqref{final.exp-N},
\begin{align}
 \tfrac{\ed}{\ed t}\norm{z^N}{V}^2&\le-\norm{z^N}{\rmD(A)}^2+C_3,\notag
\end{align}
 and, after integration,
\begin{align}
\norm{z^N(s)}{V}^2+\norm{z^N}{L^2((0,s),\rmD(A))}^2&\le\norm{z^N(0)}{V}^2+sC_3,\notag
\end{align}
Therefore~$\norm{z^N}{L^2((0,s),\rmD(A))}^2\le C_4$, with~$C_4$ independent of~$N$.
Using now~\eqref{sys-error-dyn-Gal}, it follows that
$\norm{\dot z^N}{L^2((0,s),H)}^2\le C_5$, with~$C_5$ independent of~$N$.
Hence, there exists a weak limit~$z^\infty\in W((0,s),\rmD(A),H)$ so that
\[
 z^N\xrightharpoonup[L^2((0,s),\rmD(A))]{} z^\infty\qquad\mbox{and}\qquad \dot z^N
 \xrightharpoonup[L^2((0,s),H)]{} \dot z^\infty.
\]
Clearly for the linear terms we have
\[
 \Phi^N\coloneqq Az^N+ A_{\rm rc}(t)z^N -\fkI_{S}^{[\lambda,\ell]}\clZ z^N
 \xrightharpoonup[L^2((0,s),H)]{} Az^\infty+ A_{\rm rc}(t)z^\infty\eqqcolon \Phi^\infty,
\]
from which we can derive
\[
 Az^N+P_{E_N} A_{\rm rc}(t)z^N -P_{E_N}\fkI_{S}^{[\lambda,\ell]}\clZ z^N\xrightharpoonup[L^2((0,s),H)]{}
 Az^\infty+ A_{\rm rc}(t)z^\infty,
\]
due to the facts that~$Az^N=P_{E_N}Az^N$, and that for all~$h\in L^2((0,s),H)$,
\begin{align}
 \left(P_{E_N}\Phi^N,h\right)_{L^2((0,s),H)}=\left(\Phi^N,h\right)_{L^2((0,s),H)}
 -\left(\Phi^N,(1-P_{E_N})h\right)_{L^2((0,s),H)},\notag
\end{align}
which gives us
\begin{align}
&\lim\limits_{N\to+\infty}\norm{\left(P_{E_N}\Phi^N,h\right)_{L^2((0,s),H)}}{\bbR}
\le\lim\limits_{N\to+\infty}\norm{\Phi^N}{L^2((0,s),H)}\norm{(1-P_{E_N})h}{L^2((0,s),H)},\notag
\end{align}
since~$\norm{\Phi^N}{L^2((0,s),H)}$ is bounded and~$\norm{(1-P_{E_N})h}{\clL(H)}\to0$.
Concerning the existence, it remains to prove that, that the nonlinear term also converges weakly.
Actually we can show that it converges strongly
\begin{equation}\label{NNconvL2}
 P_{E_N}\fkN_y(t,z^N)\xrightarrow[L^2((0,s),H)]{} \fkN_y(t,z^\infty).
 \end{equation}
In order to show~\eqref{NNconvL2} we follow arguments from~\cite[Sect.~4.3]{Rod20-eect}.
From Assumption~\ref{A:N} we have that
\begin{align}
&\norm{\fkN_y(t,z^N)-\fkN_y(t,z^\infty)}{H}=\norm{\clN(t,y+z^N)-\clN(t,y+z^\infty)}{H}\notag\\
&\hspace*{0em}\le C_\clN\textstyle\sum\limits_{j=1}^{n}
  \left( \norm{y+z^N}{V}^{\zeta_{1j}}\norm{y+z^N}{\rmD(A)}^{\zeta_{2j}}
  +\norm{y+z^\infty}{V}^{\zeta_{1j}}\norm{y+z^\infty}{\rmD(A)}^{\zeta_{2j}}\right)
  \! \norm{d^N}{V}^{\delta_{1j}}\!\norm{d^N}{\rmD(A)}^{\delta_{2j}}\!\notag\\
 &\hspace*{0em}= C_\clN\textstyle\sum\limits_{j=1}^{n}\sum\limits_{k=1}^{2}
   \norm{w_k}{V}^{\zeta_{1j}}\norm{w_k}{\rmD(A)}^{\zeta_{2j}}
   \norm{d^N}{V}^{\delta_{1j}}\!\norm{d^N}{\rmD(A)}^{\delta_{2j}}\notag
\end{align}
 with~$d^N\coloneqq z^N-z^\infty$, $w_1\coloneqq y+z^N$, and $w_2\coloneqq y+z^\infty$. Hence we arrive at
 \begin{align}
&\norm{\fkN_y(t,z^N)-\fkN_y(t,z^\infty)}{H}\notag\\
&\hspace*{0em}\le C_{\clN}\textstyle\sum\limits\limits_{j=1}^n
\norm{\left( \sum\limits_{k=1}^2\norm{w_k}{V}^{\zeta_{1j}}\norm{w_k}{\rmD(A)}^{\zeta_{2j}}
 \right)
 \norm{d^N}{\rmD(A)}^{\delta_{2j}}}{L^\frac{2}{\zeta_{2j}+\delta_{2j}}\!(\clJ_s,\bbR)}
 \norm{\norm{d^N}{V}^{\delta_{1j}}}{L^\frac{2}{1-\zeta_{2j}-\delta_{2j}}\!(\clJ_s,\bbR)},\notag
 \end{align}
whose right-hand side is similar to an expression we find in~\cite[Sect.~4.3]{Rod20-eect}.
Thus, we can repeat the arguments in~\cite{Rod20-eect}
to conclude that 
 \[
 \fkN_y(t,z^N)\xrightarrow[L^2((0,s),H)]{} \fkN_y(t,z^\infty),
 \]
 from which we can derive~\eqref{NNconvL2}, due to
 \begin{align}
 &\norm{P_{E_N}\fkN_y(t,z^N)- \fkN_y(t,z^\infty)}{L^2((0,s),H)}^2\notag\\
 &\hspace*{1.5em}=\norm{P_{E_N}\left(\fkN_y(t,z^N)- \fkN_y(t,z^\infty)\right)}{L^2((0,s),H)}^2
 +\norm{(1-P_{E_N})\fkN_y(t,z^\infty)}{L^2((0,s),H)}^2\notag\\
 &\hspace*{2.5em}
  +2(P_{E_N}\left(\fkN_y(t,z^N)- \fkN_y(t,z^\infty)\right),(1-P_{E_N})
  \fkN_y(t,z^\infty))_{L^2((0,s),H)},\notag\\
&\norm{P_{E_N}\left(\fkN_y(t,z^N)- \fkN_y(t,z^\infty)\right)}{L^2((0,s),H)}^2
 \le\norm{\left(\fkN_y(t,z^N)- \fkN_y(t,z^\infty)\right)}{L^2((0,s),H)}^2,  \notag
  \end{align}
which imply
 \begin{align}
 &\lim_{N\to+\infty}\norm{P_{E_N}\fkN_y(t,z^N)- \fkN_y(t,z^\infty)}{L^2((0,s),H)}^2\notag\\
 &\hspace*{1em}=\lim_{N\to+\infty}2(P_{E_N}\left(\fkN_y(t,z^N)- \fkN_y(t,z^\infty)\right),
 (1-P_{E_N})\fkN_y(t,z^\infty))_{L^2((0,s),H)}\notag\\
 &\hspace*{1em}\le\lim_{N\to+\infty}\!2\norm{\left(\fkN_y(t,z^N)- \fkN_y(t,z^\infty)\right)}{L^2((0,s),H)}
 \norm{(1-P_{E_N})\fkN_y(t,z^\infty)}{L^2((0,s),H)}\notag\\
 &\hspace*{1em}=0. \notag
 \end{align}
Therefore~$z^\infty$ solves system~\eqref{sys-error}.

 Finally, we show the uniqueness of the solution of system~\eqref{sys-error} in~$W((0,s),\rmD(A),H)$.
For an arbitrary solution~$z$ in~$W((0,s),\rmD(A),H)$, $z(0)=z_0$, for $G\coloneqq z-z^\infty$ we find
\begin{align}
    \dot G + AG +A_{\rm rc}G+\fkN_y(z)-\fkN_y(z^\infty)&=\fkI_{S}^{[\lambda,\ell]}\clZ G,\qquad  G(0)=0.\notag
\end{align}
Observe also that~$\fkN_y(z)-\fkN_y(z^\infty)=\clN(t,y+z)-\clN(t,y+z^\infty)$. Again we can repeat the argument
in~\cite[Sect.~4.3]{Rod20-eect}, by Assumption~\ref{A:N} to conclude that,
with~$z_1=y+z$ and~$z_2=y+z^\infty$,  
\begin{align}
 &2\Bigl( \left(\clN(t,z_1)-\clN(t,z_2)\right),AG\Bigr)_{H}
 \le\norm{G}{\rmD(A)}^{2}  + \Phi(t)\norm{G}{V}^{2},\notag\\
 &\Phi(t)\coloneqq\overline C_{\clN1}\sum\limits_{j=1}^n
 \left( \norm{z_1}{V}^{\frac{2\zeta_{1j}}{1-\delta_{2j}-\zeta_{2j}}}
 +\norm{z_2}{V}^{\frac{2\zeta_{1j}}{1-\delta_{2j}-\zeta_{2j}}}
 + \norm{z_1}{\rmD(A)}^2+\norm{z_2}{\rmD(A)}^2
 \right)\norm{G}{V}^{\frac{2\delta_{1j}}{1-\delta_{2j}}-2}.\notag
\end{align}

By using Assumption~\ref{A:A1} and the Young inequality, we find
\begin{align}
    \frac{\ed}{\ed t}\norm{G}{V}^2 
    &\le -2\norm{G}{\rmD(A)}^2+ \Phi(t)\norm{G}{V}^{2}
    + 2C_{\rm rc}^2\norm{G}{V}^{2}+ 2\norm{\fkI_{S}^{[\lambda,\ell]}\clZ}{\clL(H)}^2
    \norm{G}{V}^{2}+2\norm{G}{\rmD(A)}^{2}\notag\\
    &\le\Phi_2(t)\norm{G}{V}^{2}.\notag
   \end{align} 
with~$\Phi_2(t)\coloneqq 2C_{\rm rc}^2+ 2\norm{\fkI_{S}^{[\lambda,\ell]}\clZ}{\clL(H)}^2+\Phi(t)$. 
From~$z_1=y+z$ and~$z_2=y+z^\infty$, Assumption~\ref{A:realy}, and $\{z_1,z_2\}\subset \clC([0,s],V)\textstyle\bigcap
 L^2((0,s),\rmD(A))$, we see that~$\Phi_2$ is integrable on~$(0,s)$. Hence, by the Gronwall inequality,
\[
 \norm{G(t)}{V}^2\le\ex^{\int_{0}^t\Phi_2(\tau)\,\ed \tau}\norm{G(0)}{V}^2=0, \quad\mbox{for all}\quad t\in[0,s].
\]
That is, $G=0$ and~$z=z^\infty+G=z^\infty$. We have shown the uniqueness
of the solution for~\eqref{sys-error} in~$W((0,s),\rmD(A),H)$, for arbitrary~$s> 0$.
In other words, the solution for~\eqref{sys-error} 
is unique in~$W_{\rm loc}(\bbR_0,\rmD(A),H)\supset W(\bbR_0,\rmD(A),H)$.

\subsection{On the existence and uniqueness of solutions for systems~\eqref{sys-y-o} and~\eqref{sys-haty-o}}

\label{sS:unique-yhay}
Proceeding as in Section~\ref{sS:unique-error}, see also~\cite[Sect.~4.3]{Rod20-eect},
we can show that the solution~$y$ for system~\eqref{sys-y-o},
assumed in Assumption~\ref{A:realy} to exist in~$W_{\rm loc}(\bbR_0,\rmD(A),H)$, is unique.
Thus from Section~\ref{sS:unique-error} the solution~$z$, given by Theorem~\ref{T:main} for the
error dynamics, is also unique. Consequently,
the solution~$\widehat y=y+z\in W(\bbR_0,\rmD(A),H)$ for~\eqref{sys-haty-o} exists and is unique.

 \section{Parabolic equations evolving in rectangular domains}\label{S:parabolic-OK}

In order to apply Theorem~\ref{T:main} to the case of scalar parabolic equations, it is enough
to show that our Assumptions~\ref{A:A0sp}--\ref{A:DS} are satisfied, for the operators
defined as in Section~\ref{sS:illust-parabolic}. Assumptions~\ref{A:A0sp}--\ref{A:A0cdc} are satisfied with
 ~$A=-\nu\Delta+\Id$. Assumption~\ref{A:A1} is satisfied with 
 $A_{\rm rc}=a\Id+b\cdot\nabla\Id\in L^\infty(\bbR_0,\clL(V,H))$, because $a$ and ~$b$ are
 both essentially bounded, see~\eqref{assum.abf.parab}.
 Assumption~\ref{A:N} is proven in~\cite[Sect.~5.2]{Rod20-eect}. 
 Assumption~\ref{A:realy} will follow for suitable external forces~$f$;
 see discussion in Section~\ref{sS:motivation} and Remark~\ref{R:assumpt-fgy}.
 Assumption~\ref{A:output} is satisfied for outputs as in~\eqref{output-parab}.

It remains to show the satisfiability of Assumptions~\ref{A:DS}--\ref{A:Poincare}. For this purpose we borrow
arguments from~\cite[Sect.~4]{Rod20-sicon} and~\cite[Sect.~6]{Rod-20-CL}.
We restrict ourselves to the case of rectangular domains~$\Omega^\times=\bigtimes_{j=1}^d(0,L_j)\in\bbR^d$.

As set of sensors we take the set of indicators functions
\begin{subequations}\label{mxe-hd}
\begin{align}\label{sens-parab}
 W_S\coloneqq\{\indf_{\omega_i}\mid 1\le i\le S_\sigma\coloneqq (2S)^d\},
\end{align}
where the~$\omega_i$s are subrectangles
\begin{align}\label{mxe-hd_suppsensors}
 \omega_i=\omega_{i,S}\eqqcolon \bigtimes_{j=1}^d(p_j^{i,S},p_j^{i,S}+\tfrac{rL_j}{2S}),
 \qquad p_j^{i,S}=\tfrac{(2j-1)L_i}{4S}-\tfrac{rL_i}{4S}.
\end{align} 
\end{subequations}
as in~\cite[Sect.~4]{Rod20-sicon},
these regions are illustrated in Figure~\ref{fig.suppsensors},
for a planar rectangle~$\Omega^\times=(0,L_2)\times(0,L_2)\in\bbR^2$, where
the total volume (area) covered by the sensors is independent of~$S$. In the figure such volume
is given by~$\frac{1}{16}{\rm vol}(\Omega^\times)$,
which is~$6.25\%$ of the volume of~$\Omega^\times$, $r=\frac14$.


\setlength{\unitlength}{.002\textwidth}
\newsavebox{\Rectfw}%
\savebox{\Rectfw}(0,0){%
\linethickness{3pt}
{\color{black}\polygon(0,0)(120,0)(120,80)(0,80)(0,0)}%
}%
\newsavebox{\Rectfg}%
\savebox{\Rectfg}(0,0){%
{\color{lightgray}\polygon*(0,0)(120,0)(120,80)(0,80)(0,0)}%
}%

\newsavebox{\Rectref}%
\savebox{\Rectref}(0,0){%
{\color{white}\polygon*(0,0)(120,0)(120,80)(0,80)(0,0)}%
{\color{lightgray}\polygon*(45,30)(75,30)(75,50)(45,50)(45,30)}%
}%


\begin{figure}[h!]
\begin{center}
\begin{picture}(500,100)

 \put(0,0){\usebox{\Rectfw}}
 \put(0,0){\scalebox{.5}{\usebox{\Rectref}}}
 \put(60,0){\scalebox{.5}{\usebox{\Rectref}}}
 \put(60,40){\scalebox{.5}{\usebox{\Rectref}}}
 \put(0,40){\scalebox{.5}{\usebox{\Rectref}}}
 \put(190,0){\usebox{\Rectfw}}
 \put(190,0){\scalebox{.25}{\usebox{\Rectref}}}
\put(220,0){\scalebox{.25}{\usebox{\Rectref}}}
\put(250,0){\scalebox{.25}{\usebox{\Rectref}}}
\put(280,0){\scalebox{.25}{\usebox{\Rectref}}}
\put(190,20){\scalebox{.25}{\usebox{\Rectref}}}
\put(220,20){\scalebox{.25}{\usebox{\Rectref}}}
\put(250,20){\scalebox{.25}{\usebox{\Rectref}}}
\put(280,20){\scalebox{.25}{\usebox{\Rectref}}}
\put(190,40){\scalebox{.25}{\usebox{\Rectref}}}
\put(220,40){\scalebox{.25}{\usebox{\Rectref}}}
\put(250,40){\scalebox{.25}{\usebox{\Rectref}}}
\put(280,40){\scalebox{.25}{\usebox{\Rectref}}}
\put(190,60){\scalebox{.25}{\usebox{\Rectref}}}
\put(220,60){\scalebox{.25}{\usebox{\Rectref}}}
\put(250,60){\scalebox{.25}{\usebox{\Rectref}}}
\put(280,60){\scalebox{.25}{\usebox{\Rectref}}}
 \put(380,0){\usebox{\Rectfw}}
 \put(380,0){\scalebox{.1666}{\usebox{\Rectref}}}
\put(400,0){\scalebox{.1666}{\usebox{\Rectref}}}
\put(420,0){\scalebox{.1666}{\usebox{\Rectref}}}
\put(440,0){\scalebox{.1666}{\usebox{\Rectref}}}
 \put(460,0){\scalebox{.1666}{\usebox{\Rectref}}}
\put(480,0){\scalebox{.1666}{\usebox{\Rectref}}}
 \put(380,13.3333){\scalebox{.1666}{\usebox{\Rectref}}}
\put(400,13.3333){\scalebox{.1666}{\usebox{\Rectref}}}
\put(420,13.3333){\scalebox{.1666}{\usebox{\Rectref}}}
\put(440,13.3333){\scalebox{.1666}{\usebox{\Rectref}}}
 \put(460,13.3333){\scalebox{.1666}{\usebox{\Rectref}}}
\put(480,13.3333){\scalebox{.1666}{\usebox{\Rectref}}}
 \put(380,26.6666){\scalebox{.1666}{\usebox{\Rectref}}}
\put(400,26.6666){\scalebox{.1666}{\usebox{\Rectref}}}
\put(420,26.6666){\scalebox{.1666}{\usebox{\Rectref}}}
\put(440,26.6666){\scalebox{.1666}{\usebox{\Rectref}}}
 \put(460,26.6666){\scalebox{.1666}{\usebox{\Rectref}}}
\put(480,26.6666){\scalebox{.1666}{\usebox{\Rectref}}}
 \put(380,40){\scalebox{.1666}{\usebox{\Rectref}}}
\put(400,40){\scalebox{.1666}{\usebox{\Rectref}}}
\put(420,40){\scalebox{.1666}{\usebox{\Rectref}}}
\put(440,40){\scalebox{.1666}{\usebox{\Rectref}}}
 \put(460,40){\scalebox{.1666}{\usebox{\Rectref}}}
\put(480,40){\scalebox{.1666}{\usebox{\Rectref}}}
 \put(380,53.3333){\scalebox{.1666}{\usebox{\Rectref}}}
\put(400,53.3333){\scalebox{.1666}{\usebox{\Rectref}}}
\put(420,53.3333){\scalebox{.1666}{\usebox{\Rectref}}}
\put(440,53.3333){\scalebox{.1666}{\usebox{\Rectref}}}
 \put(460,53.3333){\scalebox{.1666}{\usebox{\Rectref}}}
\put(480,53.3333){\scalebox{.1666}{\usebox{\Rectref}}}
 \put(380,66.6666){\scalebox{.1666}{\usebox{\Rectref}}}
\put(400,66.6666){\scalebox{.1666}{\usebox{\Rectref}}}
\put(420,66.6666){\scalebox{.1666}{\usebox{\Rectref}}}
\put(440,66.6666){\scalebox{.1666}{\usebox{\Rectref}}}
 \put(460,66.6666){\scalebox{.1666}{\usebox{\Rectref}}}
\put(480,66.6666){\scalebox{.1666}{\usebox{\Rectref}}}
\put(40,85){$S=1$}
\put(230,85){$S=2$}
\put(420,85){$S=3$}
\linethickness{2pt}%
{\color{blue}%

\Dashline(460,0)(500,0){2}%
\Dashline(500,0)(500,26.6666){2}%
\Dashline(500,26.6666)(460,26.6666){2}%
\Dashline(460,26.6666)(460,0){2}%
}

\end{picture}
\end{center}
\caption{The sensor supports as in~\eqref{mxe-hd_suppsensors}.
Case~$\Omega^\times\subset\bbR^d$, $d=2$.} \label{fig.suppsensors}
\end{figure}
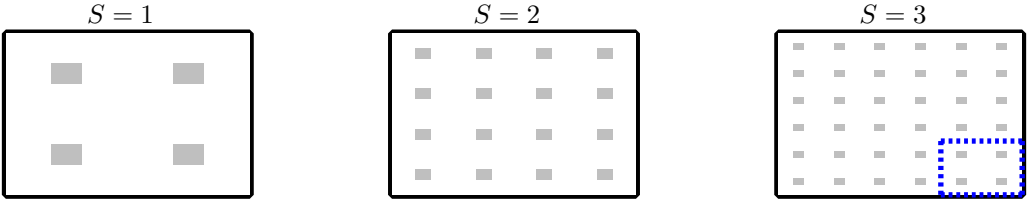

The choice of the auxiliary set~$\widetilde W_S\subset\rmD(A)$ is at our disposal.
For example, we can take the Cartesian product eigenfunctions of~$A$ as
in~\cite[Sect.~4]{Rod20-sicon}, 
\begin{subequations}\label{choice-tilW}
\begin{align}
 &\widetilde W_S=E_S\coloneqq\{e_\bfi\mid \bfi\in\widehat\bbS^d\},\qquad e_\bfi(x)
 \coloneqq\bigtimes_{j=1}^d e_{\bfi_j}(x_j),\qquad\widehat\bbS\coloneqq\{1,2,\dots,2S\};\label{choice-tilW-eigs}
 \intertext{or, the more ad-hoc functions as in~\cite[Sect.~6]{Rod-20-CL}}
  &\widetilde W_S=\mathbf\Phi_S\coloneqq\{\Phi_i\mid 1\le i\le (2S)^d\},\qquad
  \Phi_i(x)\coloneqq\bigtimes_{j=1}^d \sin^2(S\tfrac{x_j-p_j^{i,S}}{L_j})\label{choice-tilW-adhoc},
 \intertext{or, we could construct and take the functions}
  &\widetilde W_S=\fkA_S\coloneqq\{A^{-2}\indf_{\omega_i}\mid 1\le i\le (2S)^d\}.\label{choice-tilW-constrA}
\end{align}
\end{subequations}
From~\cite[Sect.~4]{Rod20-sicon} and~\cite[Sect.~6]{Rod-20-CL} we know that Assumption~\ref{A:DS} is satisfied
for both choices in~\eqref{choice-tilW}, with~$\sigma(S)\coloneqq (2S)^d$.

It remains to show the satisfiability of Assumption~\ref{A:Poincare}.

\subsection{Previous related work.}
In~\cite[Sect.~5]{Rod20-sicon} it has been shown that a Poincar\'e-like condition as
\begin{equation}\label{feedtriple-PoincareVH}
\lim_{N\to+\infty}\inf_{Q\in(V\bigcap\clO_{N^d}^\perp)\setminus\{0\}}\tfrac{\norm{Q}{V}^2}{\norm{Q}{H}^2}=+\infty.
 \end{equation}
is satisfied for the sensors as indicator functions of the regions
\begin{align}\notag
 \omega_i=\omega_{i,N}\eqqcolon \bigtimes_{j=1}^d(p_j^{i,N},p_j^{i,N}+\tfrac{rL_j}N),
 \qquad p_j^{i,N}=\tfrac{(2j-1)L_i}{2N}-\tfrac{rL_i}{2N}.
\end{align}
Here we prove that the analogous condition in Assumption~\ref{A:Poincare}
is also satisfied for the subsequence of sets of sensors as in~\eqref{mxe-hd_suppsensors}.

The proof of~\eqref{feedtriple-PoincareVH} is given for sensors constructed as
in Figure~\ref{fig.suppsensors-supseq}, with regions


\begin{figure}[h!]
\begin{center}
\begin{picture}(500,100)


\put(0,0){\usebox{\Rectfw}}%
\put(0,0){\usebox{\Rectref}}
 \put(190,0){\usebox{\Rectfw}}
 \put(190,0){\scalebox{.5}{\usebox{\Rectref}}}
 \put(250,0){\scalebox{.5}{\usebox{\Rectref}}}
 \put(250,40){\scalebox{.5}{\usebox{\Rectref}}}
 \put(190,40){\scalebox{.5}{\usebox{\Rectref}}}
 \put(380,0){\usebox{\Rectfw}}
 \put(380,0){\scalebox{.3333}{\usebox{\Rectref}}}
\put(420,0){\scalebox{.3333}{\usebox{\Rectref}}}
\put(460,0){\scalebox{.3333}{\usebox{\Rectref}}}
\put(380,26.6666){\scalebox{.3333}{\usebox{\Rectref}}}
\put(420,26.6666){\scalebox{.3333}{\usebox{\Rectref}}}
\put(460,26.6666){\scalebox{.3333}{\usebox{\Rectref}}}
\put(380,53.3333){\scalebox{.3333}{\usebox{\Rectref}}}
\put(420,53.3333){\scalebox{.3333}{\usebox{\Rectref}}}
\put(460,53.3333){\scalebox{.3333}{\usebox{\Rectref}}}
\put(40,85){$N=1$}
\put(230,85){$N=2$}
\put(420,85){$N=3$}

\end{picture}
\end{center}

\vspace*{1em}
%
\begin{center}
\begin{picture}(500,100)

 \put(0,0){\usebox{\Rectfw}}
 \put(0,0){\scalebox{.25}{\usebox{\Rectref}}}
\put(30,0){\scalebox{.25}{\usebox{\Rectref}}}
\put(60,0){\scalebox{.25}{\usebox{\Rectref}}}
\put(90,0){\scalebox{.25}{\usebox{\Rectref}}}
\put(0,20){\scalebox{.25}{\usebox{\Rectref}}}
\put(30,20){\scalebox{.25}{\usebox{\Rectref}}}
\put(60,20){\scalebox{.25}{\usebox{\Rectref}}}
\put(90,20){\scalebox{.25}{\usebox{\Rectref}}}
\put(0,40){\scalebox{.25}{\usebox{\Rectref}}}
\put(30,40){\scalebox{.25}{\usebox{\Rectref}}}
\put(60,40){\scalebox{.25}{\usebox{\Rectref}}}
\put(90,40){\scalebox{.25}{\usebox{\Rectref}}}
\put(0,60){\scalebox{.25}{\usebox{\Rectref}}}
\put(30,60){\scalebox{.25}{\usebox{\Rectref}}}
\put(60,60){\scalebox{.25}{\usebox{\Rectref}}}
\put(90,60){\scalebox{.25}{\usebox{\Rectref}}}
\put(190,0){\usebox{\Rectfw}}
 \put(190,0){\scalebox{.2}{\usebox{\Rectref}}}
\put(214,0){\scalebox{.2}{\usebox{\Rectref}}}
\put(238,0){\scalebox{.2}{\usebox{\Rectref}}}
\put(262,0){\scalebox{.2}{\usebox{\Rectref}}}
 \put(286,0){\scalebox{.2}{\usebox{\Rectref}}}
\put(190,16){\scalebox{.2}{\usebox{\Rectref}}}
\put(214,16){\scalebox{.2}{\usebox{\Rectref}}}
\put(238,16){\scalebox{.2}{\usebox{\Rectref}}}
\put(262,16){\scalebox{.2}{\usebox{\Rectref}}}
 \put(286,16){\scalebox{.2}{\usebox{\Rectref}}}
\put(190,32){\scalebox{.2}{\usebox{\Rectref}}}
\put(214,32){\scalebox{.2}{\usebox{\Rectref}}}
\put(238,32){\scalebox{.2}{\usebox{\Rectref}}}
\put(262,32){\scalebox{.2}{\usebox{\Rectref}}}
 \put(286,32){\scalebox{.2}{\usebox{\Rectref}}}
\put(190,48){\scalebox{.2}{\usebox{\Rectref}}}
\put(214,48){\scalebox{.2}{\usebox{\Rectref}}}
\put(238,48){\scalebox{.2}{\usebox{\Rectref}}}
\put(262,48){\scalebox{.2}{\usebox{\Rectref}}}
 \put(286,48){\scalebox{.2}{\usebox{\Rectref}}}
\put(190,64){\scalebox{.2}{\usebox{\Rectref}}}
\put(214,64){\scalebox{.2}{\usebox{\Rectref}}}
\put(238,64){\scalebox{.2}{\usebox{\Rectref}}}
\put(262,64){\scalebox{.2}{\usebox{\Rectref}}}
 \put(286,64){\scalebox{.2}{\usebox{\Rectref}}}
%
 \put(380,0){\usebox{\Rectfw}}
 \put(380,0){\scalebox{.1666}{\usebox{\Rectref}}}
\put(400,0){\scalebox{.1666}{\usebox{\Rectref}}}
\put(420,0){\scalebox{.1666}{\usebox{\Rectref}}}
\put(440,0){\scalebox{.1666}{\usebox{\Rectref}}}
 \put(460,0){\scalebox{.1666}{\usebox{\Rectref}}}
\put(480,0){\scalebox{.1666}{\usebox{\Rectref}}}
 \put(380,13.3333){\scalebox{.1666}{\usebox{\Rectref}}}
\put(400,13.3333){\scalebox{.1666}{\usebox{\Rectref}}}
\put(420,13.3333){\scalebox{.1666}{\usebox{\Rectref}}}
\put(440,13.3333){\scalebox{.1666}{\usebox{\Rectref}}}
 \put(460,13.3333){\scalebox{.1666}{\usebox{\Rectref}}}
\put(480,13.3333){\scalebox{.1666}{\usebox{\Rectref}}}
 \put(380,26.6666){\scalebox{.1666}{\usebox{\Rectref}}}
\put(400,26.6666){\scalebox{.1666}{\usebox{\Rectref}}}
\put(420,26.6666){\scalebox{.1666}{\usebox{\Rectref}}}
\put(440,26.6666){\scalebox{.1666}{\usebox{\Rectref}}}
 \put(460,26.6666){\scalebox{.1666}{\usebox{\Rectref}}}
\put(480,26.6666){\scalebox{.1666}{\usebox{\Rectref}}}
 \put(380,40){\scalebox{.1666}{\usebox{\Rectref}}}
\put(400,40){\scalebox{.1666}{\usebox{\Rectref}}}
\put(420,40){\scalebox{.1666}{\usebox{\Rectref}}}
\put(440,40){\scalebox{.1666}{\usebox{\Rectref}}}
 \put(460,40){\scalebox{.1666}{\usebox{\Rectref}}}
\put(480,40){\scalebox{.1666}{\usebox{\Rectref}}}
 \put(380,53.3333){\scalebox{.1666}{\usebox{\Rectref}}}
\put(400,53.3333){\scalebox{.1666}{\usebox{\Rectref}}}
\put(420,53.3333){\scalebox{.1666}{\usebox{\Rectref}}}
\put(440,53.3333){\scalebox{.1666}{\usebox{\Rectref}}}
 \put(460,53.3333){\scalebox{.1666}{\usebox{\Rectref}}}
\put(480,53.3333){\scalebox{.1666}{\usebox{\Rectref}}}
 \put(380,66.6666){\scalebox{.1666}{\usebox{\Rectref}}}
\put(400,66.6666){\scalebox{.1666}{\usebox{\Rectref}}}
\put(420,66.6666){\scalebox{.1666}{\usebox{\Rectref}}}
\put(440,66.6666){\scalebox{.1666}{\usebox{\Rectref}}}
 \put(460,66.6666){\scalebox{.1666}{\usebox{\Rectref}}}
\put(480,66.6666){\scalebox{.1666}{\usebox{\Rectref}}}
\put(40,85){$N=4$}%
\put(230,85){$N=5$}%
\put(420,85){$N=6$}%
\linethickness{2pt}%
{\color{blue}%

\Dashline(480,0)(500,0){2}%
\Dashline(500,0)(500,13.3333){2}%
\Dashline(500,13.3333)(480,13.3333){2}%
\Dashline(480,13.3333)(480,0){2}%
}

\end{picture}
\end{center}
\caption{The sensor supports as in~\eqref{mxe-hd_suppsensors}. Case~$\Omega^\times\subset\bbR^d$, $d=2$.}
\label{fig.suppsensors-supseq}
\end{figure}
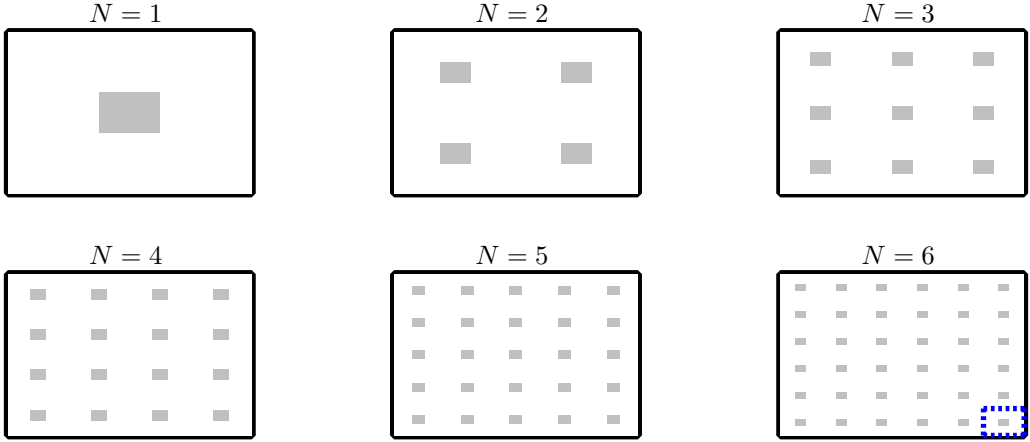

The proof in~\cite[Sect.~4]{Rod20-sicon}
takes the case of~$N=1$, corresponding to~$1$ sensor, as a reference  and
is based on the observation that the positioning of the actuators in~\eqref{mxe-hd}
gives us a partition of~$\Omega^\times=\bigcup_{\bfi\in\widehat N^d}\fkR_\bfi$,
$\widehat N=\{1,2,\dots,N\}$ into rectangles~$\fkR_\bfi$ which
are rescaled copies of the rectangle corresponding to the case of~$1^d=1$
sensor~$1_{\omega^\times}=1_{\omega^\times_{(1,1,\dots,1),1}}$,
with the rescaling factor~$N^{-1}$; see one of these copies highlighted, in
Figure~\ref{fig.suppsensors-supseq},
at the bottom-right corner of the case~$N=6$. Then, the Poincar\'e constant
in~\eqref{feedtriple-PoincareVH} is shown to satisfy, for $N>1$,
\begin{subequations}\label{PoincVH}
\begin{equation}\label{PoincVH1}
\inf_{Q\in(V\bigcap\clO_{N^d}^\perp)\setminus\{0\}}\tfrac{\norm{Q}{V}^2}{\norm{Q}{H}^2}
\ge \left(\nu N^{2}D_0C_0 +1\right),
\quad D_0\coloneqq\inf_{Q\in(V\bigcap\clO_{1}^\perp)\setminus\{0\}}\tfrac{\norm{Q}{V}^2}{\norm{Q}{H}^2},
\end{equation}
where~$D_0$ is the Poincar\'e constant in~\eqref{feedtriple-PoincareVH}, in~$\Omega^\times$,
for the case of~$1$ sensor.
Further~$C_0$ is a constant satisfying, in the case~$N=1$,
\begin{equation}\label{PoincVH2}
 C_0\norm{h}{V}^2\le\norm{\nabla_x (h)}{L^2(\Omega^\times)^d}^2
 +\norm{(h,1_{\omega^\times})}{\bbR}^2,\quad\mbox{for all}\quad h\in H^1(\Omega).
\end{equation}
\end{subequations}

\subsection{Satisfiability of Assumption~\ref{A:Poincare}}
We have mentioned that the proof in~\cite[Sect.~4]{Rod20-sicon} uses the case~$N=1$
as a reference to derive~\eqref{PoincVH2}. Here we use the case~$S=1$, corresponding to~$2^d$ sensors,
as a reference to derive the analogous estimate required in Assumption~\ref{A:Poincare}.
\begin{lemma}\label{L:Poinc1}
For~$S=1$
we have an analogous version of~\eqref{PoincVH2} as
\begin{equation}\label{PoincDAV2}
 C_0\norm{h}{\rmD(A)}^2\le\norm{\nabla_x^2 h}{L^2(\Omega^\times)^{d^2}}^{2}
 +\textstyle\sum\limits_{\bfj\in \{1,2\}^d}\norm{(h,1_{\omega_{\bfj,1}^\times})}{\bbR}^{2},
 \quad\mbox{for all}\quad h\in H^2(\Omega),
\end{equation}
where~$\{\omega_{\bfj,1}^\times\mid \bfj\in \{1,2\}^d\}=\{\omega_{i,S}\mid i\in \{1,2,\dots,2^d\}\}$.
\end{lemma}
For the proof we will need some auxiliary results.

Note that, the number of sensors is given by~$S_\sigma=(2S)^d$, thus~$2^d$ for~$S=1$.

Above,~$\nabla_x^2$ stands for second order derivatives,
\begin{align}
\norm{\nabla_x^2 h}{L^2(\Omega^\times)^{d^2}}
&\coloneqq\left(\textstyle\sum\limits_{\bfk\in K_{d,2}}\norm{\frac{\p^{\bfk_1}\p^{\bfk_2}
\dots\p^{\bfk_d} h}{\p x_1^{\bfk_1}\p x_2^{\bfk_2}\dots\p x_d^{\bfk_d}}}{L^2(\Omega^\times)}^2\right)^\frac{1}{2},\notag\\
K_{d,2}&\coloneqq\{\bfk\in\{0,1,2\}^d\mid\textstyle\sum\limits_{s=1}^d\bfk_s=2\}.\notag
\end{align}

Note that the locations as in~\eqref{mxe-hd_suppsensors} induce a
partition of~$\Omega^\times$ with $S^d$ rescaled copies of the case~$S=1$.
See Figure~\ref{fig.suppsensors}, case~$S=3$, where a rescaled copy of the case~$S=1$ is highlighted
at the bottom-right corner.

The following lemma can be found in~\cite[Ch.~1, Sect.~1.7, Thm.~1.6]{Necas67},
written in a slightly different way.
\begin{lemma}\label{L:normH2quot}
 Let $\bbP_{\times,1}\coloneqq\left\{c_0+\sum\limits_{j=1}^d a_jx_j\mid c_0\in\bbR,
 a=(a_1,a_2,\dots,a_d)\in\bbR^d\right\}$ be the set of
 polynomials of degree at most~$1$ defined in~$\Omega^\times$, and consider its orthogonal
 in~$H^2(\Omega^\times)$, $\bbP_{\times,1}^{\perp,H^2}\coloneqq
 \left\{h\in H^2(\Omega^\times)\mid (h,p)_{H^2(\Omega^\times)}=0,\mbox{ for all }
 p\in\bbP_{\times,1}\right\}$. Then
 there exists a constant~$C>0$ such that
 \[
  \norm{h}{H^2(\Omega^\times)}^2\le C\norm{\nabla_x^2 h}{L^2(\Omega^\times)^{d^2}}^{2},
  \quad\mbox{for all}\quad h\in\bbP_{\times,1}^{\perp,H^2}.
 \]
\end{lemma}

\begin{proposition}\label{P:normPoly}
Let~$\bfJ_{d,2}\coloneqq\{\bfj\in\{1,2\}^d\mid\textstyle\sum_{j=1}^d\bfj_j\le d+1\}$.
Then, the seminorm~$\fkS(\Bigcdot)\coloneqq\left(\textstyle\sum\limits_{\bfj\in \bfJ_{d,2}}
\norm{(\Bigcdot,1_{\omega_{\bfj,1}^\times})}{\bbR}^{2}\right)^\frac12$ is a norm
in~$\bbP_{\times,1}$.
\end{proposition}
The proof is given in Section~\ref{sS:proofP:normPoly}. Note that~$\bfJ_{d,2}\subset\bbS^d$ has
cardinality~$\#\bfJ_{d,2}=d+1=\dim \bbP_{\times,1}$.

\begin{corollary}\label{C:normH2}
 The usual norm~$\norm{\Bigcdot}{H^2(\Omega^\times)}$, in~$H^2(\Omega^\times)$, is equivalent to the norm
 \[
 \left(\norm{\nabla_x^2 \Bigcdot}{L^2(\Omega^\times)^{d^2}}^2
 +\textstyle\sum\limits_{\bfj\in \bfJ_{d,2}}
 \norm{(\Bigcdot,1_{\omega_{\bfj,1}^\times})_{L^2(\Omega^\times)}}{\bbR}^{2}\right)^\frac12.
 \]
\end{corollary}
\begin{proof} The proof is standard and can be done by repeating the arguments
from the proofs in~\cite[Ch.~1, Sect.~1.7, Thms.~1.8 and~1.10]{Necas67}.
Indeed, it is enough to observe that we have~$H^2(\Omega^\times)
=\bbP_{\times,1}^{\perp,H^2}\oplus\bbP_{\times,1}$,
and use Proposition~\ref{P:normPoly}.
\end{proof}

\begin{proof}[Proof of Lemma~\ref{L:Poinc1}]
Let~$h\in H^2(\Omega^\times)$. By Corollary~\ref{C:normH2} there exists~$C_1>0$ such
that
\begin{align}
 \norm{\nabla_x^2 h}{L^2(\Omega^\times)^{d^2}}^{2}
 +\textstyle\sum\limits_{\bfj\in \{1,2\}^2}\norm{(h,1_{\omega_{\bfj,1}^\times})}{\bbR}^{2}
 &\ge
 \norm{\nabla_x^2 h}{L^2(\Omega^\times)^{d^2}}^2
 +\textstyle\sum\limits_{\bfj\in \bfJ_{d,2}}\norm{(h,1_{\omega_{\bfj,1}^\times})}{\bbR}^{2}\notag\\
 &\ge C_1\norm{h}{H^2(\Omega^\times)}^2\ge C_1C_2\norm{h}{\rmD(A)}^2,\notag
 \end{align}
 where~$C_2$ is a constant satisfying~$\norm{h}{\rmD(A)}^2\le C_2^{-1}\norm{h}{H^2(\Omega^\times)}^2$.
 That is, we may take~$C_0=C_1C_2$ in~\eqref{PoincDAV2}.
 Note that
 such~$C_2$ can be found as~$\norm{h}{\rmD(A)}^2=\norm{-\nu\Delta h+h}{L^2(\Omega^\times)}^2
 \le2(\norm{-\nu\Delta h}{L^2(\Omega^\times)}^2
 +\norm{h}{L^2(\Omega^\times)}^2)\le 2(\nu^2d\norm{h}{H^2(\Omega^\times)}^2
  +\norm{h}{H^2(\Omega^\times)}^2)$, that is, we may take~$C_2^{-1}=2(\nu^2d+1)$. 
\end{proof}

Proceeding as in~\cite[Sect.~4]{Rod20-sicon}, we observe that for a suitable
translation~$\fkT_\bfi$, the injective affine transformation
\begin{align}
 \varPhi_\bfi\colon\Omega^{\times}\to \fkR_\bfi,\qquad x\mapsto z^\bfi\coloneqq \frac{x}{S}+ \fkT_\bfi,
 \qquad\bfi\in\widehat\bbS^d,\notag
\end{align} 
maps~$\Omega^{\times}$ onto~$\fkR_\bfi$, and  the sensor
regions~$\{{\omega^\times_{\bfj,1}}\mid\bfj\in\{1,2\}^d\}$ onto rescaled sensor regions
~$\{\omega_{\bfi(\bfj),S}^{\times}\subset\fkR_\bfi\mid\bfj\in\{1,2\}^d\}$
in the corresponding copy~$\fkR_\bfi$.
 \begin{align}
 \varPhi_\bfi(\Omega^{\times})=\fkR_\bfi,\qquad \varPhi(\omega_{\bfj,S}^{\times})
 =\omega_{\bfi(\bfj),S}^{\times}.\notag
\end{align}

 From~$\ed x_n=S\ed z^\bfi_n$ and~$\tfrac{\p }{\p x_n}
 =\frac{1}{S}\tfrac{\p }{\p z^\bfi_n}$,
 for~$\bfk\in K_{d,2}$ we have~$\tfrac{\p^{\bfk_1}\p^{\bfk_2}
\dots\p^{\bfk_d} Q(x)}{\p x_1^{\bfk_1}\p x_2^{\bfk_2}\dots\p x_d^{\bfk_d}}
=\frac{1}{S^2}\tfrac{\p^{\bfk_1}\p^{\bfk_2}
\dots\p^{\bfk_d} Q(z)}{\p {z_1^\bfi}^{\bfk_1}\p {z_2^\bfi}^{\bfk_2}\dots\p {z_d^\bfi}^{\bfk_d}}$.
Further for~$Q\in\rmD(A)\bigcap\clO_{2^d}^\perp$ we find
\begin{align}
 &\int_{\Omega^\times}\sum_{\bfk\in K_{d,2}}^d(\tfrac{\p^{\bfk_1}\p^{\bfk_2}
\dots\p^{\bfk_d} Q(x)}{\p x_1^{\bfk_1}\p x_2^{\bfk_2}\dots\p x_d^{\bfk_d}})^2\,\ed x
 =\int_{\fkR_\bfi}\sum_{\bfk\in K_{d,2}}^d(\tfrac{1}{S})^4(\tfrac{\p^{\bfk_1}\p^{\bfk_2}
\dots\p^{\bfk_d} Q(z)}{\p {z_1^\bfi}^{\bfk_1}\p {z_2^\bfi}^{\bfk_2}\dots\p {z_d^\bfi}^{\bfk_d}})^2
\,S^d\ed z^\bfi,\notag\\
 &\int_{\Omega^\times} Q(x)^2\,\ed x
 =\int_{\fkR_\bfi}Q(\varPhi_\bfi^{-1}(z))^2\,S^d\ed z^\bfi,\notag\\
 &\int_{\omega} g(x)Q(x)\,\ed x
 =\int_{\varPhi(\omega)}g(\varPhi_\bfi^{-1}(z))Q(\varPhi_\bfi^{-1}(z))\,S^d\ed z^\bfi,
 \mbox{ for all }g\in L^2(\omega),\quad\!\omega\subseteq\Omega^\times.\notag
\end{align}
which give us, 
\begin{align}
 \norm{\nabla_x^2 (Q)}{L^2(\Omega^\times)^{2d}}^2
 &=S^{d-4}\norm{\nabla_z^2 Q\circ\varPhi_\bfi^{-1}}{L^2(\fkR_\bfi)^d}^2,\notag\\
 \norm{Q}{L^2(\Omega^\times)}^2&=S^{d}\norm{Q\circ\varPhi_\bfi^{-1}}{L^2(\fkR_\bfi)}^2,\notag\\
 \norm{Q}{L^2({\omega})}^2&=S^{d}\norm{Q\circ\varPhi_\bfi^{-1}}{L^2(\varPhi_\bfi(\omega))}^2.\notag
 \end{align}
Further, denoting~$[\omega_1^\times]_2\coloneqq\{\indf_{\omega^\times_{\bfj,1}}\mid\bfj\in\{1,2\}^d\}$
and~$[\omega_{\bfi,S}^\times]_2\coloneqq\{\indf_{\omega^\times_{\bfi(\bfj),S}}\mid\bfj\in\{1,2\}^d\}$
and, choosing~$g\in[\omega_1^\times]_2$, we also find 
\[
Q\circ\varPhi_\bfi^{-1}\in[\omega^\times]_2^\perp\quad\Longleftrightarrow
\quad Q\in [\omega_{\bfi,S}^\times]_2^\perp.
\]

\begin{lemma}\label{L:Poinc2}
 For a suitable constant~$C>0$, we have
\begin{equation}\label{PoincDAV1}
\inf_{Q\in(\rmD(A)\bigcap\clO_{(2 S)^d}^\perp)\setminus\{0\}}\tfrac{\norm{Q}{\rmD(A)}^2}{\norm{Q}{V}^2}
\ge C D_0 S^2,
\quad D_0\coloneqq\inf_{Q\in(\rmD(A)\bigcap\clO_{2^d}^\perp)\setminus\{0\}}
\tfrac{\norm{Q}{H^2(\Omega)}^2}{\norm{Q}{H^1(\Omega)}^2}.
\end{equation}
\end{lemma}
\begin{proof}
For arbitrary given~$Q\in \clW_{S}^\perp\bigcap \rmD(A)$, since the
norms~$\norm{\Bigcdot}{\rmD(A)}^2$ and~$\norm{\Bigcdot}{H^2(\Omega^\times}^2$
are equivalent (for both Dirichlet and Neumann boundary conditions), and since~$\fkS(Q)=0$, we find 
\[
C_4\norm{\nabla_x^2 Q}{L^2(\Omega^\times)^d}^2\le\norm{Q}{\rmD(A)}^2
\ge C_3 \norm{\nabla_x^2 Q}{L^2(\Omega^\times)^d}^2
\]
for suitable constants~$C_3>0$, $C_4>0$. Furthermore,
\begin{align}
 \norm{Q}{\rmD(A)}^2&\ge C_3 \norm{\nabla_x^2 Q}{L^2(\Omega^\times)^d}^2
 = C_3\sum\limits_{\bfi\in\widehat \bbS^d}\norm{\nabla_x^2 Q}{L^2(\fkR_\bfi)^d}^2\notag\\
 &=C_3\sum\limits_{\bfi\in\widehat \bbS^d} S^{4-d}
 \norm{\nabla_{\varPhi_\bfi^{-1}(x)}^2 (Q\circ\varPhi_\bfi)}{L^2(\Omega^\times)^d}^2
 \ge C_3C_1S^{4-d}\sum\limits_{\bfi\in\bbS^d}\norm{Q\circ\varPhi_\bfi}{H^2(\Omega^\times)}^2\notag\\
 &\ge C_3C_1D_0S^{4-d}\sum\limits_{\bfi\in\widehat\bbS^d}
 \norm{Q\circ\varPhi_\bfi}{H^1(\Omega^\times)}^2\notag\\
 &= C_3C_1D_0S^{4-d}\sum\limits_{\bfi\in\widehat\bbS^d}
 \left(\norm{\nabla_x^1Q\circ\varPhi_\bfi}{L^2(\Omega^\times)}^2
 +\norm{Q\circ\varPhi_\bfi}{L^2(\Omega^\times)}^2\right)\notag
 \end{align}
 with~$D_0$ as in~\eqref{PoincDAV1}.
  By using the relation
 \begin{align}
  \norm{\nabla_{\varPhi_\bfi^{-1}(x)}^1 (Q)}{L^2(\Omega^\times)^{d}}^2
  &=S^{d-2}\norm{\nabla_x Q\circ\varPhi^{-1}}{L^2(\fkR_\bfi)^d}^2,\notag
 \end{align}
 which we can find in~\cite{Rod20-sicon}, we arrive at
 \begin{align}
 \norm{Q}{\rmD(A)}^2&\ge C_3C_1D_0S^{4-d}\sum\limits_{\bfi\in\widehat\bbS^d}
 \left(S^{d-2}\norm{\nabla_x^1Q}{L^2(\fkR_{\bfi})}^2+S^{d}\norm{Q}{L^2(\fkR_{\bfi})}^2\right)\notag\\
 &= C_3C_1D_0S^2\sum\limits_{\bfi\in\widehat\bbS^d}\!
 \left(\norm{\nabla_x^1Q}{L^2(\fkR_{\bfi})}^2\!+S^2\norm{Q}{L^2(\fkR_{\bfi})}^2\right)
 \ge C_3C_1D_0S^2\norm{Q}{H^1(\Omega^\times)}^2\notag\\
 &\ge C_3C_1C_5D_0S^2\norm{Q}{V}^2, \notag
 \end{align}
 which gives us~\eqref{PoincDAV1}, with~$C=C_3C_1C_5$,
 and~$C_5\coloneqq\inf\limits_{Q\in V\setminus\{0\}}\frac{\norm{Q}{H^1(\Omega^\times)}^2}{\norm{Q}{V}^2}$.
 \end{proof}

 Note that~\eqref{PoincDAV1} implies the satisfiability of Assumption~\ref{A:Poincare}.

 Note that we have proven the satisfiability of Assumption~\ref{A:Poincare} for rectangular domains.
  We end this section with the following conjecture.
  \begin{conjecture}\label{Cj:nonrect}
Assumption~\ref{A:Poincare} can be satisfied for smooth domains.
\end{conjecture}

An analogous conjecture has been stated in~\cite[Sect.~7.3]{Rod20-sicon}, where we can also find
arguments supporting the conjecture.

Finally, we end this section with the following result concerning Remark~\ref{R:underAlpha}.
\begin{proposition}\label{P:underAlpha}
For pairwise disjoint sensor rectangular regions~$\omega_{i,S}$ as
in~\eqref{mxe-hd_suppsensors}, and auxiliary functions
$\widetilde W_S$ as in~\eqref{choice-tilW-adhoc},
we have that, for any~$\theta\in\widetilde W_S\subset\rmD(A)$,
\begin{align}\notag
\norm{\theta}{V}^2= (C_1S^2 +1)\norm{\theta}{H}^2\quad\mbox{and}\quad
\norm{\theta}{\rmD(A)}^2=(C_2S^4 +2 C_1S^2+1)\norm{\theta}{H}^2,
 \end{align}
 with~$C_1= \tfrac{4\nu\pi^2}{3}{\textstyle\sum\limits_{i=1}^{d}}
\tfrac{1}{L_i^2}$ and~$
C_2= \tfrac{\nu^216\pi^4}{3}{\textstyle\sum\limits_{i=1}^{d}}
\tfrac{1}{L_i^2}$.
 \end{proposition}  
 The proof is given in the Appendix, Section~\ref{Apx:proofP:underAlpha}.

As a consequence it follows that the Poincar\'e-like constant~$\underline\alpha_{S,\ell}$
in~\eqref{und-alpha}, satisfies the limit
$\lim\limits_{S\to+\infty}\underline\alpha_{S,\ell}=0$, for~$\ell\in\{0,1\}$.
In these cases it may be necessary
to choose firstly~$S$ and subsequently~$\lambda$ (depending on~$S$) in Theorem~\ref{T:main};
see Remark~\ref{R:underAlpha}.

 \section{Numerical simulations}\label{S:simul}
 
Here we show the results of simulations illustrating the stabilizability result
stated in Main Result in the Introduction; see main Theorem~\ref{T:main}.
We consider the following scalar parabolic system as an academic model for
the error dynamics; see~\eqref{sys-z-o}.
\begin{align}
&\tfrac{\p}{\p t}z + (-\nu\Delta +\Id)z+a z +b\cdot\nabla z-\norm{z}{\bbR}^{3}z
+(\tfrac{\p}{\p x_1}z-2\tfrac{\p}{\p x_2}z)z\notag\\
&\hspace*{5em}=-\lambda A^{-1}P_{\clW_S}^{\widetilde\clW_S^\perp}A^\ell
P_{\widetilde\clW_S}^{\clW_S^\perp}\bfZ^{W_S}\clZ z,\qquad
\notag\\
 &\hspace*{0em} z(0)=z_0\in H^1(\Omega),\qquad \tfrac{\p }{\p \bfn}z\rest{\p\Omega}=0\notag
\end{align}
evolving in~$V=H^1(\Omega)$ under Neumann boundary conditions,
where~$\Omega=(0,1)\times(0,1)\in\bbR^2$ is the unit square.
As parameters we set
\begin{align}\notag
 \ell=2,\quad \nu=0.1,\quad a=-2+x_1 -\norm{\sin(t+x_1)}{\bbR},\quad b=\begin{bmatrix}
                                                         x_1+x_2\\ \cos(t)x_1x_2
                                                        \end{bmatrix}.
\end{align}
As sensors we take indicator functions~$1_{\omega_i}=1_{\omega_i}(x)$ of rectangular
subdomains~$\omega_i\subset\Omega$
as in Figure~\ref{fig.suppsensors}. Hence,
the output~$\clZ y(t)\in\bbR^{S_\sigma}$ consists of the ``averages'' of the solution
over the same subdomains,
\[
(\clZ y(t))_i=(1_{\omega_{i}}, y(t))_{L^2(\Omega)}=\int_{\omega_{i}}y(t),\,\rmd\Omega
\]
and the output error is~$\clZ z(t)=\clZ \widehat y(t)-\clZ y(t)\in\bbR^{S_\sigma}$,
\[
(\clZ z(t))_i=(1_{\omega_{i}},z(t))_{L^2(\Omega)}
=(1_{\omega_{i}},\widehat y(t))_{L^2(\Omega)}-(1_{\omega_{i}}, y(t))_{L^2(\Omega)}.
\]

Finally, we set the normalized initial condition as
\begin{equation}\notag
 z_0=\frac{2-x_1x_2}{\norm{2-x_1x_2}{V}}\in V.
\end{equation}

The number of sensors~${S_\sigma}$ and the parameter~$\lambda$, which we know should
be both large enough (cf.~Main Result in Introduction),
will be set later on.

As auxiliary functions we will take the functions in~\eqref{choice-tilW-adhoc}.

\subsection{Discretization}
The following simulations have been performed in {\sc matlab} and correspond to a piecewise
linear (hat functions based)
finite element discretization of
the equation in the spatial variable. Subsequently, for the time variable and for
discrete time instants~$t_j=kj$, $j\in\bbN$
and with time step~$k>0$,
we use the standard linear approximation for the time derivative,
a Crank--Nicolson scheme 
to approximate the symmetric operator~$\clA=(-\nu\Delta +\Id+a\Id)$, and a Adams--Bashforth
scheme for the remaining terms~$\clR$,
that is, denoting~$t_j^*\coloneqq\frac{t_j+t_{j+1}}2$ we take $\tfrac{\p}{\p t}z(t_j^*)
\approx\frac{z(t_{j+1})-z(t_{j})}{k}$,
$\clA(t_j^*) z(t_j^*)\approx\frac{\clA(t_j) z(t_j)+\clA(t_{j+1}) z(t_{j+1})}2$ and
$\clR(t_j^*, z(t_j^*))\approx\frac{3\clR(t_j,z(t_j))-\clR(t_{j-1},z(t_{j-1}))}2$.

We will consider the cases where the number~$S_\sigma$ of sensors belongs to~$\{4,9,16\}$.
The corresponding triangulations
of the spatial domain, used in the simulations, are shown in Figure~\ref{Fig:Mesh_sensors_location}.

\begin{figure}[ht]
\centering
\subfigure
{\includegraphics[width=0.31\textwidth]{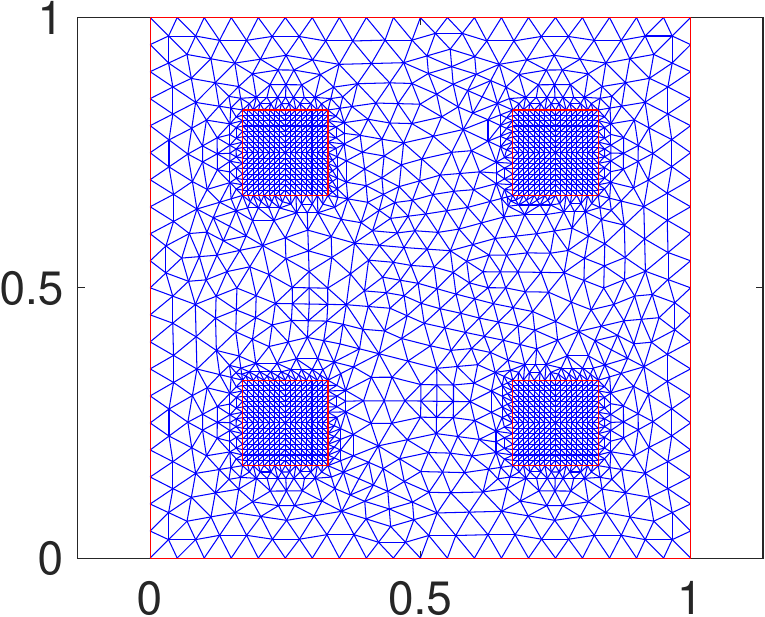}}
\quad
\subfigure
{\includegraphics[width=0.31\textwidth]{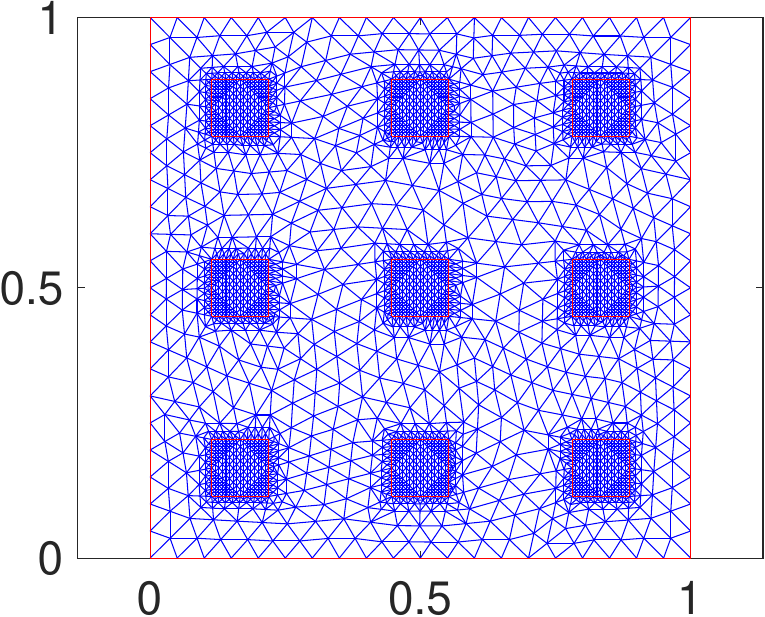}}
\quad
\subfigure
{\includegraphics[width=0.31\textwidth]{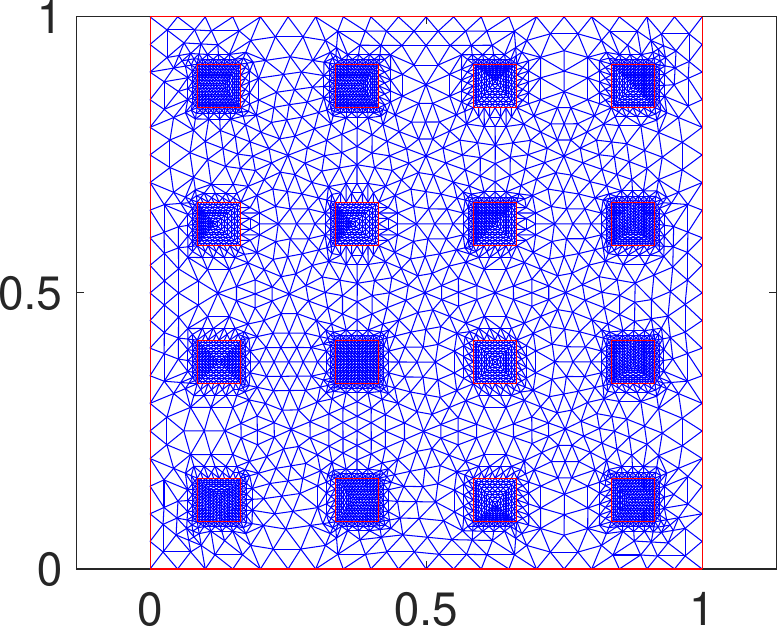}}
\caption{Locations of sensors in cases $S_\sigma\in\{4,9,16\}$.}
\label{Fig:Mesh_sensors_location}
\end{figure}

In the figures below, the symbol ``${\rm npts}\Omega$'' stands
 for the number of mesh points in the triangulation of the spatial
 domain~$\Omega$, the symbol ``$k$''
 stands for the time step, and~$T$ stands for the end point of the time
 interval~$[0,T]$ where the simulations have been run in. 
If the plots in the figures  do not include the entire
interval~$[0,T]$, then it means that the norm~$\norm{z}{V}$ of the error blows
up at time~$t_{\rm bu}<T$ near the last plotted time
instant.

\subsection{Necessity of large~$S_\sigma$ for error stability}
In Figure~\ref{Fig:zo4Sens} we see that the free error dynamics (i.e., under no output injection)
is blowing up in finite time, namely, at time~$t_{\rm bu}\approx 0.11$.
The simulations correspond to the mesh corresponding to~$4$ sensors.
Also in Figure~\ref{Fig:zo4Sens} we see that the error norm, for the output
injection corresponding to the case of~$4$ sensors,
blows at time~$t_{\rm bu}\approx 9.5$ for $\lambda=0.004$,
while it blows up at time~$t_{\rm bu}\approx 10.5$ for larger~$\lambda\in\{0.02,0.1,0.5\}$.
That is, the blow up time increases due to the
the output injection, but such injection is not able to stabilize the error dynamics.
In particular, we see that the blow up time seems to converge to a value in the
interval~$[10.5,11]$ as $\lambda$ increases.
Therefore, we can conclude that~$4$ sensors are likely not able to stabilize the
estimation error norm. This confirms the statement
of Theorem~\ref{T:main} on the necessity of a large enough number of sensors.
\begin{figure}[ht]
\centering
\subfigure[Free error dynamics.]
{\includegraphics[width=0.45\textwidth]{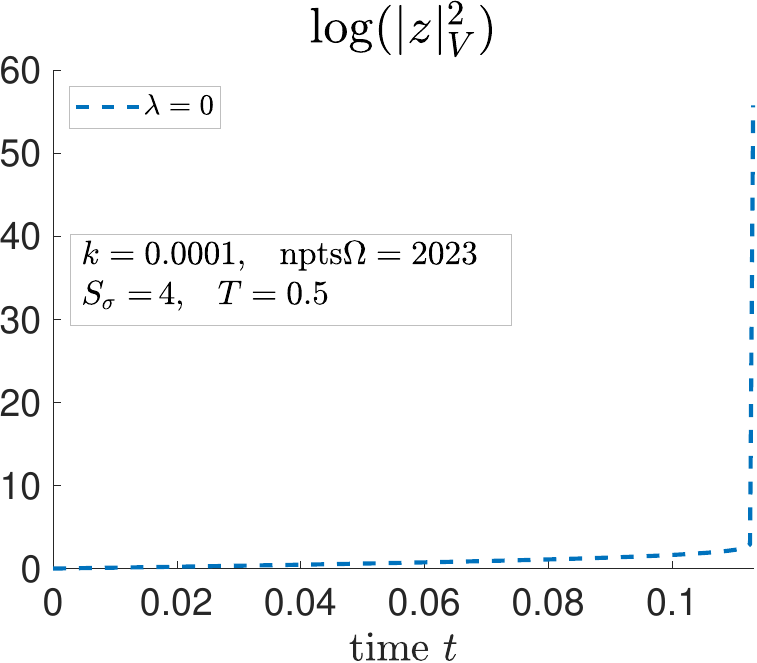}}
\quad
\subfigure[The case of~$4$ sensors.]
 {\includegraphics[width=0.45\textwidth]{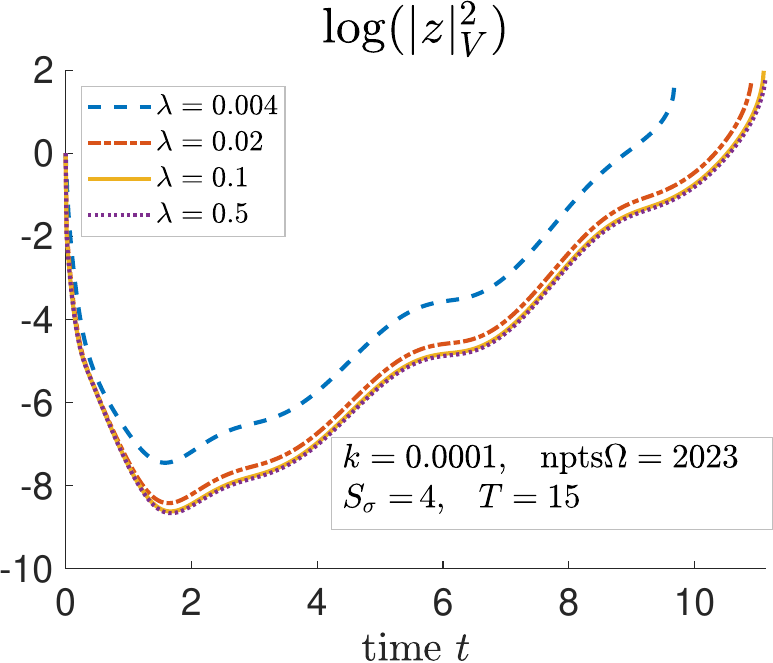}}
\caption{The free dynamics and the case of~$4$ sensors.}
\label{Fig:zo4Sens}
\end{figure}

In Figure~\ref{Fig:zo9and16Sens} we see that~$9$ and~$16$ sensors are able
to stabilize the error norm for the considered values of~$\lambda$.
We also see that, for a fixed number~$S_\sigma$ of sensors, the exponential
stabilization rate increases with~$\lambda$,
and converges to a bounded value.
This means that if we want to achieve a larger stability rate it is not enough to increase~$\lambda$;
we will need to increase also the number of sensors as stated in
Theorem~\ref{T:main}. This is confirmed in Figure~\ref{Fig:zo9and16Sens}
where we see that with~$16$ sensors we obtain a faster decreasing of
the error norm~$\norm{z}{V}$, namely, for~$\lambda=0.02$
we find the rates~$\mu\approx5=\frac12\frac{150}{15}$
for~$S_\sigma=9$, and~$\mu\approx11.5\approx\frac12\frac{350}{15}$
for~$S_\sigma=16$.

\begin{figure}[ht]
\centering
\subfigure[The case of~$9$ sensors.]
{\includegraphics[width=0.45\textwidth]{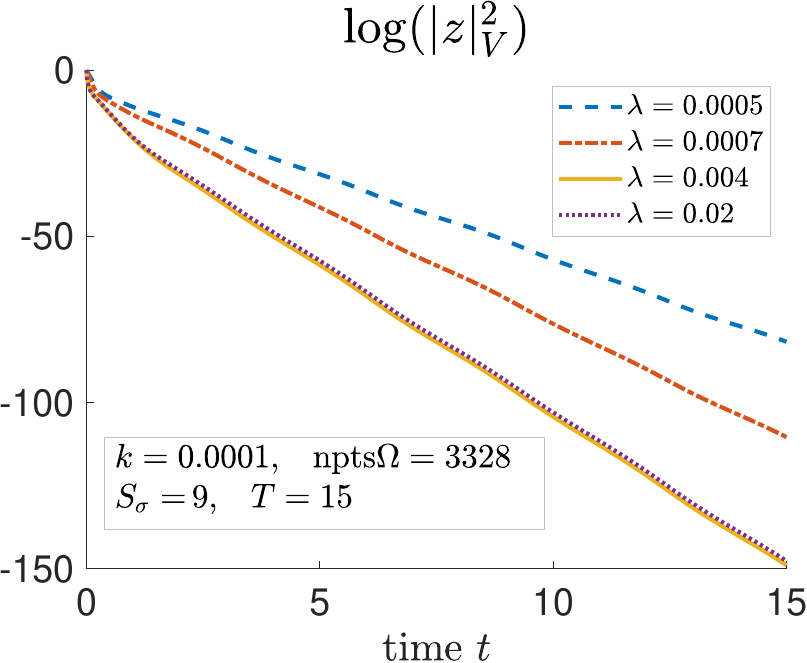}}
\quad
\centering
\subfigure[The case of~$16$ sensors.]
{\includegraphics[width=0.45\textwidth]{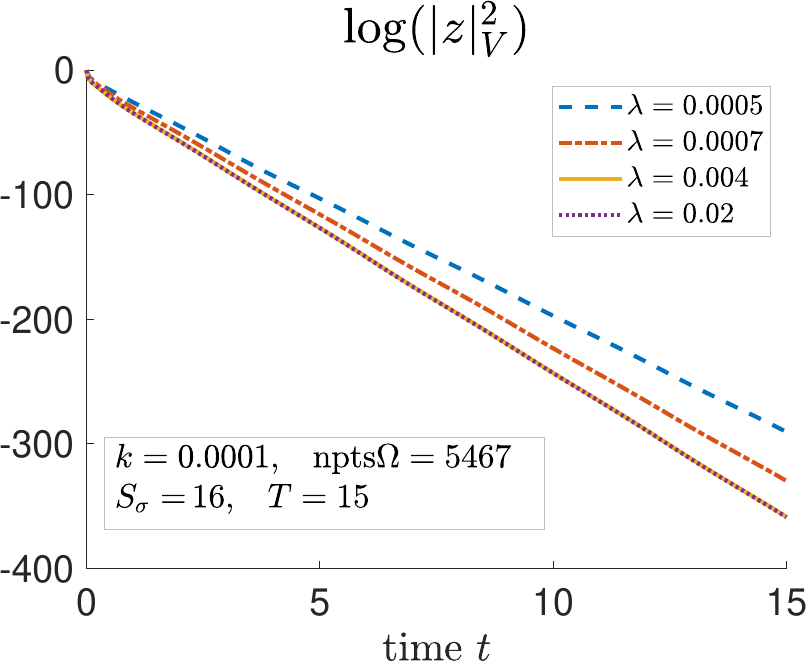}}
\caption{The case of~$9$ and~$16$ sensors.}
\label{Fig:zo9and16Sens}
\end{figure}

\subsection{Necessity of large~$\lambda$ for error stability}
We know that the free error dynamics, with~$\lambda=0$, is not stable. Here we
show that~$\lambda>0$ must be large enough in order to achieve stability of the error dynamics.
Indeed in Figure~\ref{Fig:zo9and16Sens_slam} we see that, for small~$\lambda$, neither~$9$ nor~$16$ sensors are
able to stabilize
the error dynamics.
\begin{figure}[ht]
\centering
\subfigure[With~$9$ sensors.]
{\includegraphics[width=0.45\textwidth]{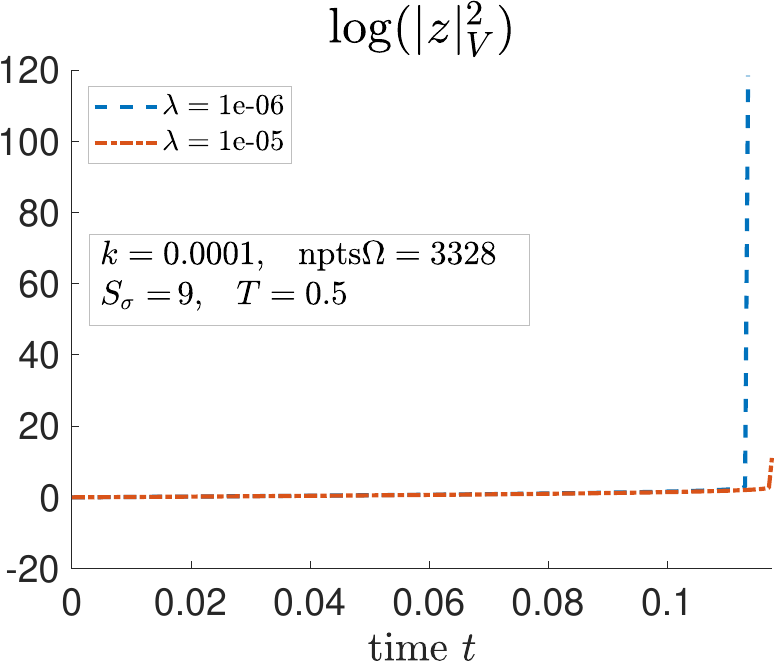}}
\quad
\subfigure[With~$16$ sensors.]
 {\includegraphics[width=0.45\textwidth]{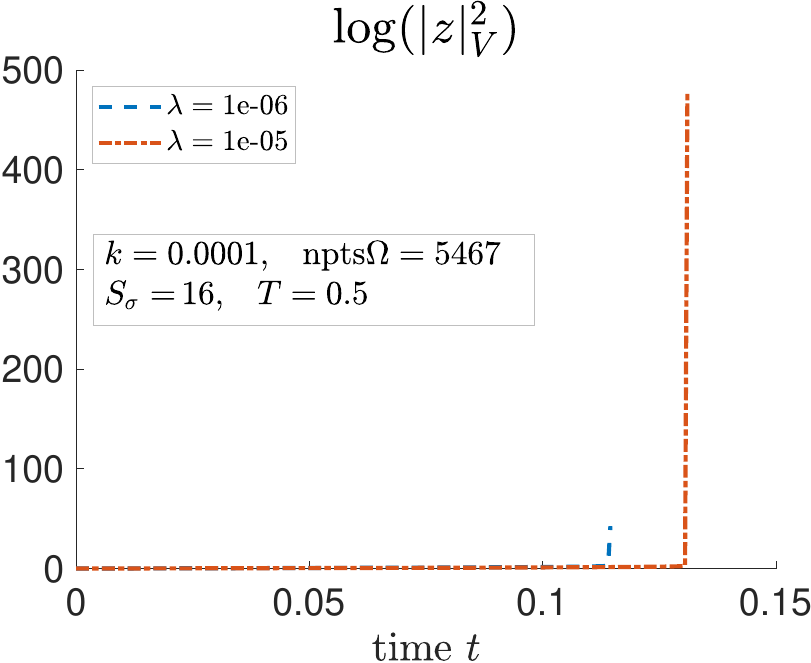}}
\caption{The case of small~$\lambda$.}
\label{Fig:zo9and16Sens_slam}
\end{figure}

The above results show that both~$S_\sigma$ and~$\lambda$ must be taken large enough
to achieve the stability of the error dynamics, which agree with the theoretical results.

\section{Final remarks}\label{S:finalremks}

Though the ``best'' choice of all the parameters involved in the output
injection operator~$\fkI_{S}^{[\lambda,\ell]}$
is not the main focus of this paper. Such choice is (or, may be) important for applications
(e.g., numerical simulations). Here, we just discuss briefly
the semiglobal estimatability result presented in this manuscript,
from practicability viewpoints, and mention related problems which could
be the subject of further investigation.

\subsection{On the choice of~$S$ and~$\lambda$}

Let us fix~$(\mu,R)$. For a given~$\ell$, the estimatability property of the output
injection operator~$\fkI_{S}^{[\lambda,\ell]}$, in Theorem~\ref{T:main}, depends
on the desired exponential decreasing rate~$\mu$ and on the upper bound~$R$
for the norm of the initial error,
simply because the pair~$(S,\lambda)$ depends on, and ``increases with'',~$(\mu,R)$.
In practice the initial error~$z_0$ is unknown for us,
thus we will not be able to surely choose an appropriate~$\fkI_{S}^{[\lambda,\ell]}$
stabilizing the nonlinear error dynamics.
However, on the other hand, we are sure that it is enough
to increase both~$S$ and~$\lambda$ to find a stabilizing~$\fkI_{S}^{[\lambda,\ell]}$.
Furthermore, the fact that the transient bound~$\varrho>1$ will get smaller,
for large~$S$ and~$\lambda$, can be used in applications
to decide whether we should (still)
increase~$S$ and~$\lambda$. Namely, we increase~$S$ and/or~$\lambda$ if (e.g., in simulations)
we realize that the error norm is not starting decreasing after a suitable amount of time. 

If we knew that the transient bound is~$\varrho=1$ then we would know that for large enough~$S$ and~$\lambda$,
the error norm must be strictly decreasing.
Hence we
would increase~$S$ and~$\lambda$ if we realize that the norm is not strictly decreasing.

\subsection{Strictly decreasing estimate error norm}

Let us fix again~$(\mu,R)$. We shall see now that we can achieve the optimal
transient bound constant~$\varrho=1$ when we have~$\zeta_{2i}=0$
for all~$i\in\{1,2,\dots,n\}$ in Assumption~\ref{A:N}. This case holds for
parabolic equations~\eqref{sys-y-o-parab} in the
cases~$d\in\{1,2\}$ with~$r>1$ and~$s\ge1$, and also in the case~$d=3$ with~$r\in(1,3]$
and~$\widetilde b=0$. These facts have been proven
in~\cite[Sects.~5.2 and~5.3]{Rod20-eect}. In particular, in the cases~$d\in\{1,2\}$ we
can take arbitrary large exponents/degrees
for the nonlinearities, ~$r\in(1,+\infty)$ and~$s\in[1,+\infty)$, in~\eqref{sys-y-o-parab}.
Note that our simulations have been performed
in a two-dimensional domain, and in Figure~\ref{Fig:zo9and16Sens} the estimation
error norm is strictly decreasing.

 To show that we can take~$\varrho=1$ if~$\zeta_{2i}=0$, $i\in\{1,2,\dots,n\}$,
 we observe that in such case we have~$\chi_2=0$, due to~\eqref{chi12}. Thus we obtain
$
 \Psi(y)=2\left(1+\norm{y}{V}^{\chi_1} \right) 
$
in~\eqref{thetas-est3}, with~$\norm{h(y)}{L^{\infty}(\bbR_0,\bbR)}\le\widetilde C$ and
\begin{equation}\notag
\tfrac{\ed}{\ed t}\norm{z}{V}^2
  \le -\Bigl(\overline\mu
  -\norm{h(y)}{}\left(1+\norm{z}{V}^{\chi_3}\right)\Bigr)\norm{z}{V}^2
  \le -\Bigl(\breve C_1
  -\breve C_2\norm{z}{V}^{\chi_3}\Bigr)\norm{z}{V}^2
\end{equation}
in~\eqref{dtz5}, with~$\breve C_1\coloneqq\overline\mu
  -\widetilde C$ and~$\breve C_2=\widetilde C$. Recalling~\eqref{choice-Slambda} and the fact that~$\overline\mu$
  can be made arbitrarily large by choosing both~$S$ and~$\lambda$ large enough, it is clear that
  for any given~$R>0$ and~$\mu>0$,
  we can set both~$S$ and~$\lambda$ large enough
  so that~$\widehat\mu\coloneqq\breve C_1 -\breve C_2R^\frac{\chi_3}{2}\ge\mu$. 
  Then by Proposition~4.3 in~\cite{Rod20-eect}, we find the following estimate,
  with transient bound~$\varrho=1$, 
\[
\norm{z(t)}{V}^2\le \ex^{-\widehat\mu(t-s)}\norm{z(s)}{V}^2\le\ex^{-\mu(t-s)}\norm{z(s)}{V}^2,
\qquad\mbox{if}\quad\norm{z(0)}{V}^2\le R.
\]

\subsection{On the choice of the set of auxiliary functions~$\widetilde W_{S}$ and~$\ell$}

The choice of the auxiliary set~$\widetilde W_S\subset\rmD(A)$ is at our disposal.
In Section~\ref{S:parabolic-OK} we have suggested three possible choices, namely,
those in~\eqref{choice-tilW}. In Section~\ref{S:simul}
we have taken only the choice in~\eqref{choice-tilW-adhoc}. We did not compare with
other possible choices because
the ``optimal'' choice for~$\widetilde W_S$ is not the main goal of this paper.
However, we must say that, though
the operator norm of the 
oblique projection~$P_{\widetilde\clW_{S}}^{\clW_{S}^\perp}$ does not play any crucial role in
the estimatability result, it plays a role in the norm of the infection operator, as we have seen in
Section~\ref{sS:boundInj-OPN}. A large operator norm of the 
oblique projection can influence negatively
the practicability of the observer in applications (e.g., leading to the need of taking
a very small time step~$k$ in simulations), as shown/discussed through numerical
results presented in~\cite[Sect.~6]{Rod20-sicon}
(in there, for choices as spans of eigenfunctions, cf.~\eqref{choice-tilW-eigs}).
By this reason, it could be interesting to investigate the performance of the feedback
for different choices of~$\widetilde W_S$ (e.g., those in~\eqref{choice-tilW}), or even
try to define and investigate the ``optimal choice''. 

In our simulations we have taken only the border case~$\ell=2$ for the power~$\ell$ of the diffusion
taken in the injection operator~$\fkI_S^{[\lambda,\ell]}$.
Another point that could be investigated is the performance of the observer for different values of~$\ell$.

\subsection{On the time step}
In Table~\ref{Tab:nInj0} we see that the $H$-norm of the output injection at initial time,
for some pairs~$(S_\sigma,\lambda)$.
This is the reason we took a small time step as~$k=10^{-4}$. Note that such norm increases
with $\lambda$, so for larger~$\lambda$ we
may need to take a smaller time step to capture (or, accurately approximate) the effect of
the output injection on the dynamics. A very small time
step may be impracticable for real world applications, thus it could be interesting to
investigate, in a future work, whether an appropriate choice of
~$\widetilde W_{S}$ and/or $\ell$ allows us to take larger~$k$.
\begin{table}[ht]
\centering
\begin{tabular}{l|lll}
\hline
 $(S_\sigma,\lambda)$&(4,\,0.5)&(9,\,0.02)&(16,\,0.02)\\\hline
 $\norm{\fkI_S^{[\lambda,\ell]}z_0}{H}$&3537.9599&747.3875&2594.0443\\ 
 \hline \multicolumn{4}{c}{\vspace*{-.5em}}
\end{tabular}
\caption{Norm of output injection at initial time.}\label{Tab:nInj0}
\end{table}

\appendix
\section*{Appendix}
\setcounter{section}{1}
\setcounter{theorem}{0} \setcounter{equation}{0}
\numberwithin{equation}{section}

\subsection{Proof of Proposition~\ref{P:fkN}}\label{sS:proofP:fkN}

 With~$\widehat y_1\coloneqq y+z_1$ and~$\widehat  y_2\coloneqq y+z_2$, we write
 \begin{align}
 \fkN_y(t,z_1)-\fkN_y(t,z_2)&=\clN(t,\widehat y_1)-\clN(t, y)-(\clN(t,\widehat y_2)-\clN(t, y))\notag\\
 &=\clN(t,\widehat y_1)-\clN(t,\widehat y_2)
 =\clN(t,y+z_1)-\clN(t,y+z_2),\notag
 \end{align}
which leads us to~$\widehat y_1-\widehat y_2=z_1-z_2\eqqcolon d$ and, using Lemma~\ref{L:NN},
\begin{align}
 &2\Bigl( \fkN_y(t,z_1)-\fkN_y(t,z_2),A(z_1-z_2)\Bigr)_{H}
 =2\Bigl( \clN(t,\widehat y_1)-\clN(t,\widehat y_2),A(\widehat y_1-\widehat y_2))\Bigr)_{H}\notag\\
 &\hspace*{0em}\le \widehat\gamma_0 \norm{d}{\rmD(A)}^{2}
  +\!\left(\!1+\overline\gamma_0^{-\frac{1+\|\delta_2\|}{1-\|\delta_2\|} }\right)
 \!\overline C_{\clN1}\sum\limits_{j=1}^n\norm{d}{V}^{\frac{2\delta_{1j}}{1-\delta_{2j}}}\sum\limits_{k=1}^2
 \norm{\widehat y_k}{V}^\frac{2\zeta_{1j}}{1-\delta_{2j}}
 \norm{\widehat y_k}{\rmD(A)}^\frac{2\zeta_{2j}}{1-\delta_{2j}}\notag
  .
\end{align}
Therefore~\eqref{fkNyAy} follows with~$\widetilde C_{\fkN1}=\overline C_{\clN1}$.

By setting~$z_2=0$ in~\eqref{fkNyAy}, we obtain for each~$\widetilde\gamma_0>0$,
\begin{align}
  &2\Bigl( \fkN_y(t,z_1),Az_1\Bigr)_{H}\label{aux_fkN}\\
 &\le \widetilde\gamma_0 \norm{z_1}{\rmD(A)}^{2}
   +\!\left(\!1+\widetilde\gamma_0^{-\frac{1+\|\delta_2\|}{1-\|\delta_2\|} }\right)
 \!\widetilde C_{\clN1}
 \sum\limits_{j=1}^n\norm{z_1}{V}^{\frac{2\delta_{1j}}{1-\delta_{2j}}}
 \sum\limits_{l=0}^1\norm{y+lz_1}{V}^\frac{2\zeta_{1j}}{1-\delta_{2j}}
 \norm{y+lz_1}{\rmD(A)}^\frac{2\zeta_{2j}}{1-\delta_{2j}}.\notag
\end{align}
Now, for simplicity we fix~$j$, and set
\begin{align}\notag
& r=r_j\coloneqq\tfrac{2\delta_{1j}}{1-\delta_{2j}}\ge2,\quad p
=p_j\coloneqq\tfrac{2\zeta_{1j}}{1-\delta_{2j}}\quad\mbox{and}\quad
q=q_j\coloneqq\tfrac{2\zeta_{2j}}{1-\delta_{2j}}<2.
\end{align}
Note that~$p\ge0$, $r\ge2$, and~$q\in[0,2)$ due to the relations~$\delta_{1j}+\delta_{2j}\ge1$ and~$\zeta_{2j}+\delta_{2j}<1$, in Assumption~\ref{A:N}.

We consider first the case~$q\ne0$. By the triangle inequality and~\cite[Prop.~2.6]{PhanRod17}, we obtain
\begin{align}
 \Upsilon_j&\coloneqq\norm{z_1}{V}^r\sum\limits_{l=0}^1\norm{y+lz_1}{V}^p
 \norm{y+lz_1}{\rmD(A)}^q\notag\\
 &\hspace*{0em}=\norm{z_1}{V}^r\norm{y}{V}^p \norm{y}{\rmD(A)}^q
 +\norm{z_1}{V}^r\norm{y+z_1}{V}^p \norm{y+z_1}{\rmD(A)}^q\notag\\
 &\hspace*{0em}\le\norm{z_1}{V}^r\norm{y}{V}^p \norm{y}{\rmD(A)}^q\notag\\
 &\hspace*{1em}+\norm{z_1}{V}^{r}(1+2^{p-1})(1+2^{q-1})\left(\norm{y}{V}^{p}+\norm{z_1}{V}^{p}\right)
 \left(\norm{y}{\rmD(A)}^q+\norm{z_1}{\rmD(A)}^q\right).\label{Y1}
 \end{align}
 Setting~$D_{p,q}\coloneqq(1+2^{p-1})(1+2^{q-1})$ and using the Young
 inequality, the last term satisfies for each~$\gamma_2>0$,
 \begin{align}
 D_{p,q}^{-1}\clT_j&\coloneqq \norm{z_1}{V}^{r}\left(\norm{y}{V}^{p}+\norm{z_1}{V}^{p}\right)
 \left(\norm{y}{\rmD(A)}^q+\norm{z_1}{\rmD(A)}^q\right)\notag\\
 &\le\norm{z_1}{V}^{r}\left(\norm{y}{V}^{p}+\norm{z_1}{V}^{p}\right)\norm{z_1}{\rmD(A)}^q
 +\norm{z_1}{V}^{r}\left(\norm{y}{V}^{p}+\norm{z_1}{V}^{p}\right)\norm{y}{\rmD(A)}^q
 \notag\\
 &\hspace*{0em}\le \gamma_2^\frac{2}{q}\norm{z_1}{\rmD(A)}^2
 + \gamma_2^{-(1-\frac{q}{2})^{-1}}\left(\norm{z_1}{V}^{r}
 \left(\norm{y}{V}^{p}+\norm{z_1}{V}^{p}\right)\right)^{(1-\frac{q}{2})^{-1}}\notag\\
 &\hspace*{1em}+\norm{z_1}{V}^{r}\left(\norm{y}{V}^{p}+\norm{z_1}{V}^{p}\right)
 \left(1+\norm{y}{\rmD(A)}^{\dnorm{\frac{2\zeta_{2}}{1-\delta_{2}}}{}}\right)
 ,\notag
\end{align}
which implies, since~$1-\frac{q}{2}=\frac{2-q}{2}$,
\begin{align}
 D_{p,q}^{-1}\clT_j
 &\hspace*{0em}\le \gamma_2^\frac{2}{q}\norm{z_1}{\rmD(A)}^2
 + \gamma_2^{-\frac{2}{2-q}}\norm{z_1}{V}^{\frac{2r}{2-q}}\left(\norm{y}{V}^{p}
 +\norm{z_1}{V}^{p}\right)^{\frac{2}{2-q}}\notag\\
 &\hspace*{1em}+\norm{z_1}{V}^{r}\left(\norm{y}{V}^{p}+\norm{z_1}{V}^{p}\right)
 \left(1+\norm{y}{\rmD(A)}^{\dnorm{\frac{2\zeta_{2}}{1-\delta_{2}}}{}}\right)\notag\\
 &\hspace*{0em}\le \gamma_2^\frac{2}{q}\norm{z_1}{\rmD(A)}^2
 + \gamma_2^{-\frac{2}{2-q}}(1+2^{\frac{2}{2-q}-1})\norm{z_1}{V}^{\frac{2r}{2-q}}
 \left(\norm{y}{V}^{\frac{2p}{2-q}}+\norm{z_1}{V}^{\frac{2p}{2-q}}\right)
 \notag\\
 &\hspace*{1em}+\norm{z_1}{V}^{r}\left(\norm{y}{V}^{p}+\norm{z_1}{V}^{p}\right)
 \left(1+\norm{y}{\rmD(A)}^{\dnorm{\frac{2\zeta_{2}}{1-\delta_{2}}}{}}\right)\notag\\
 &\hspace*{0em}\le \gamma_2^\frac{2}{q}\norm{z_1}{\rmD(A)}^2\notag\\
 &\hspace*{1em}+ D_{q,\gamma_2}\left(\norm{z_1}{V}^{\frac{2r}{2-q}}
 \norm{y}{V}^{\frac{2p}{2-q}}+\norm{z_1}{V}^{\frac{2(r+p)}{2-q}}
 +\norm{z_1}{V}^{r}\norm{y}{V}^{p}+\norm{z_1}{V}^{r+p}\right)
 \left(1+\norm{y}{\rmD(A)}^{\dnorm{\frac{2\zeta_{2}}{1-\delta_{2}}}{}}\right)\notag
\end{align}
 with
 \begin{subequations}\label{Y3}
 \[
  D_{q,\gamma_2}\coloneqq 1+\gamma_2^{-\frac{2}{2-q}}(1+2^{\frac{2}{2-q}-1}),
 \]
 and then
  \begin{align}
 &D_{q,\gamma_2}^{-1}\left(D_{p,q}^{-1}\clT_j-\gamma_2^\frac{2}{q}\norm{z_1}{\rmD(A)}^2\right)
 \left(1+\norm{y}{\rmD(A)}^{\dnorm{\frac{2\zeta_{2}}{1-\delta_{2}}}{}}\right)^{-1}\notag\\
  &\hspace*{0em}\le \left(\norm{z_1}{V}^{\frac{2r}{2-q}-2}\norm{y}{V}^{\frac{2p}{2-q}}
  +\norm{z_1}{V}^{\frac{2(r+p)}{2-q}-2}
 +\norm{z_1}{V}^{r-2}\norm{y}{V}^{p}+\norm{z_1}{V}^{r+p-2}\right)\norm{z_1}{V}^2.
\end{align}
\end{subequations}

Observe that
\begin{subequations}\label{ineqs-rpq}
\begin{align}
 \tfrac{2p}{2-q}&=\tfrac{4\zeta_{1j}}{2-2\delta_{2j}-2\zeta_{2j}}
 \le\tfrac{2\dnorm{\zeta_{1}}{}}{1-\dnorm{\delta_{2}+\zeta_{2}}{}},\\
  \tfrac{2(r+q+p)-4}{2-q}
 &=\tfrac{4(\delta_{1j}+\zeta_{1j}+\zeta_{2j})-4(1-\delta_{2j})}{2-2\delta_{2j}-2\zeta_{2j}}
 \le\tfrac{2\dnorm{\delta_{1}+\delta_{2}+\zeta_{1}+\zeta_{2}}{}-2}{1-\dnorm{\delta_{2}+\zeta_{2}}{}},
\end{align}
and that~$\frac{2c}{2-q}\ge c\Longleftrightarrow 0\ge -qc$. Thus, we obtain
that $\frac{2c}{2-q}\ge c$ for all~$c\ge0$, and
\[
 \tfrac{2p}{2-q}\ge p\quad\mbox{and}\quad\tfrac{2(r+q+p)-4}{2-q}\ge\tfrac{2(r+p)-4}{2-q}\ge(r+p)-2\ge r-2\ge0,
\]
\end{subequations}
which together with~\eqref{Y3} give us
 \begin{align}
 &D_{q,\gamma_2}^{-1}\left(D_{p,q}^{-1}\clT_j-\gamma_2^\frac{2}{q}\norm{z_1}{\rmD(A)}^2\right)
 \left(1+\norm{y}{\rmD(A)}^{\dnorm{\frac{2\zeta_{2}}{1-\delta_{2}}}{}}\right)^{-1}\notag\\
  &\hspace*{0em}\le \left(\norm{z_1}{V}^{\frac{2(r+q)-4}{2-q}}+\norm{z_1}{V}^{\frac{2(r+q+p)-4}{2-q}}
 +\norm{z_1}{V}^{r-2}+\norm{z_1}{V}^{r+p-2}\right)
 \left(2+\norm{y}{V}^{\frac{2p}{2-q}}
 +\norm{y}{V}^{p}\right)
 \norm{z_1}{V}^2\notag\\
  &\hspace*{0em}\le \left(4+4\norm{z_1}{V}^{\tfrac{2\dnorm{\delta_{1}+\delta_{2}+\zeta_{1}+\zeta_{2}}{}-2}
  {1-\dnorm{\delta_{2}+\zeta_{2}}{}}}\right)
 \left(4+2\norm{y}{V}^{\tfrac{2\dnorm{\zeta_{1}}{}}{1-\dnorm{\delta_{2}+\zeta_{2}}{}}}
 \right)
 \norm{z_1}{V}^2,\notag
 \end{align}
 which implies
 \begin{subequations}\label{Y4}
 \begin{align}
 &\clT_j
   \le D_{p,q}\gamma_2^\frac{2}{q}\norm{z_1}{\rmD(A)}^2
  +  \widehat D\left(1+\norm{y}{V}^{\chi_1} \right)
  \left(1+\norm{y}{\rmD(A)}^{\chi_2}\right)
  \left(1+\norm{z_1}{V}^{\chi_3}\right)
  \norm{z_1}{V}^2,\\
  &D_{p,q}\coloneqq (1+2^{p-1})(1+2^{q-1}),\quad \widehat D\coloneqq 16D_{p,q}D_{q,\gamma_2},\\
  &\chi_1\coloneqq \tfrac{2\dnorm{\zeta_{1}}{}}{1-\dnorm{\delta_{2}+\zeta_{2}}{}}\ge0,
  \quad \chi_2\coloneqq{\dnorm{\tfrac{2\zeta_{2}}{1-\delta_{2}}}{}}\ge0 ,\quad
  \chi_3\coloneqq \tfrac{2\dnorm{\delta_{1}+\delta_{2}+\zeta_{1}+\zeta_{2}}{}-2}
  {1-\dnorm{\delta_{2}+\zeta_{2}}{}}\ge0.
 \end{align} 
 \end{subequations}

 For the first term on the right-hand side of~\eqref{Y1}, we also obtain
\begin{align}\label{Y5}
 \clF_j\coloneqq\norm{z_1}{V}^r\norm{y}{V}^p \norm{y}{\rmD(A)}^q
 \le\left(1+\norm{y}{V}^{\chi_1}\right) \left(1+\norm{y}{\rmD(A)}^{\chi_2}\right)
 (1+\norm{z_1}{V}^{\chi_3})\norm{z_1}{V}^2
\end{align}
because~$p\le\chi_1$, $0<q\le\chi_2$, and~$r-2\le\chi_3$. See~\eqref{ineqs-rpq}.

Therefore, by~\eqref{Y1}, \eqref{Y4}, and~\eqref{Y5}, it follows that for all~$\gamma_2\in(0,1]$,
 \begin{subequations}\label{Y6}
 \begin{align}
 \Upsilon_j&\le\clF_j+\clT_j\notag\\
 &\hspace*{0em} \le D_{p,q}\gamma_2^\frac{2}{q}\norm{z_1}{\rmD(A)}^2
  +  (1+\widehat D)\left(1+\norm{y}{V}^{\chi_1} \right)
  \left(1+\norm{y}{\rmD(A)}^{\chi_2}\right)
  \left(1+\norm{z_1}{V}^{\chi_3}\right)
  \norm{z_1}{V}^2,\notag\\
  &\hspace*{0em} \le D_{p,q}\gamma_2^{\frac{2}{\chi_2}}\norm{z_1}{\rmD(A)}^2
  +  (1+\widehat D)\left(1+\norm{y}{V}^{\chi_1} \right)
  \left(1+\norm{y}{\rmD(A)}^{\chi_2}\right)
  \left(1+\norm{z_1}{V}^{\chi_3}\right)
  \norm{z_1}{V}^2,\\
  \mbox{for }&q\ne0,\quad\mbox{with}\quad \gamma_2\le1.
 \end{align} 
 \end{subequations}
Note that
$
 \gamma_2^\frac{2}{q}\le\gamma_2^\frac{2}{\chi_2}$
because~$\gamma_2\le1$ and~$0<q\le\chi_2$.

Finally, we consider the case~$q=0$. We find
\begin{subequations}\label{Y7}
\begin{align}
 \Upsilon_j&\coloneqq\norm{z_1}{V}^r\sum\limits_{l=0}^1\norm{y+lz_1}{V}^p
 =\norm{z_1}{V}^r\norm{y}{V}^p
 +\norm{z_1}{V}^r\norm{y+z_1}{V}^p\notag\\
 &\hspace*{0em}\le
 \norm{z_1}{V}^{r}\left((2+2^{p-1})\norm{y}{V}^{p}+(1+2^{p-1})\norm{z_1}{V}^p\right)\notag\\
&\hspace*{0em}\le
(2+2^{p-1}) \left(1+\norm{y}{V}^{p}\right)\left(\norm{z_1}{V}^{r-2}
+\norm{z_1}{V}^{r+p-2}\right)\norm{z_1}{V}^2\notag\\
&\hspace*{0em}\le
\left(2+2^{p-1}\right) \left(2+\norm{y}{V}^{\chi_1}\right)\left(2+2\norm{z_1}{V}^{\chi_3}\right)\norm{z_1}{V}^2
\notag\\
&\hspace*{0em}\le8
\left(1+2^{p-1}\right) \left(1+\norm{y}{V}^{\chi_1}\right)
\left(1+\norm{z_1}{V}^{\chi_3}\right)\norm{z_1}{V}^2,\qquad q=0.
\\
\mbox{for }&q=0,\quad\mbox{with}\quad \gamma_2\le1,
  \end{align}
\end{subequations}
  with~$\chi_1$ and~$\chi_3$ as in~\eqref{Y4}, where we have used~\eqref{ineqs-rpq}.

Now we observe that
\begin{subequations}\label{Dpqgamma}
\begin{align}
(1+2^{p-1})&\le D_{p,q}= (1+2^{p-1})(1+2^{q-1})\notag\\
&\le\left(1+2^{\frac{ \dnorm{2\zeta_1+\delta_2}{} -1 }{1-\dnorm{\delta_2}{}  }}\right)
\left(1+2^{\frac{ \dnorm{2\zeta_2+\delta_2}{} -1 }{1-\dnorm{\delta_2}{}  }}\right)\eqqcolon\widetilde D_{1},\\
\widehat D&= 16D_{p,q}D_{q,\gamma_2}\le 16\widetilde D_{1}(1+2^{\frac{2q}{2-q}})
\left(1+\gamma_2^{-\frac{2}{2-q  }}\right)\notag\\
&\le16\widetilde D_{1}\left(2+2^{\frac{2\black\dnorm{\zeta_2}{}}{1-\dnorm{\delta_2+\zeta_2}{}}}\right)
\left(2+\gamma_2^{-\frac{1}{1-\dnorm{\delta_2+\zeta_2}{}  }}\right)\notag\\
&\le 32\widetilde D_{1}\left(2+2^{\frac{2\black\dnorm{\zeta_2}{}}{1-\dnorm{\delta_2+\zeta_2}{}}}\right)
\left(1+\gamma_2^{-\frac{1}{1-\dnorm{\delta_2+\zeta_2}{}  }}\right)
=\widetilde D_{2}\left(1+\gamma_2^{-\chi_4 }\right),\\
\widetilde D_{2}&\coloneqq32\widetilde D_{1}
\left(2+2^{\frac{2\black\dnorm{\zeta_2}{}}{1-\dnorm{\delta_2+\zeta_2}{}}}\right),\\
\chi_4&\coloneqq\tfrac{1}{1-\dnorm{\delta_2+\zeta_2}{}  }>1.
\end{align}
\end{subequations}
Further, from~$\widetilde D_{2}>64\widetilde D_{1}
\ge8\left(1+2^{p-1}\right)$
and from from~\eqref{Y6}, \eqref{Y7}, and~\eqref{Dpqgamma}, we conclude that  for both cases, $q>0$ and~$q=0$, we have
 \begin{subequations}\label{Y8}
 \begin{align}
 &\Upsilon_j\le \vartheta\norm{z_1}{\rmD(A)}^2
 +  \widetilde D_2\left(1+\gamma_2^{-\chi_4}\right)\left(1+\norm{y}{V}^{\chi_1} \right)
  \left(1+\norm{y}{\rmD(A)}^{\chi_2}\right)
  \left(1+\norm{z_1}{V}^{\chi_3}\right)
  \norm{z_1}{V}^2,\\
  &\qquad\mbox{for all}\quad \gamma_2\in(0,1],\quad\mbox{with}\quad\vartheta
  \coloneqq\begin{cases}
   \widetilde D_{1}\gamma_2^{\frac{2}{\chi_2}},&\mbox{ for }\chi_2>0,\\
   0,&\mbox{ for }\chi_2=0,
 \end{cases}\\
  &\qquad\mbox{and}\quad\widetilde D_{1}=\ovlineC{\dnorm{\zeta_{1}}{},
  \dnorm{\zeta_{2}}{},\frac{1}{1-\dnorm{\delta_{2}+\zeta_{2}}{}}}
  ,\qquad
  \widetilde D_{2}=\ovlineC{\dnorm{\zeta_{1}}{},\dnorm{\zeta_{2}}{},
  \frac{1}{1-\dnorm{\delta_{2}+\zeta_{2}}{}}}.
 \end{align} 
 \end{subequations}

Now, from~\eqref{aux_fkN} and~\eqref{Y8}, for all~$\widetilde\gamma_0>0$ and~$\gamma_2\in(0,1]$, and with
\begin{equation}\notag
 \chi_5\coloneqq\tfrac{1+\|\delta_2\|}{1-\|\delta_2\|}>1,
\end{equation}
we derive that
\begin{align}
  &2\Bigl( \fkN_y(t,z_1),Az_1\Bigr)_{H}\le \widetilde\gamma_0 \norm{z_1}{\rmD(A)}^{2}
  +\left(1+\widetilde\gamma_0^{-\chi_5 }\right)
 \widetilde C_{\clN1}
 \sum\limits_{j=1}^n\Upsilon_j,\notag\\
  & \le \left(\widetilde\gamma_0+n\vartheta\left(1+\widetilde\gamma_0^{-\chi_5 }\right)
 \widetilde C_{\clN1}\right) \norm{z_1}{\rmD(A)}^{2}\label{NNchis}\\
 &\hspace*{.5em}  +n\widetilde D_{2}\left(1+\gamma_2^{-\chi_4}\right)
 \left(1+\widetilde\gamma_0^{-\chi_5 }\right)\widetilde C_{\clN1}
   \left(1+\norm{y}{V}^{\chi_1} \right)
  \left(1+\norm{y}{\rmD(A)}^{\chi_2}\right)
  \left(1+\norm{z_1}{V}^{\chi_3}\right)
  \norm{z_1}{V}^2.\notag
\end{align}

For an arbitrary~$\widehat\gamma_0>0$, we can choose
\[
 \gamma_2
 =\left(\tfrac{\widehat\gamma_0}{(n+1)\widetilde D_{1}\widetilde C_{\clN1}\left(1+\widetilde\gamma_0^{-\chi_5 }\right)
 +\widehat\gamma_0}\right)^{\frac{\chi_2}2}\le1,
\]
and~$\widetilde\gamma_0=\frac{\widehat\gamma_0}{n+1}$. Note that, in particular,
\[
\widetilde D_{1}\gamma_2^{\frac{2}{\chi_2}}\left(1+\widetilde\gamma_0^{-\chi_5 }\right)
 \widetilde C_{\clN1}< \gamma_2^{\frac{2}{\chi_2}}\left(\widetilde D_{1}
 \widetilde C_{\clN1} \left(1+\widetilde\gamma_0^{-\chi_5 }\right)
 +\tfrac{\widehat\gamma_0}{n+1}\right)=\tfrac{\widehat\gamma_0}{n+1},\quad\mbox{for}\quad \chi_2>0.
\]
and thus for the coefficient of~$\norm{z_1}{\rmD(A)}^{2}$ in~\eqref{NNchis}, we find
\begin{align}\label{CoefDA}
\begin{cases}
 \widetilde\gamma_0+n\vartheta\left(1+\widetilde\gamma_0^{-\chi_5 }\right)
 \widetilde C_{\clN1} < \tfrac{\widehat\gamma_0}{n+1}+n\tfrac{\widehat\gamma_0}{n+1}=\widehat\gamma_0,
&\qquad\mbox{if } \chi_2>0;\\
\widetilde\gamma_0+n\vartheta\left(1+\widetilde\gamma_0^{-\chi_5 }\right)
 \widetilde C_{\clN1} =\tfrac{\widehat\gamma_0}{n+1}< \widehat\gamma_0,
&\qquad\mbox{if } \chi_2=0.
\end{cases}
\end{align}

Observe, next, that
 \begin{align}\label{1-gamma}
&1+\widetilde\gamma_0^{-\chi_5 }=1+(n+1)^{\chi_5 }\widehat\gamma_0^{-\chi_5 },
\end{align}
and
\begin{align}
1+\gamma_2^{-\chi_4}&=2,\quad\mbox{if}\quad\chi_2=0,\notag\\
 1+\gamma_2^{-\chi_4}&\le
 1+\left((n+1)\widetilde D_{1}\widetilde C_{\clN1}\left(1+\widetilde\gamma_0^{-\chi_5 }\right)
 +\widehat\gamma_0\right)^{\frac{\chi_2\chi_4}2}\widehat\gamma_0^{-\frac{\chi_2\chi_4}2}\notag\\
 &\le
 1+(1+2^{\frac{\chi_2\chi_4}2-1})\left(\left((n+1)\widetilde D_{1}
 \widetilde C_{\clN1}\left(1+\widetilde\gamma_0^{-\chi_5 }\right)\right)^{\frac{\chi_2\chi_4}2}
 +\widehat\gamma_0^{\frac{\chi_2\chi_4}2}\right)\widehat\gamma_0^{-\frac{\chi_2\chi_4}2}\notag\\
 &\le
 \widehat C_1+\widehat C_2\left(1+\widetilde\gamma_0^{-\chi_5 }\right)^{\frac{\chi_2\chi_4}2}
 \widehat\gamma_0^{-\frac{\chi_2\chi_4}2},\quad\mbox{if}\quad\chi_2>0;
 \notag
 \end{align}
 with
 \begin{align}
 &\widehat C_1\coloneqq 1+(1+2^{\frac{\chi_2\chi_4}2-1})
 =\ovlineC{\dnorm{\zeta_{2}}{},\frac{1}{1-\dnorm{\delta_{2}}{}},\frac{1}{1-\dnorm{\delta_{2}+\zeta_{2}}{}}},\notag\\
 &\widehat C_2\coloneqq (1+2^{\frac{\chi_2\chi_4}2-1})\left((n+1)
 \widetilde D_{1}\widetilde C_{\clN1}\right)^{\frac{\chi_2\chi_4}2}
 =\ovlineC{n,\widetilde C_{\clN1},\dnorm{\zeta_{1}}{},\dnorm{\zeta_{2}}{},\frac{1}{1-\dnorm{\delta_{2}}{}}}.\notag
 \end{align}

Since~$\widehat C_1\ge2$ holds for~$\chi_2\ge0$ we can write
\[
 1+\gamma_2^{-\chi_4}\le
 \widehat C_1+\widehat C_2\left(1+\widetilde\gamma_0^{-\chi_5 }\right)^{\frac{\chi_2\chi_4}2}
 \widehat\gamma_0^{-\frac{\chi_2\chi_4}2},\quad\mbox{for}\quad\chi_2\ge0.
\]
Further, we see that 
 \begin{align}
1+\gamma_2^{-\chi_4}&\le
 \widehat C_1+\widehat C_2(1+2^{\frac{\chi_2\chi_4}2-1})
 \left(\widehat\gamma_0^{-\frac{\chi_2\chi_4}2}+\widetilde\gamma_0^{-\frac{\chi_5\chi_2\chi_4}2 }
 \widehat\gamma_0^{-\frac{\chi_2\chi_4}2}\right)\notag
\end{align}
and
 \begin{align}
&\widetilde\gamma_0^{-\frac{\chi_5\chi_2\chi_4}2 }\widehat\gamma_0^{-\frac{\chi_2\chi_4}2}
=(n+1)^{\frac{\chi_5\chi_2\chi_4}2}\widehat\gamma_0^{-\frac{(\chi_5+1)\chi_2\chi_4}2 },\notag\\
&\widehat\gamma_0^{-\frac{\chi_2\chi_4}2}\le 1+\widehat\gamma_0^{-\frac{(\chi_5+1)\chi_2\chi_4}2},\notag
\end{align}
from which we obtain
\begin{subequations}\label{1-gamma-2}
\begin{align}
 1+\gamma_2^{-\chi_4}
 &\le\widehat C_1
 +\widehat C_2(1+2^{\frac{\chi_2\chi_4}2-1})\left(1+(1+(n+1)^{\frac{\chi_5\chi_2\chi_4}2})
 \widehat\gamma_0^{-\frac{(\chi_5+1)\chi_2\chi_4}2}\right)\notag\\
 &\le\widehat C_3
 +\widehat C_4\widehat\gamma_0^{-\frac{(\chi_5+1)\chi_2\chi_4}2},
 \intertext{with}
 &\hspace*{-4em}\widehat C_3\coloneqq\widehat C_1+\widehat C_2(1+2^{\frac{\chi_2\chi_4}2-1}),\qquad
 \hspace*{0em}\widehat C_4\coloneqq \widehat C_2(1+2^{\frac{\chi_2\chi_4}2-1})(1+(n+1)^{\frac{\chi_5\chi_2\chi_4}2})
 .
 \end{align}
\end{subequations}

Therefore~\eqref{1-gamma} and~\eqref{1-gamma-2} lead us to
\begin{subequations}\label{1-gamma-3}
 \begin{align}
&\left(1+\gamma_2^{-\chi_4}\right)\left(1+\widetilde\gamma_0^{-\chi_5 }\right)
\le\widehat C_5\left(1+\widehat\gamma_0^{-\chi_5 }\right)
\left(1+\widehat\gamma_0^{-\frac{(\chi_5+1)\chi_2\chi_4}2  }\right),
\intertext{with, recalling that~$\chi_2\chi_4\ge0$,}
&\widehat C_5\coloneqq (n+1)^{\chi_5 }+\widehat C_3+\widehat C_4
=\ovlineC{n,\widetilde C_{\clN1},\dnorm{\zeta_{1}}{},\dnorm{\zeta_{2}}{},
\frac{1}{1-\dnorm{\delta_{2}}{}},\frac{1}{1-\dnorm{\zeta_{2}+\delta_{2}}{}}}.
  \end{align}
\end{subequations}

Hence~\eqref{NNchis}, \eqref{CoefDA}, and~\eqref{1-gamma-3} give us
\begin{align}
  &2\Bigl( \fkN_y(t,z_1),Az_1\Bigr)_{H}
 \le \widehat\gamma_0 \norm{z_1}{\rmD(A)}^{2}\notag\\
 &\hspace*{.0em}  +\widetilde C_{\fkN2}\left(1+\widehat\gamma_0^{-\chi_5 }\right)\!
\left(1+\widehat\gamma_0^{-\frac{(\chi_5+1)\chi_2\chi_4}2  }\right)\!
   \left(1+\norm{y}{V}^{\chi_1} \right)\!
  \left(1+\norm{y}{\rmD(A)}^{\chi_2}\right)\!
  \left(1+\norm{z_1}{V}^{\chi_3}\right)\!
  \norm{z_1}{V}^2,\notag
 \end{align}
with~$\widetilde C_{\fkN2}\coloneqq n\widetilde D_{2}\widehat C_5\widetilde C_{\clN1}
  =\ovlineC{n,\widetilde C_{\clN1},\dnorm{\zeta_{1}}{},
  \dnorm{\zeta_{2}}{},\frac{1}{1-\dnorm{\delta_{2}}{}},\frac{1}{1-\dnorm{\zeta_{2}+\delta_{2}}{}}}$.
This ends the proof of Proposition~\ref{P:fkN}.\qed

\subsection{Proof of Proposition~\ref{P:restOblPro}}

Recall that~$\rmD(A^\xi)\xhookrightarrow{}H$, for~$\xi\ge0$, and~$H=\widetilde\clW_{S}\oplus\clW_{S}^\perp$.
We prove firstly that~$\widetilde\clW_{S}$ and~$\clW_{S}^\perp\cap \rmD(A^\xi)$
are closed subspaces of~$\rmD(A^\xi)$. Clearly~$\widetilde\clW_{S}$
is closed, because it is finite-dimensional.
Let now be an arbitrary sequence~$(h_n)_{n\in\bbN_0}$ in~$\clW_{S}^\perp\cap \rmD(A^\xi)$
and a vector~$\overline h\in \rmD(A^\xi)$,
so that
$\norm{h_n-\overline h}{\rmD(A^\xi)}\to 0$, as~$n\to+\infty$.
Since~$\norm{h_n-\overline h}{H}\le C\norm{h_n-\overline h}{\rmD(A^\xi)}$, for a suitable constant~$C>0$,
it follows that~$\norm{h_n-\overline h}{H}\to 0$, and since~$\clW_{S}^\perp$ is closed in~$H$,
it follows that~$\overline h\in \clW_{S}^\perp$. Thus~$\overline h\in \clW_{S}^\perp\cap \rmD(A^\xi)$,
and we can conclude 
that~$\clW_{S}^\perp\cap \rmD(A^\xi)$ is a closed subspace of~$V$.
Next we observe that~$\rmD(A^\xi)=\widetilde\clW_{S}\oplus(\clW_{S}^\perp\cap \rmD(A^\xi))$,
which is a straightforward consequence of~$H=\widetilde\clW_{S}\oplus\clW_{S}^\perp$.
To show that the oblique
projection~$P_{\widetilde\clW_{S}}^{\clW_{S}^\perp\cap \rmD(A^\xi)}$ in~$\rmD(A^\xi)$ coincides with the
restriction~$\left.P_{\widetilde\clW_{S}}^{\clW_{S}^\perp}\right|_{\rmD(A^\xi)}$ of
the oblique
projection~$P_{\widetilde\clW_{S}}^{\clW_{S}^\perp}$ in~$H$, it is enough to observe that by definition
of a projection we have that
\begin{align}
  \left.P_{\widetilde\clW_{S}}^{\clW_{S}^\perp\cap \rmD(A^\xi)}\right|_{\rmD(A^\xi)}w_1=w_1
 =P_{\widetilde\clW_{S}}^{\clW_{S}^\perp}w_1,\quad\mbox{for all}\quad w_1\in \widetilde\clW_{S},\notag\\
  \left.P_{\widetilde\clW_{S}}^{\clW_{S}^\perp\cap \rmD(A^\xi)}\right|_{\rmD(A^\xi)}w_2=0
 =P_{\widetilde\clW_{S}}^{\clW_{S}^\perp}w_2,\quad\mbox{for all}
 \quad w_2\in\clW_{S}^\perp\cap \rmD(A^\xi).\notag
\end{align}

Finally, we have~$P_{\widetilde\clW_{S}}^{\clW_{S}^\perp\cap \rmD(A^\xi)}\in\clL(\rmD(A^\xi))$
because (oblique) projections are continuous, 
see~\cite[Sect.~2.4,~Thm. 2.10]{Brezis11}.\qed

\subsection{Proof of Proposition~\ref{P:extOblPro}}\label{sSproofP:extOblPro}

 It is clear that~$P_{\widetilde\clW_{S}}^{\clW_{S}^\perp}\Bigr|^{\rmD(A^{-\xi})}$
is an extension of the oblique projection~$P_{\widetilde\clW_{S}}^{\clW_{S}^\perp}\in\clL(H)$
to~$\rmD(A^{-\xi})\supseteq H$,
because for~$z\in H$ we have
that $\left\langle P_{\clW_{S}}^{\widetilde\clW_{S}^\perp}\Bigr|^{\rmD(A^{-\xi})}
z,w\right\rangle_{\rmD(A^{-\xi}),\rmD(A^{\xi})}=
 (z,\clP_{\widetilde\clW_{S}}^{\clW_{S}^\perp}w)_{H}
 =(P_{\clW_{S}}^{\widetilde\clW_{S}^\perp}z,w)_{H}$, where for the last identity we have
 used~$P_{\clW_{S}}^{\widetilde\clW_{S}^\perp}=(\clP_{\widetilde\clW_{S}}^{\clW_{S}^\perp})^*$;
 see~\cite[Lem.~3.8]{RodSturm20}. By the relation~\eqref{extOblPro} and Proposition~\ref{P:extOblPro}
 it follows the inequality
 $\norm{P_{\clW_{S}}^{\widetilde\clW_{S}^\perp}\Bigr|^{\rmD(A^{-\xi})}}{\clL(\rmD(A^{-\xi}))}
  \le\norm{P_{\widetilde\clW_{S}}^{\clW_{S}^\perp}\Bigr|_{\rmD(A^{\xi})}}{\clL(\rmD(A^{\xi}))}<+\infty$,
  and afterwards the same relation~\eqref{extOblPro} gives us the converse inequality. Hence we obtain the
  stated norm identity. Finally, by definition of the adjoint operator
  we also have the stated adjoint identity.
\qed

\subsection{Proof of Proposition~\ref{P:maxpoly1r}}

Observe that, since~$\fks\in(0,1)$, we have that~$g(\tau)\coloneqq-\eta_1\tau+\eta_2\tau^\fks$
satisfies~$g(0)=0$, $\lim\limits_{\tau\to+\infty}g(\tau)=-\infty$, and
$\frac{\ed}{\ed\tau}\rest{\tau=\tau_0}g(\tau)=-\eta_1+\fks\eta_2\tau_0^{\fks-1} $, for~$\tau_0>0$.
In particular,~$g$ is differentiable at each~$\tau_0>0$. Furthermore,
$\frac{\ed}{\ed\tau}\rest{\tau=\tau_0}g(\tau)>0\Longleftrightarrow \tau_0^{\fks-1}
>\eta_1(\fks\eta_2)^{-1}\Longleftrightarrow
\tau_0^{1-\fks}<(\fks\eta_2)\eta_1^{-1}\Longleftrightarrow\tau_0
<(\fks\eta_2)^\frac{1}{1-\fks}\eta_1^{-\frac{1}{1-\fks}}$.
Thus $g(\tau)$ strictly increases only if~$\tau\in(0,\overline\tau)$,
with~$\overline\tau\coloneqq(\fks\eta_2)^\frac{1}{1-\fks}\eta_1^{-\frac{1}{1-\fks}}
=(\fks\eta_2)^\frac{1}{1-\fks}\eta_1^{\frac{1}{\fks-1}}$. 
Analogously we find that~$\frac{\ed}{\ed\tau}\rest{\tau=\tau_0}g(\tau)>0\Longleftrightarrow \tau_0>\overline\tau$.
Necessarily, the maximum is attained at~$\overline\tau>0$, and can be computed as
\begin{align}
-\eta_1\overline\tau+\eta_2\overline\tau^\fks& =-\eta_1(\fks \eta_2)^\frac{1}{1-\fks}\eta_1^\frac{1}{\fks-1}
 + \eta_2\left((\fks \eta_2)^\frac{1}{1-\fks}\eta_1^\frac{1}{\fks-1}\right)^{\fks}
 = \eta_2^\frac{1}{1-\fks}\eta_1^\frac{\fks}{\fks-1}\left(-\fks^\frac{1}{1-\fks}+\fks^\frac{\fks}{1-\fks}\right).\notag
 \end{align}
Thus, $-\eta_1\overline\tau+\eta_2\overline\tau^\fks
=(1-\fks)\fks^\frac{\fks}{1-\fks}\eta_2^\frac{1}{1-\fks}\eta_1^\frac{\fks}{\fks-1}$, which finishes the proof.\qed

\subsection{Proof of Proposition~\ref{P:ode-stab0}}

For the sake of simplicity we shall omit the subscript in the usual norm in~$\bbR$,
that is, $\norm{\Bigcdot}{}\coloneqq\norm{\Bigcdot}{\bbR}$.
The solution of~\eqref{ode-nonl-p0} is given by
\begin{equation}\label{ode-v1}
 v(t)=\ex^{-\overline\mu (t-s)+\int_s^t\norm{h(\tau)}{}\,\ed\tau}v(s),\quad t\ge s\ge 0,\quad v(0)=v_0.
\end{equation}

Observe that the exponent satisfies, using~\eqref{ode-h},
\begin{align}
-\overline\mu (t-s)+\textstyle\int_s^t\norm{h(\tau)}{}\,\ed\tau
&\le-\overline\mu (t-s)+(t-s)^\frac{\fkr-1}{\fkr}\left(\textstyle\int_s^t
\norm{h(\tau)}{}^\fkr\,\ed\tau\right)^\frac{1}{\fkr}\notag\\
&\le-\overline\mu (t-s)+(t-s)^\frac{\fkr-1}{\fkr}\left(\textstyle\int_s^{s+T\lceil \tfrac{t-s}{T}\rceil}
\norm{h(\tau)}{}^\fkr\,\ed\tau\right)^\frac{1}{\fkr}\notag\\
&\le-\overline\mu (t-s)+(t-s)^\frac{\fkr-1}{\fkr}\left( \lceil \tfrac{t-s}{T}\rceil
C_h^\fkr \right)^\frac{1}{\fkr},\label{ode-v-exp0}
\end{align}
where~$\lceil r\rceil\in\bbN_0$ is the positive integer defined by
\begin{equation}\label{ceil}
  r\le\lceil r\rceil < r+1.
\end{equation}

From~\eqref{ode-v-exp0} and~\eqref{ceil}, it follows that
\begin{align}
-\overline\mu (t-s)+\textstyle\int_s^t\norm{h(\tau)}{}\,\ed\tau
&\le-\overline\mu (t-s)+(t-s)^\frac{\fkr-1}{\fkr}\left( \tfrac{t-s}{T}+1\right)^\frac{1}{\fkr} C_h\notag\\
&\le T^{-\frac{1}{\fkr}}(-\overline\mu T^\frac{1}{\fkr}+C_h )(t-s)+(t-s)^\frac{\fkr-1}{\fkr} C_h,\label{ode-v-exp1}
\end{align}
where we have used~$\left( \tfrac{t-s}{T}+1\right)^\frac{1}{\fkr}\le (\tfrac{t-s}{T})^\frac{1}{\fkr}+1$,
since~$\fkr>1$, see~\cite[Proposition~2.6]{PhanRod17}.

By~\eqref{ode-condstab0}, we have that
\begin{equation}\label{mu-large}
\widehat\mu\coloneqq T^{-\frac{1}{\fkr}}(\overline\mu T^\frac{1}{\fkr}-C_h ) \ge
\max\left\{2\tfrac{\fkr-1}{\fkr}\left(\tfrac{C_h^\fkr}{\fkr\log(\varrho)}\right)^\frac{1}{\fkr-1}, 2\mu\right\}>0,
\end{equation}
from which, together with~\eqref{ode-v-exp1} and Proposition~\ref{P:maxpoly1r}, we obtain 
\begin{align}
-\overline\mu (t-s)+\textstyle\int_s^t\norm{h(\tau)}{}\,\ed\tau
&\le -\tfrac{1}{2}\widehat\mu(t-s)-\tfrac{1}{2}\widehat\mu(t-s)+(t-s)^\frac{\fkr-1}{\fkr} C_h\notag\\
&\le -\tfrac{\widehat\mu}{2}(t-s) +\tfrac{1}{\fkr}(\tfrac{\fkr-1}{\fkr})^{\fkr-1}C_h^\fkr
(\tfrac{\widehat\mu}{2})^{1-\fkr},\label{ode-v-exp2}
\end{align}
because by Proposition~\ref{P:maxpoly1r}, with~$\fks=\frac{\fkr-1}{\fkr}$ and~$\eta_1
=\tfrac{\widehat\mu}{2}$, $\eta_2=C_h$,
\[
 \max_{t-s\ge0}\{-\eta_1(t-s)+(t-s)^\frac{\fkr-1}{\fkr} \eta_2\}
 =(1-\fks)\fks^\frac{\fks}{1-\fks}\eta_2^\frac{1}{1-\fks}\eta_1^\frac{\fks}{\fks-1}
 =\tfrac{1}{\fkr}(\tfrac{\fkr-1}{\fkr})^{\fkr-1}\eta_2^{\fkr}\eta_1^{1-\fkr}.
\]

Therefore, from~\eqref{ode-v1}, ~\eqref{mu-large}, and~\eqref{ode-v-exp2}, we derive that
\begin{equation}\notag
 \norm{v(t)}{}\le\ex^{\tfrac{C_h^\fkr}{\fkr}\left(\tfrac{2(\fkr-1)}{\fkr}\right)^{\fkr-1}
 \widehat\mu^{1-\fkr}}\ex^{-\frac{\widehat\mu}{2}(t-s)}\norm{v(s)}{}
 \le\varrho\ex^{-\mu(t-s)}\norm{v(s)}{},
\end{equation}
which gives us~\eqref{ode-v1}. Indeed, observe that
\begin{align}
&\ex^{\tfrac{C_h^\fkr}{\fkr}\left(\tfrac{2(\fkr-1)}{\fkr}\right)^{\fkr-1}\widehat\mu^{1-\fkr}}\le\varrho
\quad\Longleftrightarrow\quad
\tfrac{C_h^\fkr}{\fkr}\left(\tfrac{2(\fkr-1)}{\fkr}\right)^{\fkr-1}\widehat\mu^{1-\fkr}\le\log(\varrho)\notag\\
\Longleftrightarrow\quad&
\tfrac{C_h^\fkr}{\fkr\log(\varrho)}\left(\tfrac{2(\fkr-1)}{\fkr}\right)^{\fkr-1}\le\widehat\mu^{\fkr-1}
\quad\Longleftrightarrow\quad
\widehat\mu\ge\left(\tfrac{C_h^\fkr}{\fkr\log(\varrho)}\right)^\frac{1}{\fkr-1}\tfrac{2(\fkr-1)}{\fkr},\notag
\end{align}
and the last inequality follows from~\eqref{mu-large}, which also gives us~$\frac{\widehat\mu}{2}>\mu$.
\qed

\subsection{Proof of Proposition~\ref{P:ode-stab}}

We shall use a fixed point argument, through the contraction principle, in the closed subset
\[
 \clZ_{\varrho,\norm{\varpi_0}{}}^{\mu_0}\coloneqq\left\{g\in L^\infty(\bbR_0,\bbR)\left|\;\;
 \norm{\ex^{\mu_0 t}g(t)}{}\le\varrho\norm{\varpi_0}{}\right.\!\right\}
\]
of the Banach space           
\begin{align}
 \clZ^\mu_0&\coloneqq\left\{g\in L^\infty(\bbR_0,\bbR)\left|\;\;
 \ex^{\mu_0(\Bigcdot)}g\in L^\infty(\bbR_0,\bbR)\right.\!\right\},
 \qquad\norm{g}{\clZ^{\mu_0}}\coloneqq\sup_{t\ge0}\norm{\ex^{{\mu_0} t}g(t)}{}.\notag
\end{align}
We show  now that, since~\eqref{ode-condstab} holds true, the mapping
\begin{equation}\notag
 \varPsi\colon\clZ_{\varrho,\norm{\varpi_0}{}}^{\mu_0}\to\clZ_{\varrho,\norm{\varpi_0}{}}^{\mu_0},\qquad
 \breve \varpi\mapsto \varpi,
\end{equation}
where~$\varpi$ solves
\begin{equation}\label{ode-Psi}
 \dot \varpi=-(\overline\mu-\norm{h}{})\varpi+\norm{h}{}\norm{\breve\varpi}{}^p\breve \varpi,\quad \varpi(0)=\varpi_0,
\end{equation}
is well defined and is a contraction in~$\clZ_{\varrho,\norm{\varpi_0}{}}^{{\mu_0}}$.

We look at~\eqref{ode-Psi} as a perturbation of the nominal linear system
\begin{equation}\label{ode-Psi-nom}
 \dot v=-(\overline\mu-\norm{h}{})v,\quad v(0)=v_0=\varpi_0\in\bbR.
\end{equation}
Note that~\eqref{ode-condstab} implies that
\[
 \overline\mu  \ge \max\left\{2\tfrac{\fkr-1}{\fkr}
 \left(\tfrac{C_h^\fkr}{\fkr\log\left(\varrho^\frac12\right)}\right)^\frac{1}{\fkr-1}, 4{\mu_0}\right\}
 -T^{-\frac{1}{\fkr}}C_h
 \]
which we use together with Proposition~\ref{P:ode-stab0}  to conclude that the solution
\[
 v(t)\eqqcolon\clS(t,s)v(s)
\]
 of~\eqref{ode-Psi-nom} satisfies
\begin{equation}\label{odev2-exp}
\norm{v(t)}{}=\norm{\clS(t,s)v(s)}{}\le\varrho^\frac{1}{2}\ex^{-2{\mu_0}(t-s)}\norm{v(s)}{},
\quad t\ge s\ge 0,\quad v(0)=v_0.
\end{equation}

By the Duhamel formula we have that the solution~$w$ of~\eqref{ode-Psi} is given as
\begin{equation}\label{ode-Psi-sol}
\varpi(t)=\clS(t,s)\varpi(s)+\textstyle\int_s^t\clS(t,\tau)
\norm{h(\tau)}{}\norm{\breve\varpi(\tau)}{}^p\breve\varpi(\tau)\,\ed\tau,
\quad \varpi=\varPsi(\breve\varpi).
\end{equation}

\begin{enumerate}[noitemsep,topsep=5pt,parsep=5pt,partopsep=0pt,leftmargin=0em]
\renewcommand{\theenumi}{{\sf\arabic{enumi}}} 
 \renewcommand{\labelenumi}{} 
 \item \textcircled{\bf s}~Step~\theenumi:\label{st-PsifixZ}
 {\em $\varPsi$ maps~$\clZ_{\varrho,\norm{\varpi_0}{}}^{\mu_0}$ into itself, if~$\norm{\varpi_0}{}<\varrho R$.}
We observe that~\eqref{odev2-exp} and~\eqref{ode-Psi-sol} give us the estimate
\begin{equation}\label{ode-Psi-1}
 \norm{\varpi(t)}{}\le\varrho^\frac{1}{2}\ex^{-2{\mu_0} t}\norm{\varpi_0}{}
 +\textstyle\int_0^t\varrho^\frac{1}{2}\ex^{-2{\mu_0}(t-\tau)}\norm{h(\tau)}{}
 \norm{\breve\varpi(\tau)}{}^{p+1}\,\ed\tau.
\end{equation}
Next, we also find, since~$\breve\varpi\in\clZ_{\varrho,\norm{\varpi_0}{}}^{\mu_0}$,
 \begin{align}
 &\textstyle\int_0^t\ex^{-2{\mu_0}(t-\tau)}\norm{h(\tau)}{}\norm{\breve\varpi(\tau)}{}^{p+1}\,\ed\tau
    \le\varrho^{p+1}\norm{\varpi_0}{}^{p+1}\textstyle\int_0^t\ex^{-2{\mu_0}(t-\tau)}
    \ex^{-{\mu_0} (p+1)\tau }\norm{h(\tau)}{}\,\ed\tau
 \notag\\
 &\hspace*{2em}\le\varrho^{p+1}\norm{\varpi_0}{}^{p+1} \ex^{-{\mu_0} t}
 \textstyle\int_0^t\ex^{-{\mu_0} (t-\tau)}\ex^{-{\mu_0} p\tau }\norm{h(\tau)}{}\,\ed\tau
 \notag\\
 &\hspace*{2em}\le\varrho^{p+1}\norm{\varpi_0}{}^{p+1} \ex^{-{\mu_0} t}
 \left(\textstyle\int_0^t\ex^{-\frac{\fkr}{\fkr-1}{\mu_0} (t-\tau)}\,\ed\tau\right)^{\frac{\fkr-1}{\fkr}}
 \left(\textstyle\int_0^t\ex^{-\fkr {\mu_0} p\tau }\norm{h(\tau)}{}^\fkr\,\ed\tau\right)^{\frac{1}{\fkr}}
 \notag\\
 &\hspace*{2em}\le\varrho^{p+1}\norm{\varpi_0}{}^{p+1} \ex^{-{\mu_0} t}
 \left(\tfrac{\fkr-1}{\fkr{\mu_0}}\right)^{\frac{\fkr-1}{\fkr}}
 \left(\textstyle\sum\limits_{i=1}^{\lceil \frac{t}{T}\rceil}\ex^{-\fkr {\mu_0} p(i-1)T }
 \int_{(i-1)T}^{iT} \norm{h(\tau)}{}^\fkr\,\ed\tau\right)^{\frac{1}{\fkr}}
 \notag\\
 &\hspace*{2em}\le\varrho^{p+1}\norm{\varpi_0}{}^{p+1} \ex^{-{\mu_0} t}
 \left(\tfrac{\fkr-1}{\fkr{\mu_0}}\right)^{\frac{\fkr-1}{\fkr}}C_h
 \left(\textstyle\sum\limits_{i=1}^{\lceil \frac{t}{T}\rceil}\ex^{-\fkr{\mu_0} p(i-1)T }\right)^{\frac{1}{\fkr}}
 \notag\\
 &\hspace*{2em}\le\varrho^{p+1}\norm{\varpi_0}{}^{p+1}C_h
 \left(\tfrac{1}{1-\ex^{-\fkr {\mu_0} pT }} \right)^{\frac{1}{\fkr}}
 \left(\tfrac{\fkr-1}{\fkr{\mu_0}}\right)^{\frac{\fkr-1}{\fkr}} \ex^{-{\mu_0} t}.\label{ode2}
  \end{align}

By combining~\eqref{ode-Psi-1} with~\eqref{ode2}, we arrive at
\begin{align}
 \ex^{{\mu_0} t}\norm{\varpi(t)}{}&\le\varrho^\frac{1}{2}\ex^{-{\mu_0} t}\norm{\varpi_0}{}
 +\varrho^{p+\frac{3}{2}}\norm{\varpi_0}{}^{p+1}C_h\left(\tfrac{1}{1-\ex^{-\fkr {\mu_0} pT }} \right)^{\frac{1}{\fkr}}
 \left(\tfrac{\fkr-1}{\fkr}\right)^{\frac{\fkr-1}{\fkr}}  {\mu_0}^{\frac{1-\fkr}{\fkr}} \notag\\
 &\le\varrho^\frac{1}{2}\left(1
 +\varrho^{p+1}\norm{\varpi_0}{}^{p}C_h\left(\tfrac{1}{1-\ex^{-\fkr {\mu_0} pT }} \right)^{\frac{1}{\fkr}}
 \left(\tfrac{\fkr-1}{\fkr}\right)^{\frac{\fkr-1}{\fkr}}  {\mu_0}^{\frac{1-\fkr}{\fkr}}\right)
 \norm{\varpi_0}{}.\label{ode-comb}
 \end{align}
 Next we use~\eqref{ode-nonl-mumu0} and~$\norm{\varpi_0}{}\le \varrho R$ to obtain
 \begin{subequations}\label{ode-nonl-mumu0-1}
 \begin{align}
 \tfrac{1}{1-\ex^{-\fkr{\mu_0} pT }}\le\tfrac{1}{1-\ex^{-{\mu_0} pT }}\le 2,\label{log2} 
 \end{align}
and
 \begin{align}
&1 +\varrho^{p+1}\norm{\varpi_0}{}^{p}C_h\left(\tfrac{1}{1-\ex^{-\fkr {\mu_0} pT }} \right)^{\frac{1}{\fkr}}
 \left(\tfrac{\fkr-1}{\fkr}\right)^{\frac{\fkr-1}{\fkr}}  {\mu_0}^{\frac{1-\fkr}{\fkr}} 
\le1 +\varrho^{2p+1}R^{p}C_h\left(\tfrac{\fkr-1}{\fkr}\right)^{\frac{\fkr-1}{\fkr}}
{\mu_0}^{\frac{1-\fkr}{\fkr}}2^\frac{1}{\fkr}\notag\\
&\hspace*{1em}\le1 +\varrho^{2p+1}R^{p}C_h\left(\tfrac{\fkr-1}{\fkr}\right)^{\frac{\fkr-1}{\fkr}}2^\frac{1}{\fkr}
\left(\tfrac{\varrho^{2p+1}R^{p}C_h}{\varrho^\frac12-1}\right)^{-1}2^{-\frac{1}{\fkr}}
\left(\tfrac{\fkr-1}{\fkr}\right)^{\frac{1-\fkr}{\fkr}}
=1 +\left(\tfrac{1}{\varrho^\frac12-1}\right)^{-1}\notag\\
&\hspace*{1em}=\varrho^\frac12.
\end{align}
 \end{subequations}

From~\eqref{ode-comb} and~\eqref{ode-nonl-mumu0-1}, we
find~$\ex^{{\mu_0} t}\norm{\varpi(t)}{}\le\varrho\norm{\varpi_0}{}$, hence
$\varpi=\varPsi(\breve\varpi)\in\clZ_{\varrho,\norm{\varpi_0}{}}^{\mu_0}$.

\item \textcircled{\bf s}~Step~\theenumi:\label{st-PsicontrZ}
 {\em $\varPsi$ is a contraction in~$\clZ_{\varrho,\norm{\varpi_0}{}}^{\mu_0}$, if~$\norm{\varpi_0}{}<\varrho R$.}
 For an arbitrary given~$(\breve\varpi_1,\breve\varpi_2)\in\clZ_{\varrho,\norm{\varpi_0}{}}^{\mu_0}
 \times\clZ_{\varrho,\norm{\varpi_0}{}}^{\mu_0}$, we have that the difference
 \[
 D\coloneqq\varPsi(\breve\varpi_1)-\varPsi(\breve\varpi_2) 
 \]
solves
\[
 \dot D=-(\overline\mu-\norm{h}{})D
 +\norm{h}{}\left(\norm{\breve\varpi_1}{}^p\breve\varpi_1-\norm{\breve\varpi_2}{}^p\breve\varpi_2\right),\quad D(0)=0,
\]

By the Duhamel formula and the Mean Value Theorem, we obtain
\begin{align}
\norm{D(t)}{} &=\norm{\clS(t,0)D(0)}{}
+\norm{\textstyle\int_0^t\clS(t,\tau)\norm{h(\tau)}{}
\norm{\norm{\breve\varpi_1}{}^p\breve\varpi_1-\norm{\breve\varpi_2}{}^p\breve\varpi_2}{}
\,\ed\tau}{}\notag\\
&\hspace*{0em}\le\varrho^\frac12 (p+1)\textstyle\int_0^t\ex^{-{\mu_0}(t-\tau)}\norm{h(\tau)}{}
\left(\norm{\breve\varpi_1(\tau)}{}^p+\norm{\breve\varpi_2(\tau)}{}^p\right)
\norm{\breve\varpi_1(\tau)-\breve\varpi_2(\tau)}{}\,\ed\tau\notag\\
&\hspace*{0em}\le\varrho^\frac12 (p+1)
\norm{\breve\varpi_1-\breve\varpi_2}{\clZ_{\varrho,\norm{\varpi_0}{}}^{\mu_0}}\ex^{-{\mu_0} t}
\textstyle\int_0^t\norm{h(\tau)}{}\left(\norm{\breve\varpi_1(\tau)}{}^p+\norm{\breve\varpi_2(\tau)}{}^p\right)
\,\ed\tau\notag\\
 &\hspace*{0em}\le2\varrho^{p+\frac12} (p+1)\norm{\varpi_0}{}^p
 \norm{\breve\varpi_1-\breve\varpi_2}{\clZ_{\varrho,\norm{\varpi_0}{}}^{\mu_0}}
 \ex^{-{\mu_0} t}\textstyle\int_0^t\ex^{-{\mu_0}\tau p}\norm{h(\tau)}{}\,\ed\tau\label{D-1}
 \end{align}
 Note that
 \begin{align}
 &\textstyle\int_0^t\ex^{-{\mu_0}\tau p}\norm{h(\tau)}{}\,\ed\tau
 =\textstyle\int_0^t\ex^{-\frac{\fkr-1}{\fkr}{\mu_0}\tau p}\ex^{-\frac{1}{\fkr}{\mu_0}\tau p}
 \norm{h(\tau)}{}\,\ed\tau\notag\\
 &\hspace*{2em}\le\left(\textstyle\int_0^t\ex^{-{\mu_0}\tau p}\,\ed\tau\right)^\frac{\fkr-1}{\fkr}
 \left(\textstyle\int_0^t\ex^{-{\mu_0}\tau p}\norm{h(\tau)}{}^\fkr\,\ed\tau\right)^\frac{1}{\fkr}\notag\\
 &\hspace*{2em} \le\left({\mu_0} p\right)^\frac{1-\fkr}{\fkr}
  \left(\textstyle\sum\limits_{i=1}^{\lceil\frac{t}{T}\rceil} \ex^{-{\mu_0} p(i-1)T}\int_{(i-1)T}^{iT}
  \norm{h(\tau)}{}^\fkr\,\ed\tau\right)^\frac{1}{\fkr}\notag\\
  &\hspace*{2em}\le\left({\mu_0} p\right)^\frac{1-\fkr}{\fkr}C_h
  \left(\textstyle\sum\limits_{i=1}^{\lceil\frac{t}{T}\rceil} \ex^{-{\mu_0} p(i-1)T}\right)^\frac{1}{\fkr}
  \le\left({\mu_0} p\right)^\frac{1-\fkr}{\fkr}C_h
  \left(\tfrac{1}{1- \ex^{-{\mu_0} pT}}\right)^\frac{1}{\fkr}.\label{D-2}
  \end{align}
 
 From~\eqref{D-1} and~\eqref{D-2},
 \begin{align}
 \ex^{{\mu_0} t}\norm{D(t)}{}&\le2\varrho^{p+\frac12} (p+1) p^\frac{1-\fkr}{\fkr}
 C_h  \left(\tfrac{1}{1- \ex^{-{\mu_0} pT}}\right)^\frac{1}{\fkr}\norm{\varpi_0}{}^p{\mu_0}^\frac{1-\fkr}{\fkr}
 \norm{\breve\varpi_1-\breve\varpi_2}{\clZ_{\varrho,\norm{\varpi_0}{}}^{\mu_0}},\notag
 \end{align}
 which together~$\norm{\varpi_0}{}\le\varrho R$ and~${\mu_0}\ge\frac{\log(2)}{pT}$, see~\eqref{ode-nonl-mumu0},
 give us~$\tfrac{1}{1- \ex^{-{\mu_0} pT}}\le2$ and
 \begin{align}
 &\ex^{{\mu_0} t}\norm{D(t)}{}\le2^\frac{\fkr+1}{\fkr}\varrho^{2p+\frac12} (p+1) p^\frac{1-\fkr}{\fkr}
 C_h  R^p{\mu_0}^\frac{1-\fkr}{\fkr}
 \norm{\breve\varpi_1-\breve\varpi_2}{\clZ_{\varrho,\norm{\varpi_0}{}}^{\mu_0}}\notag\\
 &\le2^\frac{\fkr+1}{\fkr}\varrho^{2p+\frac12} (p+1) p^\frac{1-\fkr}{\fkr}
 C_h  R^p{\mu_0}^\frac{1-\fkr}{\fkr}
 \norm{\breve\varpi_1-\breve\varpi_2}{\clZ_{\varrho,\norm{\varpi_0}{}}^{\mu_0}}\notag\\
 &\le2^\frac{\fkr+1}{\fkr}\varrho^{2p+\frac12} (p+1) p^\frac{1-\fkr}{\fkr}
 C_h  R^p\left(2^\frac{\fkr+1}{\fkr-1} \left(
  \varrho^{2p+\frac12} C_h\tfrac{p+1}{p}R^pc\right)^\frac{\fkr}{\fkr-1} p^\frac{1}{\fkr-1}\right)^\frac{1-\fkr}{\fkr}
\norm{\breve\varpi_1-\breve\varpi_2}{\clZ_{\varrho,\norm{\varpi_0}{}}^{\mu_0}}\notag\\
&\le  c^{-1} p^\frac{1-\fkr}{\fkr}
 \left( \left(
  \tfrac{1}{p}\right)^\frac{\fkr}{\fkr-1} p^\frac{1}{\fkr-1}\right)^\frac{1-\fkr}{\fkr}
\norm{\breve\varpi_1-\breve\varpi_2}{\clZ_{\varrho,\norm{w_0}{}}^{\mu_0}}
=c^{-1}\norm{\breve\varpi_1-\breve\varpi_2}{\clZ_{\varrho,\norm{\varpi_0}{}}^{\mu_0}}\label{D-3}
 \end{align}
with~$c>1$ as in~\eqref{ode-nonl-mumu0}. Therefore, \eqref{D-3} implies that
\begin{align}
 \norm{\varPsi(\breve\varpi_1)-\varPsi(\breve\varpi_2)}{\clZ_{\varrho,\norm{\varpi_0}{}}^{\mu_0}}
 &=\norm{D}{\clZ_{\varrho,\norm{\varpi_0}{}}^{\mu_0}}\le c^{-1}\norm{d}{\clZ_{\varrho,\norm{\varpi_0}{}}^{\mu_0}}
 =c^{-1}\norm{\breve\varpi_1-\breve\varpi_2}{\clZ_{\varrho,\norm{\varpi_0}{}}^{\mu_0}},\notag
\end{align}
which shows that~$\varPsi$ is a contraction.

\item \textcircled{\bf s}~Step~\theenumi:\label{st-existence}
 {\em Existence of a solution in~$\clZ_{\varrho,\norm{\varpi_0}{}}^{\mu_0}$, if~$\norm{\varpi_0}{}<\varrho R$.} 
 By the contraction mapping principle, there exists a fixed point for~$\Psi$
 in~$\clZ_{\varrho,\norm{\varpi_0}{}}^{\mu_0}$.
 Such fixed point is a solution for~\eqref{ode-nonl}.

\item \textcircled{\bf s}~Step~\theenumi:\label{st-uniqueness}
 {\em Uniqueness of the solution in~$L^\infty(\bbR_0,\bbR)$.} The uniqueness follows
 from the fact that the right-hand side of~\eqref{ode-nonl} is locally
 Lipschitz.

 \item \textcircled{\bf s}~Step~\theenumi:\label{st-exp-s}
 {\em Estimate~\eqref{ode-stab} holds true.}
Fix~$s\ge0$ and note that~$\widetilde h(\tau)\coloneqq h(\tau+s)$ also
satisfies~\eqref{ode-h}, with~$C_{\widetilde h}\le C_h$.

Let~$\varpi_{\underline s}\coloneqq \varpi\rest{\bbR_s}$ be the restriction to~$\bbR_s=[s,+\infty)$ of
the solution~$\varpi\in\clZ_{\varrho,\norm{\varpi_0}{}}^{\mu_0}$ of~\eqref{ode-nonl},
and observe that~$z(\tau)\coloneqq \varpi_{\underline s}(\tau+s)$ solves
\begin{equation}\notag
 \tfrac{\ed}{\ed\tau} z=-(\overline\mu-\norm{\widetilde h}{})z+\norm{\widetilde h}{}\norm{z}{}^pz,
 \quad z(0)=z_0,\qquad \tau\ge0.
\end{equation}

If~$\norm{\varpi_0}{}<R$ it follows that~$\norm{z_0}{}=\norm{\varpi(s)}{}
\le\varrho\ex^{-{\mu_0} s}\norm{\varpi_0}{}\le\varrho R$.
Then, by Step~\ref{st-existence}
we have that $z\in\clZ_{\varrho,\norm{z_0}{}}^{\mu_0}$, which implies that
for~$t\ge s$,
\[
 \norm{\varpi(t)}{}=\norm{\varpi_{\underline s}(s+t-s)}{}=\norm{z(t-s)}{}
 \le\varrho\ex^{-{\mu_0} (t-s)}\norm{z(0)}{}
 =\varrho\ex^{-{\mu_0} (t-s)}\norm{\varpi(s)}{},
\]
which gives us~\eqref{ode-stab}.
 \end{enumerate}
 
 The  proof is finished.\qed

\subsection{Proof of Proposition~\ref{P:normPoly}}\label{sS:proofP:normPoly}

Let us denote by~$\tau^i=(\tau^i_1,\tau^i_2,\dots,\tau^i_d))\in\bbR^d$
the unit vector whose coordinates are~$\tau_i^i=1$ and~$\tau_j^i=0$ for~$j\ne i$.
Observe that~$\bfJ_{d,2}$ has exactly~$d+1$ vectors. The only element in~$\bfJ_{d,2}$
with~$\textstyle\sum_{j=1}^d\bfj_j= d$ is~${\bf1}^d\coloneqq(1,1,\dots,1)$.
All the other elements in~$\bfJ_{d,2}$ are of the form
${\bf1}^d+\tau^i$, $i=1,2,\dots,d$.

Let now $p\in\bbP_{\times,1}$ such that~$\fkS(p)=0$, which implies that
\[
\norm{(p,1_{\omega_{{\bf1}^d,1}^\times})}{\bbR}^{2}=0,\quad\mbox{and}\quad
\norm{(p,1_{\omega_{{\bf1}^d+\tau^i,1}^\times})}{\bbR}^{2}=0,
\mbox{ for all } i=\{1,2,\dots,d\},
\]
that is, with~$\omega_{*}\coloneqq\omega_{{\bf1}^d,1}$
\[
\textstyle\int_{\omega_{*}}p(x)\,\ed x=0,\quad\mbox{and}\quad\int_{\omega_{*}}p(x-\tau^i)\,\ed x=0.
\quad 1\le i\le d,
\]
Denoting~$\fkL_ax\coloneqq\sum_{i=1}^da_ix_i$, and~$p(x)\eqqcolon a_0+\fkL_ax$, we obtain 
\[
\textstyle\int_{\omega_{*}}c_0+\fkL_ax\,\ed x=0,\quad\mbox{and}\quad\int_{\omega_{*}}c_0+\fkL_a(x-\tau^i)\,\ed x=0,
\quad 1\le i\le d,
\]
which implies
\begin{equation}\label{d-ints=0}
\textstyle\int_{\omega_{*}}c_0+\fkL_ax\,\ed x=0,\quad\mbox{and}\quad\int_{\omega_{*}}\fkL_a\tau^i\,\ed x=0,
\quad 1\le i\le d.
\end{equation}
Note that for fixed~$i$ we have
\[
 \int_{\omega_{*}}\fkL_a\tau^i\,\ed x=0\quad\Longleftrightarrow\quad\int_{\omega_{*}}a_i\,\ed x=0
 \quad\Longleftrightarrow\quad a_i=0,
\]
which together with~\eqref{d-ints=0} leads us to~$a_i=0$, $1\le i\le d$, and $c_0=0$.

We have just shown that~$p\in\bbP_{\times,1}$ and~$\fkS(p)=0$ imply that $p=0$.
Therefore, we can conclude that~$\fkS(\Bigcdot)$ is a norm on~$\bbP_{\times,1}$.\qed

\subsection{Proof of Proposition~\ref{P:underAlpha}}\label{Apx:proofP:underAlpha}

Let $\theta=\sum\limits_{k=1}^{S_\sigma}\theta_k\Phi_k\in\widetilde\clW_{S}$,
with the auxiliary functions~$\Phi_i$ as in~\eqref{choice-tilW-adhoc}.
Then, after a translation, for the~$H$-norm we find that
\[
\norm{\Phi_k}{H}^2=\bigtimes_{j=1}^d
\norm{\sin^2(\tfrac{S\pi x_j}{L_j})}{L^2((0,\frac{L_j}{S}),\bbR)}^2=(\tfrac{3}{8S})^d\bigtimes_{j=1}^d L_j
\]
and, with~$L^\times\coloneqq\bigtimes_{j=1}^d L_j$, since the~$\Phi_i$s are pairwise orthogonal, we arrive at
\begin{align}\notag
 \norm{\theta}{H}^2={\textstyle\sum\limits_{k=1}^{S_\sigma}}\theta_k^2\norm{\Phi_k}{H}^2
 =S^{-d}(\tfrac{3}{8})^dL^\times{\textstyle\sum\limits_{k=1}^{S_\sigma}}\theta_k^2.
\end{align}

Next, for the~$V$-norm we find
\begin{align}
 \norm{\theta}{V}^2={\textstyle\sum\limits_{k=1}^{S_\sigma}}\theta_k^2\norm{\Phi_k}{V}^2
 =\nu{\textstyle\sum\limits_{k=1}^{S_\sigma}}\theta_k^2\norm{\nabla\Phi_k}{L^2(\Omega)^d}^2+\norm{\theta}{H}^2\notag
 \end{align}
and, due to
\begin{align}
\norm{\nabla\Phi_k}{L^2(\Omega)^d}^2&={\textstyle\sum\limits_{i=1}^{d}}
\norm{\tfrac{S\pi}{L_i}\sin(\tfrac{2S\pi x_i}{L_i})}{L^2((0,\frac{L_i}{S}),\bbR)}^2
\bigtimes_{i\ne j=1}^d
\norm{\sin^2(\tfrac{S\pi x_j}{L_j})}{L^2((0,\frac{L_j}{S}),\bbR)}^2\notag\\
&={\textstyle\sum\limits_{i=1}^{d}}
(\tfrac{S\pi}{L_i})^2\tfrac{L_i}{2S}
\bigtimes_{i\ne j=1}^d\tfrac{3}{8S} L_j={\textstyle\sum\limits_{i=1}^{d}}
(\tfrac{S\pi}{L_i})^2\tfrac{4}{3}
\bigtimes_{j=1}^d\tfrac{3}{8S} L_j
\notag\\
&=(\tfrac{3}{8S})^dL^\times \tfrac{4 S^2\pi^2}{3}{\textstyle\sum\limits_{i=1}^{d}}
\tfrac{1}{L_i^2}=S^{2-d}(\tfrac{3}{8})^{d}\tfrac{4\pi^2}{3}L^\times{\textstyle\sum\limits_{i=1}^{d}}
\tfrac{1}{L_i^2}, \notag
\end{align}
we obtain
\begin{align}
 \norm{\theta}{V}^2 &=\nu L^\times{\textstyle\sum\limits_{k=1}^{S_\sigma}}\theta_k^2S^{2-d}
 (\tfrac{3}{8})^{d}\tfrac{4\pi^2}{3}{\textstyle\sum\limits_{i=1}^{d}}
\tfrac{1}{L_i^2}
  +\norm{\theta}{H}^2
=\left(S^2\tfrac{4\nu\pi^2}{3}{\textstyle\sum\limits_{i=1}^{d}}
\tfrac{1}{L_i^2} +1\right)\norm{\theta}{H}^2,\notag
 \end{align}
That is,
\begin{align}
  \norm{\theta}{V}^2
=\left(C_1S^2 +1\right)\norm{\theta}{H}^2,\quad\mbox{with}\quad C_1\coloneqq
\tfrac{4\nu\pi^2}{3}{\textstyle\sum\limits_{i=1}^{d}}
\tfrac{1}{L_i^2}.\notag
 \end{align}

Finally, for the~$\rmD(A)$-norm we find 
\begin{align}
 \norm{\theta}{\rmD(A)}^2&=\norm{-\nu\Delta\theta+\theta}{H}^2
 =\nu^2\norm{\Delta\theta}{H}^2+2\nu\norm{\nabla\Phi_k}{L^2(\Omega)^d}^2+\norm{\theta}{H}^2\notag\\
 &=\nu^2\norm{\Delta\theta}{H}^2+2\norm{\theta}{V}^2-\norm{\theta}{H}^2
 =\nu^2\norm{\Delta\theta}{H}^2+\left(2C_1S^2+1\right)\norm{\theta}{H}^2\notag
 \end{align}
and from
\begin{align}
\norm{\Delta\Phi_k}{H}^2 &={\textstyle\sum\limits_{i=1}^{d}}
\norm{2(\tfrac{S\pi}{L_i})^2\cos(\tfrac{2S\pi x_i}{L_i})}{L^2((0,\frac{L_i}{S}),\bbR)}^2
\bigtimes_{i\ne j=1}^d
\norm{\sin^2(\tfrac{S\pi x_j}{L_j})}{L^2((0,\frac{L_j}{S}),\bbR)}^2\notag\\
&={\textstyle\sum\limits_{i=1}^{d}}
4(\tfrac{S\pi}{L_i})^4\tfrac{L_i}{2S}
\bigtimes_{i\ne j=1}^d\tfrac{3}{8S} L_j={\textstyle\sum\limits_{i=1}^{d}}
(\tfrac{S\pi}{L_i})^4\tfrac{16}{3}
\bigtimes_{j=1}^d\tfrac{3}{8S} L_j
\notag\\
&=(\tfrac{3}{8S})^dL^\times \tfrac{16 S^4\pi^4}{3}{\textstyle\sum\limits_{i=1}^{d}}
\tfrac{1}{L_i^2}=S^{4-d}(\tfrac{3}{8})^{d}\tfrac{16\pi^4}{3}L^\times{\textstyle\sum\limits_{i=1}^{d}}
\tfrac{1}{L_i^2}, \notag
\end{align}
we obtain
\begin{align}
 \norm{\Delta\theta}{H}^2 &={\textstyle\sum\limits_{k=1}^{S_\sigma}}\theta_k^2\norm{\Delta\Phi_k}{H}^2
 =S^{4-d}(\tfrac{3}{8})^{d}\tfrac{16\pi^4}{3}{\textstyle\sum\limits_{i=1}^{d}}
\tfrac{1}{L_i^2}L^\times{\textstyle\sum\limits_{k=1}^{S_\sigma}}\theta_k^2  
 =S^{4}\tfrac{16\pi^4}{3}{\textstyle\sum\limits_{i=1}^{d}}
\tfrac{1}{L_i^2}\norm{\theta}{H}^2,\notag
 \end{align}
hence
\begin{align}
  \norm{\theta}{\rmD(A)}^2
=\left(C_2S^4+ 2C_1S^2 +1\right)\norm{\theta}{H}^2,\quad\mbox{with}\quad
C_2\coloneqq \tfrac{\nu^216\pi^4}{3}{\textstyle\sum\limits_{i=1}^{d}}\notag
\tfrac{1}{L_i^2},
 \end{align}
which finishes the proof.
\qed


%
%

\bigskip\noindent
{\bf Aknowlegments.} The author is supported by ERC advanced grant 668998 (OCLOC) under the EU’s H2020 research program.
The author acknowledges partial support from
the Austrian Science
Fund (FWF): P 33432-NBL.

\end{document}

%% file: Mathcommands.tex
\newcommand{\linspan}{\mathop{\rm span}\nolimits}

\newcommand{\rest}{\left.\kern-2\nulldelimiterspace\right|_}
\newcommand{\norm}[2]{\left|#1\right|_{#2}}
\newcommand{\dnorm}[2]{\left\|#1\right\|_{#2}}

\newcommand{\vol}{\mathop{\rm vol}\nolimits}

\newcommand{\Id}{{\mathbf1}}
\newcommand{\indf}{1}

\newcommand{\ex}{\mathrm{e}}
\newcommand{\p}{\partial}
\newcommand{\ed}{\mathrm d}

\newcommand*{\Bigcdot}{\raisebox{-.25ex}{\scalebox{1.25}{$\cdot$}}}

\newcommand{\clA}{{\mathcal A}}

\newcommand{\clC}{{\mathcal C}}

\newcommand{\clF}{{\mathcal F}}
\newcommand{\clG}{{\mathcal G}}

\newcommand{\clJ}{{\mathcal J}}

\newcommand{\clL}{{\mathcal L}}

\newcommand{\clN}{{\mathcal N}}
\newcommand{\clO}{{\mathcal O}}
\newcommand{\clP}{{\mathcal P}}

\newcommand{\clR}{{\mathcal R}}
\newcommand{\clS}{{\mathcal S}}
\newcommand{\clT}{{\mathcal T}}

\newcommand{\clV}{{\mathcal V}}
\newcommand{\clW}{{\mathcal W}}

\newcommand{\clY}{{\mathcal Y}}
\newcommand{\clZ}{{\mathcal Z}}

\newcommand{\bbN}{{\mathbb N}}

\newcommand{\bbP}{{\mathbb P}}

\newcommand{\bbR}{{\mathbb R}}
\newcommand{\bbS}{{\mathbb S}}

\newcommand{\bfJ}{{\mathbf J}}

\newcommand{\bfZ}{{\mathbf Z}}

\newcommand{\fkA}{{\mathfrak A}}

\newcommand{\fkI}{{\mathfrak I}}

\newcommand{\fkL}{{\mathfrak L}}

\newcommand{\fkN}{{\mathfrak N}}

\newcommand{\fkR}{{\mathfrak R}}
\newcommand{\fkS}{{\mathfrak S}}
\newcommand{\fkT}{{\mathfrak T}}

\newcommand{\rmD}{{\mathrm D}}

\newcommand{\bfi}{{\mathbf i}}
\newcommand{\bfj}{{\mathbf j}}
\newcommand{\bfk}{{\mathbf k}}

\newcommand{\bfn}{{\mathbf n}}

\newcommand{\bfz}{{\mathbf z}}

\newcommand{\rmd}{{\mathrm d}}

\newcommand{\fki}{{\mathfrak i}}

\newcommand{\fkm}{{\mathfrak m}}

\newcommand{\fkr}{{\mathfrak r}}
\newcommand{\fks}{{\mathfrak s}}

\newcommand{\fkw}{{\mathfrak w}}

\newcommand{\ovlineC}[1]{\overline C_{\left[#1\right]}}

\definecolor{DarkBlue}{rgb}{0,0.08,0.45}
\definecolor{DarkRed}{rgb}{.65,0,0}
\definecolor{applegreen}{rgb}{0.55, 0.71, 0.0}

\newcounter{mymac@matlab}
\newcommand{\matlab}{MATLAB%
   \ifnum\value{mymac@matlab}<1%
   \textregistered%
   \setcounter{mymac@matlab}{1}%
   \fi%
  }

\newcommand{\black}{ \color{black} }
